\documentclass[letterpaper,11pt,reqno]{amsart} 
\usepackage[portrait,margin=1in]{geometry}
\usepackage{comment}
\usepackage{systeme}
\usepackage{mathrsfs,xfrac} 
\usepackage{xstring}
\usepackage{enumitem}
\usepackage[colorlinks=true,linkcolor=blue,citecolor=blue,urlcolor=blue]{hyperref} 
\usepackage{cleveref}
\usepackage{amsmath,amssymb,amsthm,amsfonts,amsbsy,latexsym,dsfont,color} 
\usepackage[numeric,initials,nobysame]{amsrefs} 
\usepackage[utf8]{inputenc}
\usepackage[english]{babel}
\usepackage{comment}
\usepackage{mathtools}
\usepackage[textsize=tiny]{todonotes}

\usepackage{pgfplots}
\pgfplotsset{compat=1.18}
\usepackage[foot]{amsaddr}

\usepackage{subcaption} 

\setuptodonotes{inline}

\newcommand\vect[1]{\ifstrequal{#1}{0}{\ensuremath{\mathbf{0}}}{\ensuremath{\boldsymbol{#1}}}} %vector symbol
\newtheorem{thm}{Theorem}[section]
\newtheorem{lem}[thm]{Lemma}
\newtheorem{lemma}[thm]{Lemma}
\newtheorem{cor}[thm]{Corollary}
\newtheorem{prop}[thm]{Proposition}
\newtheorem{defn}[thm]{Definition}

\newtheorem{conj}[thm]{Conjecture}   
\theoremstyle{definition}
\newtheorem{rmk}[thm]{Remark}
\newcommand{\conv}{\circledast}

\newcommand{\eset}{\varnothing}

\newcommand{\TV}{\mathrm{TV}}
\newcommand{\Law}{\mathcal L}
\newcommand{\IK}{\mathfrak{K}}

\DeclareMathOperator\supp{supp}

\newcommand{\SR}{{\sf SR}}
\newcommand{\Red}{{\sf Red}}

\newcommand{\abs}[1]{\left\vert#1\right\vert}

\newcommand{\dist}{\mathrm{dist}}

\newcommand{\vs}{\mathbf{s}}

\makeatletter
\renewcommand\paragraph{\@startsection{paragraph}{4}{\z@}%
  {3.25ex \@plus1ex \@minus.2ex}%
  {-1em}%
  {\normalfont\normalsize\bfseries}}
\makeatother
\begin{document}
\title[Potts]{Information Percolation and cutoff\\for Potts Glauber Dynamics}

\author[A. Sly and S. Yang]
       {Allan Sly$^1$ and Seoyeon Yang$^1$}

\address{$^1$Department of Mathematics,
  Princeton University,
  Princeton, NJ 08540, USA}

\email{allansly@princeton.edu}
\email{syeon.y@princeton.edu}
%%%%%%%%%%%%%%%%%%%%%%%%%%%%%%%%%%%%%%%%%%%%%%%%%%%%
\begin{abstract}
We prove cutoff with an $O(1)$ window for the continuous-time heat-bath Glauber dynamics of the ferromagnetic $q$-state Potts model on the discrete torus $\Lambda_n=(\mathbb Z/n\mathbb Z)^d$
at sufficiently high temperature. More precisely, for every fixed $d\ge2$ and $q\ge3$, there exists $\beta_0=\beta_0(d,q)>0$ such that, for $0<\beta<\beta_0$, the cutoff occurs in an optimal $O(1)$ window centered at
\[
t_\star=t_\star^{(n)}:=\frac{1}{2\mathfrak{r}}\log |\Lambda_n|,
\]
where $\mathfrak{r}\in(0,1)$ is the exact exponential decay rate of the one-site color bias.
Thus, the mixing time is determined up to an
additive \(O(1)\). It is characterized by the time at which the macroscopic color-density bias from the monochromatic initial condition enters the scale of equilibrium fluctuations. Moreover, our proof shows that the monochromatic initial condition uniquely maximizes the color bias.

{This is the first implementation of information percolation to prove cutoff for a non-monotone spin system. On a transitive graph, information percolation naturally gives an upper bound on the mixing time of one half of the grand coupling time which is sharp for the Ising model. However, for many non-monotone systems including the Potts model, this misses additional averaging over the update histories giving an upper bound that is off by a constant multiplicate factor.  We develop a new framework that treats the information percolation histories as random polymers and show how to average over them to give the effective bias from the initial condition.  This accounts for all the additional mixing, recovering a sharp upper bound and cutoff.  We believe this approach will be applicable in a much wider range of settings.}
\end{abstract}
\maketitle

% \setcounter{tocdepth}{1}
%\tableofcontents

%%%%%%%%%%%%%%%%%%%%%%%%%%%%%%%%%%%%%%%%%%%%%%%%%%%%

\section{Introduction}\label{sec:intro} 
%\todo{title;Information Percolation Beyond Monotonicity: \(O(1)\)-Window Cutoff for High-Temperature Potts Glauber Dynamics}

\subsection{High-temperature cutoff for Potts Glauber dynamics}
The cutoff phenomenon---the abrupt transition of a Markov chain from far
from equilibrium to close to equilibrium---is a central theme in the study of stochastic dynamics. High-temperature Ising Glauber dynamics provides a particularly well-studied setting.  On the lattice boxes with periodic boundary conditions, cutoff has been established up to the critical temperature with an optimal $O(1)$ cutoff window and the mixing time location given by the point at which the magnetization started from all plus initial conditions is square root of the volume~\cite{IsingMeanfield,StochasticIP,Universality}. The sharpest result was proved with information percolation which gives a graphical interpretation of the dependence on the initial condition.  However, a key ingredient of using information percolation to prove cutoff is the monotonicity of the Markov chain.

For the Potts model with \(q\ge3\) the Glauber dynamics is not monotone and consequently the standard information percolation gives upper and lower bounds that differ by a constant factor, even at very high temperatures. Earlier work approximating mixing to a product chain proves cutoff for the Potts Glauber dynamics at high enough temperatures~\cite{CutoffSpin} but achieves a suboptimal $O(\log\log n)$ cutoff window. Recent breakthroughs of Pedrotti and Salez~\cite{SalezCurvature,PedrottiSalezLocalProduct} introduced the use of curvature criteria to establish cutoff yield $O(1)$ windows. These powerful general methods, however, establish cutoff without pinpointing the location of the mixing time.

{
The principal goal of this paper is to extend information-percolation beyond the monotone setting. In doing so, we prove that the continuous-time single-site heat-bath Glauber dynamics for the ferromagnetic Potts model on the torus exhibits cutoff with an \(O(1)\) window at sufficiently high temperature, and we identify its center through the exact exponential decay rate of one-site magnetization.}
Theorems~\ref{thm:mono-extremal} and~\ref{thm:worstini} further show
that monochromatic initial states exactly maximize the expected total occupation of a fixed color in finite volume and asymptotically maximize the corresponding one-site bias in infinite volume.

\subsection{Main results}

Fix integers $d\ge2$ and $q\ge3$, and let $\Lambda_n=(\mathbb Z/n\mathbb Z)^d$ with nearest-neighbor
edge set. Let $(\sigma_t)_{t\ge0}$ be the continuous-time heat-bath Glauber dynamics for the
ferromagnetic $q$-state Potts model at inverse temperature $\beta>0$, with each vertex updated at
rate $1$. Write $\mu_{\Lambda_n}$ for the Gibbs measure and
\[
d_{\Lambda_n}(t):=
\max_{\sigma_0\in[q]^{\Lambda_n}}
\left\|
\mathbb P_{\sigma_0}(\sigma_t\in\cdot)-\mu_{\Lambda_n}
\right\|_{\mathrm{TV}}.
\]

We first state the cutoff theorem and then turn to two extremality results for monochromatic initial data.

\begin{thm}[Cutoff]\label{thm:main}
There exists \(\beta_0=\beta_0(d,q)>0\) such that for every \(\beta\in(0,\beta_0)\) there exists a
constant \(\mathfrak{r}=\mathfrak{r}(d,q,\beta)\in(0,1)\) with the following property: the 
continuous-time heat-bath Glauber dynamics for the ferromagnetic \(q\)-state Potts model on
\(\Lambda_n\) exhibits cutoff with an \(O(1)\) window centered at
\begin{equation}\label{eq:tstar-def}
t_\star=t_\star^{(n)}:=\frac{1}{2\mathfrak{r}}\log |\Lambda_n| .
\end{equation}
\end{thm}
The constant \(\mathfrak r\) is the exact exponential decay rate of the one-site magnetization \eqref{eq:magnetization-def}, namely the signed deviation of a fixed color's one-site
marginal from \(1/q\). For the analysis, it is convenient to write 
\(\kappa:=\mathfrak r-1\in(-1,0)\). We prove that $-\kappa=\Theta_{d,q}(\beta)$; see Lemma \ref{lem:kappa bound}.
This also explains the cutoff location: a one-site bias of order
\(e^{-\mathfrak{r}t}\) produces a total color bias of order
\(|\Lambda_n|e^{-\mathfrak{r}t}\), while equilibrium fluctuations are of
order \(|\Lambda_n|^{1/2}\). These scales become comparable when 
$e^{-\mathfrak{r}t}=|\Lambda_n|^{-1/2}$,
which gives \(t=t_\star\).

\smallskip

Our next theorem gives an exact finite-volume extremality statement for the expected total occupation of a fixed color.

\begin{thm}[Monochromatic extremality]\label{thm:mono-extremal}
Assume \(\beta\in(0,\beta_0)\). Fix \(c\in[q]\), and let
\(\mathfrak c\in[q]^{\Lambda_n}\) be the monochromatic configuration
\(\mathfrak c(x)=c\) for all \(x\in\Lambda_n\). Then, for every \(t\ge0\),
\[
\mathbb E_{\mathfrak c}\!\left[\sum_{w\in\Lambda_n}\mathbf 1\{\sigma_t(w)=c\}\right]
=
\sup_{\sigma_0\in[q]^{\Lambda_n}}
\mathbb E_{\sigma_0}\!\left[\sum_{w\in\Lambda_n}\mathbf 1\{\sigma_t(w)=c\}\right].
\]
\end{thm}

We also obtain an asymptotic one-site extremality statement.
The theorem is stated in infinite volume $\mathbb Z^d$, where the proof framework, including the Fourier-analytic part, takes its cleanest form.

\begin{thm}[Asymptotic one-site extremality]\label{thm:worstini}
Assume \(\beta\in(0,\beta_0)\). Let \((\sigma_t)_{t\ge0}\) denote the infinite-volume heat-bath
Glauber dynamics for the ferromagnetic \(q\)-state Potts model on \(\mathbb Z^d\). Fix \(c\in[q]\),
let \(\mathfrak c\in[q]^{\mathbb Z^d}\) be the monochromatic configuration \(\mathfrak c(x)=c\) for
all \(x\in\mathbb Z^d\), and let \(o\in\mathbb Z^d\) be the origin. Then there exist constants
\(C<\infty\), \(\Delta>0\), and \(t_0<\infty\), depending only on \(d,q,\beta\), such that for all
\(t\ge t_0\),
\[
\mathbb P_{\mathfrak c}\bigl(\sigma_t(o)=c\bigr)-\frac1q
\ge
\left(1-C(\log t)^{-\Delta}\right)
\sup_{\sigma_0\in[q]^{\mathbb Z^d}}
\left(
\mathbb P_{\sigma_0}\bigl(\sigma_t(o)=c\bigr)-\frac1q
\right).
\]
\end{thm}

We expect the logarithmic loss \((\log t)^{-\Delta}\) to be inessential and monochromatic initial data to be exactly extremal for the infinite-volume one-site bias; cf.\
Conjecture~\ref{conj:one-site-exact} below.

\begin{conj}\label{conj:one-site-exact}
Fix \(c\in[q]\), and let \(\mathfrak c\in[q]^{\mathbb Z^d}\) be the monochromatic configuration
\(\mathfrak c(x)=c\) for all \(x\in\mathbb Z^d\). Then, for every \(t\ge0\),
\[
\mathbb P_{\mathfrak c}\bigl(\sigma_t(o)=c\bigr)-\frac1q
=
\sup_{\sigma_0\in[q]^{\mathbb Z^d}}
\left(
\mathbb P_{\sigma_0}\bigl(\sigma_t(o)=c\bigr)-\frac1q
\right).
\]
\end{conj}

We expect the analogous finite-volume one-site statement on \(\Lambda_n\) to hold as well.

\subsection{Background and related work}
Cutoff for spin systems was first established in mean-field settings for the Curie--Weiss Ising and Potts models~\cite{IsingMeanfield,PottsMeanfield} where the symmetry of the system reduces it to analysing the much simpler magnetization chain.
{Multi-component Curie--Weiss extensions give a full phase diagram for fixed-group Ising and a rapid-to-exponential transition for homogeneous Potts models \cite{MultiIsing,MultiPotts}.}

Beyond mean field, Lubetzky and the first author gave an Ising cutoff criterion for bounded-degree graphs \cite{CutoffSpin} that reduces the analysis to proving cutoff on a random product chain. It assumes subexponential ball growth and uniformly positive local log-Sobolev constants, while allowing arbitrary boundaries and external fields. For lattice boxes, exponentially decaying sitewise disagreements under a Markovian grand coupling yield an \(O(\log\log n)\) cutoff window; they verified this condition for sufficiently high-temperature non-monotone systems, including Potts dynamics with arbitrary boundary conditions.

Sharper results hold for high-temperature Ising dynamics. The same authors introduced information percolation, proving cutoff with \(O(1)\) windows throughout the full high-temperature regime on fixed-dimensional tori and, at sufficiently high temperature, on arbitrary bounded-degree graphs \cite{StochasticIP,Universality}; see also \cite{Exposition}. The method traces update histories backward in time and exploits the subcritical spread of information to control dependence on the initial condition. Variants have also been developed for the random-cluster model and Swendsen--Wang dynamics; see \cite{RCM,SW}.

More recently, curvature methods have provided a new approach to proving cutoff, giving quantitative window bounds without monotonicity.
For irreducible chains with symmetric support, Salez proved cutoff under nonnegative curvature and a refined product condition via entropic concentration \cite{SalezCurvature}.
Pedrotti and Salez bounded the total-variation window in terms of a local Poincar\'e constant; a volume-uniform positive Bakry--\'Emery curvature lower bound therefore yields a bounded cutoff window \cite{PedrottiSalezLocalProduct}.
These methods are very general and so tend not to yield explicit model-specific information about the mixing time.

A complementary line of work links strong spatial mixing to rapid single-site mixing for monotone systems and sufficiently large-block mixing for general lattice systems \cite{RapidMixLattice}.
For ferromagnetic Potts dynamics on general graphs, rapid-mixing bounds relate interaction strength, maximum degree, and number of colors \cite{GeneralPotts}.
Blanca et al.\ prove an \(O(|\Lambda_n|\log|\Lambda_n|)\) mixing bound
for discrete-time single-site dynamics. Under our continuous-time convention, in which every site updates at rate \(1\), this becomes
\(O(\log|\Lambda_n|)\) \cite{Spectral}.
For fixed \(q\ge3\) on \(\Lambda_n=(\mathbb Z/n\mathbb Z)^d\), these approaches apply in suitable high-temperature regimes, but do not by themselves establish cutoff or identify its location.

%Accordingly, to the best of our knowledge, Theorem~\ref{thm:main} is the first result for continuous-time single-site heat-bath Potts Glauber dynamics on \((\mathbb Z/n\mathbb Z)^d\) that proves an \(O(1)\) cutoff window and identifies its center through the exact decay rate of the infinite-volume one-site magnetization.

\subsection{Proof idea}

In the standard information-percolation analysis for monotone systems (\cite{Exposition,Universality}), backward-update histories are classified into \emph{Blue/Green/Red} clusters. In this decomposition, most of the dependence structure is expressed in the Green clusters, Blue corresponds to IID noise and only Red clusters retain information from the initial condition. In the original information framework, the mixing time was upper bounded at the point at which the terminal state on the blue and red clusters is indistinguishable from random noise which occurs at half the grand coupling time.
{For \(q\ge3\), however, different red histories can translate initial conditions for one state into a different state at the end and so bounding only the total Red mass is no longer sufficient.}

Our main new ingredient is a refinement of Red clusters splitting them into Purple, Yellow, and Strong Red.
Geometrically bad Red clusters are designated \emph{Purple}, which are rare enough to not affect the total variation distance close to the mixing time. For each remaining Red cluster, a one-site mixture decomposition is performed at the merge point, expressing the law as a convex combination of a common uniform law and a residual law. The uniform component contributes no dependence on the initial condition and gives rise to what we call \emph{Yellow} clusters which can be treated effectively as Green. The residual component carries all remaining dependence on initial conditions  and are called \emph{Strong Red} clusters.

We then combine geometric control of backward histories with Fourier estimates for signed convolution powers on \(\mathbb Z^d\), in the spirit of \cite{ConvPower}. This identifies the precise exponential decay rate $\mathfrak{r}$ of the one-site color bias, even though the natural influence kernels are signed.
To obtain the required estimates, we decompose histories into regeneration blocks, called \emph{sausages}.
An overlap argument from \cite{MillerPeres} reduces total-variation mixing to bounding overlaps of Purple and Strong Red clusters; combined with the preceding decay estimate, this yields the cutoff location.
Finally, global and local positivity properties of the same signed kernel yield the finite-volume and one-site extremality results.

The argument thus separates the information-percolation decomposition
from the model-dependent analysis of the residual signed influence.
This suggests a route to sharp mixing results for other non-monotone
finite-state dynamics.

\subsection{Organization of the paper}

Section~\ref{sec:prelim} introduces the Potts model, the heat-bath Glauber dynamics, the graphical construction, and the mixing-time and magnetization notation.
Section~\ref{sec:IP} develops the information-percolation framework, including the refinement of Red clusters into Purple, Yellow, and Strong Red clusters, and reduces the total-variation upper bound to overlap estimates for the Purple and Strong Red sets.
Section~\ref{sec:sausage} develops the quantitative estimates on one-site influence that underpin the cutoff and extremality arguments. To this end, it introduces the sausage renewal decomposition, analyzes the resulting signed influence kernel, identifies its decay rate $\mathfrak{r}$, and proves the magnetization and Strong--Red bounds needed later.
Section~\ref{sec:cutoff} combines these ingredients to establish matching upper and lower bounds around \(t_\star\), thereby proving Theorem~\ref{thm:main}.
Section~\ref{sec:extremality} shows that monochromatic initial states are exactly extremal for the total occupation of a fixed color in finite volume and asymptotically near-extremal for the one-site color bias in infinite volume, proving Theorems~\ref{thm:mono-extremal} and~\ref{thm:worstini}.
Finally, Appendix~\ref{sec:appendix} provides the deferred proofs of the Fourier and local-kernel estimates.

\medskip

\section{Model, Dynamics, and Graphical Construction}\label{sec:prelim}
\subsection{Potts model on the discrete torus}
Let $\Lambda=\Lambda_n=(\mathbb{Z}/n\mathbb{Z})^d$ with nearest-neighbor edge set $E(\Lambda)$. We fix throughout $d\geq 2,\ q\geq 3$. For a configuration $\sigma\in[q]^\Lambda=:\Omega$, the ferromagnetic $q$-state Potts Hamiltonian at inverse temperature $\beta>0$ is
\[
H(\sigma) \;=\; - \sum_{\{x,y\}\in E(\Lambda)} \mathbf{1}\{\sigma_x=\sigma_y\},
\]
and the Gibbs measure is
\[
\mu_\Lambda(\sigma)
\;=\;
Z_\Lambda^{-1}\exp\big(-\beta H(\sigma)\big),
\]
where $Z_\Lambda$ is the normalizing constant.

We work throughout in a sufficiently high-temperature regime,
\begin{equation}\label{eq:high-temp-assumption}
0<\beta<\beta_0(d,q),
\end{equation}
where \(\beta_0(d,q)\) is the constant from Theorem~\ref{thm:main}. {Any additional smallness assumptions on \(\beta\) that arise later are absorbed by further decreasing \(\beta_0(d,q)\).}

\subsection{Glauber dynamics}
We consider the continuous-time single-site Glauber dynamics $(\sigma_t)_{t\ge0}$ on $[q]^\Lambda$, reversible with respect to $\mu_\Lambda$. Each vertex $x\in\Lambda$ updates at rate $1$.
At an update time, the spin at $x$ is resampled from the conditional distribution
\begin{equation}\label{eq:heat-bath-update}
\mathbb P\!\left(\sigma_t(x)=a \,\middle|\, \sigma_{t^-}(y),\, y\neq x\right)
\propto
\exp\!\left(
\beta\sum_{y\sim x}\mathbf 1\{a=\sigma_{t^-}(y)\}
\right),
\qquad a\in[q].
\end{equation}

\subsection{Graphical construction and backward histories}\label{subsec:graphical}

We realize $(\sigma_t)$ via a graphical construction: for each $x\in\Lambda$ attach an independent
rate-$1$ Poisson clock, and equip each ring with an independent auxiliary mark
$U_i\sim\mathrm{Unif}[0,1]$ that determines the resampling according to
\eqref{eq:heat-bath-update}. Write $\mathcal U^\infty=\{(x_i,t_i,U_i)\}_i$ for the full sequence
of update events. When analyzing dynamics up to a terminal time $T>0$, we write
$\mathcal U=\{(x_i,t_i,U_i)\}_i$ for the restriction to $\Lambda\times[0,T]$.

\smallskip
\noindent\textbf{Oblivious updates.}
For every neighborhood configuration, the heat-bath probability assigned to any given color is at
least
\[
\frac{1}{e^{2d\beta}+q-1}.
\]
Accordingly, each update can be coupled as follows: with probability
\[
p_{\mathrm{obl}}:=\frac{q}{e^{2d\beta}+q-1},
\]
the new spin is sampled from \(\mathrm{Unif}([q])\), independently of the neighboring spins; with
the remaining probability, the spin is sampled from a residual law depending on the neighborhood.
We call the first type of update \emph{oblivious}. Set
\[
\alpha(\beta):=1-p_{\mathrm{obl}}
=
\frac{e^{2d\beta}-1}{e^{2d\beta}+q-1}.
\]
Since \(\alpha(\beta)\) is strictly increasing on \((0,\infty)\), we write \(\beta(\alpha)\) for
its inverse:
\[
\beta(\alpha):=\frac{1}{2d}\log\frac{1+\alpha(q-1)}{1-\alpha},
\qquad
\alpha\in(0,1).
\]

For \(0<2d\beta\le 1\), $\frac{2d}{q-1+e}\,\beta
\le
\alpha
\le
\frac{2de}{q}\,\beta,$
so \(\alpha=\Theta(\beta)\) as \(\beta\downarrow0\).

\smallskip
\noindent\textbf{Backward history.}
Fix \(v\in\Lambda\). The backward history \(\mathcal H_v\subset\Lambda\times[0,T]\) is obtained by
starting from \((v,T)\) and exploring backward through the update sequence \(\mathcal U\):
\begin{itemize}
\item If the exploration encounters an oblivious update at \((x,t)\), then the branch at \((x,t)\)
terminates;
\item Otherwise the new spin at \(x\) depends on the neighboring spins just before time \(t\), so
the exploration branches to all \((y,t)\) with \(y\sim x\) and continues backward from each such
neighbor;
\item Between consecutive encountered updates, a branch remains at the same spatial site, producing a
vertical space--time segment.
\end{itemize}
For \(A\subset\Lambda\), define
\[
\mathcal H_A:=\bigcup_{v\in A}\mathcal H_v,
\qquad
\mathcal H_A([s,t]):=
\mathcal H_A\cap \big(\Lambda\times[s,t]\big),\qquad
\mathcal H_A(s):=\mathcal H_A([s,s]).
\]
The connected components of \(\mathcal H_\Lambda\) are the \emph{information-percolation clusters}.

This history exploration is stochastically dominated by a branching process with offspring
distribution \(0\) with probability \(1-\alpha\) and \(2d\) with probability \(\alpha\). After
decreasing \(\beta_0\) if necessary, we assume throughout that
\[
{2d\,\alpha<1,}
\]
so the dominating branching process is subcritical.

\subsection{Mixing time and cutoff}\label{subsec:mixing}
Denote by $(X_t)_{t\geq 0}$ the Markov chain. For $t\ge0$, define worst-case total-variation distance
\[
d_\Lambda(t)
:=
\max_{\sigma_0\in\Omega}
\left\|
\mathbb P_{\sigma_0}(X_t\in\cdot)-\mu_\Lambda
\right\|_{\mathrm{TV}}.
\]
For $\varepsilon\in(0,1)$, the total-variation \emph{mixing time} is
\[
t_{\mathrm{mix}}(\varepsilon):=\inf\{t\ge0:\ d_\Lambda(t)\le\varepsilon\},
\qquad
t_{\mathrm{mix}}:=t_{\mathrm{mix}}(1/4).
\]

We say that the family of Glauber dynamics on \((\Lambda_n)\) exhibits \emph{cutoff} at $t_{\mathrm{mix}}^{(n)}$ with \emph{window} of order $w_n$ if \(w_n=o\bigl(t_{\mathrm{mix}}^{(n)}\bigr)\) and
\[
\lim_{c\to\infty}\ \liminf_{n\to\infty}
d_{\Lambda_n}(t_{\mathrm{mix}}^{(n)}-cw_n)=1,
\quad
\lim_{c\to\infty}\ \limsup_{n\to\infty}
d_{\Lambda_n}(t_{\mathrm{mix}}^{(n)}+c w_n)=0 ,
\]
so the drop of $d_{\Lambda_n}(t)$ from near $1$ to near $0$ occurs within a window $o(t_{\mathrm{mix}})$ around $t_{\mathrm{mix}}$. 
Figure~\ref{fig:cutoff} gives a schematic illustration.

\begin{figure}[tb]
\centering
\includegraphics[width=0.4\linewidth]{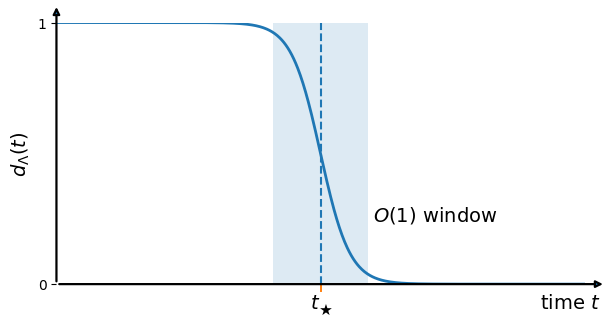}
\caption{Cutoff: total-variation distance drops from near \(1\) to
near \(0\) in a bounded window around the cutoff location \(t_\star\).}
\label{fig:cutoff}
\end{figure}

A key quantity for us is the one-site color bias (which we also refer to as the one-site
magnetization). Fix \(c\in[q]\) and define
\begin{equation}\label{eq:magnetization-def}
\mathfrak m_t^{(n)}(\sigma_0,c)
:=
\mathbb P_{\sigma_0}\bigl(\sigma_t(o)=c\bigr)-\frac1q,
\end{equation}
where \(o\in\Lambda\) is the origin. In the high temperature regime, we will
identify a constant \(\mathfrak{r}=\mathfrak{r}(d,q,\beta)\) governing
the exponential decay of this quantity.

%%%%%%%%%%%%%%%%%%%%%%%%%%%%%%%%%%%%%%%%%%%%%%%%%%%%%%%%%%%%%%%%%%%%%%%%%%%%%%%%%%%%%%%%%%%%%%%%%%%%%%%%%%%%%%%%%%%%%%%%%%%%%%%%%%%%%%%%%%%%%%%%%%%%%%%%%%%%%%%%%%%%%%%%%%%%%%%%%%%%%%%%%%%%%%%%%%%%%%%%%%%%%%%%%%%%%%%%%%%%%%%%%%%%%%%%%%%%%%%%

\medskip
%%%%%%%%%%%%%%%%%%%%%%%%%%%%%%%%%%%%%%%%%%%%%%%%%%%%%%%%%%%%%%%%%%%%
\section{Information Percolation}\label{sec:IP}

Throughout this section we fix \(\Lambda=\Lambda_n\) and a terminal time \(T>0\).
In Section~\ref{subsec:BGR}, we recall the Blue, Green, and Red cluster
classification. Section~\ref{subsec:yellowSR} then introduces the main new
ingredient of this paper: a refinement of the Red clusters into Purple, Yellow,
and Strong Red clusters. Finally, Section~\ref{subsec:overlap-reduction} reduces
the total-variation upper bound to overlap estimates for the Purple and Strong
Red clusters.

\subsection{Blue, Green, and Red clusters}\label{subsec:BGR}

We begin with the standard Blue/Green/Red classification from~\cite{Exposition}. A connected cluster
\(\mathcal C=\mathcal H_A\) is called

\begin{itemize}
\item \emph{Blue} if \(A=\{v\}\) for some \(v\in\Lambda\) and
    \(\mathcal H_A(T-1)=\varnothing\);
\item \emph{Red} if
    \(\mathcal H_A\cap(\Lambda\times\{0\})\neq\varnothing\);
\item \emph{Green} otherwise.
\end{itemize}
{The Blue and Green clusters do not transmit information from the initial condition to time \(T\); only Red clusters can do so. For the Ising model, monotonicity implies that if $A$ is red cluster the final configuration on $X_t(A)$ will be larger started from the all plus initial condition than the all minus one. In the Potts model,
however, the structure of the history of the red cluster leads to more variability in how the initial conditions translate to $X_t(A)$.  In particular, averaging over possible red clusters on $A$ means that we can split off a part of $X_t(A)$ that can be treated as if it is Green and is independent of the initial condition. We therefore refine the Red clusters into three classes:
\[
\text{Red}=\text{Purple}\sqcup\text{Yellow}\sqcup\text{Strong Red}.
\]
Purple clusters are the geometrically bad Red clusters, from which we show are rare enough that they have a vanishing effect on the total variation distance close to the mixing time. For each remaining Red clusters, the
Yellow part corresponds to an averaged component that is independent of the initial condition and the Strong Red part to the
residual component. Consequently, Blue, Green, and Yellow clusters carry no
dependence on the initial condition; all possible dependence is confined to the
Purple and Strong Red clusters.}

We use the following notation for these five cluster types throughout the
paper. Let \(\mathcal H_{\mathrm{Blue}},\mathcal H_G,\mathcal H_P,
\mathcal H_Y,\mathcal H_{\mathrm{SR}}\) denote the unions of histories of
the corresponding types, including their type labels as well as their
geometry, and write
\[
\mathcal H_{G,Y}:=\mathcal H_G\cup\mathcal H_Y,
\qquad
\mathcal H_{G,P}:=\mathcal H_G\cup\mathcal H_P.
\]
Let \(V_{\mathrm{Blue}},V_{\mathrm P},V_{\mathrm{SR}}\) be the unions of
the top sets of Blue, Purple, and Strong Red clusters, respectively, and
let \(V_{G,Y}\) be the union of the top sets of Green and Yellow clusters.
The refinement below determines the Purple, Yellow, and Strong Red types.

\subsection{Red refinement: Purple, Yellow, and Strong Red}
\label{subsec:yellowSR}

We now construct the refinement described above. 
Fix
\(c_{\rm cone}:=2\), and set
\[
        R:=(\log|\Lambda|)^2 .
\]
In the cutoff regime,
\(T=O(\log|\Lambda|)\); hence, for all sufficiently large tori, we assume that $R\ge 4c_{\rm cone}T$.
For a space--time history
\(\mathcal H\subset\Lambda\times[0,T]\), define its spatial span by
\[
        \operatorname{Span}(\mathcal H):=
        \{x\in\Lambda:\exists s\in[0,T]\text{ such that }(x,s)\in \mathcal H\}.
\]

We first designate the geometrically bad Red clusters as \emph{Purple}. A Red
cluster with top set \(A\) is Purple if either
\[
\max_{u\in\operatorname{Span}(\mathcal H_A)}
\operatorname{dist}(u,A)>R/3,
\]
or there exists another Red cluster \(A'\neq A\) such that
\[
\operatorname{dist}(A,A')\le R.
\]
A Red cluster that is not Purple is called
\emph{geometrically regular}.

We split each geometrically regular Red cluster into Yellow and Strong Red. For \(r\in[0,T]\), define the \emph{cone} with tip
\((v,T-r)\) by
\[
    \mathsf C_r(v)
    :=
    \{(y,s)\in\Lambda\times[0,T-r]:
      \dist(y,v)\le c_{\rm cone}(T-r-s)\},
\]
and set
\[
\mathsf C_0(A):=\bigcup_{v\in A}\mathsf C_0(v).
\]

Fix a geometrically regular Red cluster \(A\). Let \(\mathcal H_A^-\) be the
union of the histories generated from \(\Lambda\setminus A\). Whenever
\(|\mathcal H_A(s)|=1\), write \(v_{A,s}\) for its unique vertex. Define

\[
\tau_A:=\sup\left(
\{0\}\cup
\left\{
\begin{aligned}
&s\in(0,T-1]:\quad\\
&|\mathcal H_A(s)|=1,\,
(v_{A,s},s)\in\mathsf C_0(A),\\
&\mathsf C_{T-s}(v_{A,s})
  \cap\mathcal H_{G,P}=\eset,\\
&\mathcal H_A([0,s])
  \subset\mathsf C_{T-s}(v_{A,s})
\end{aligned}
\right\}
\right).
\]

% Figure source: tikz-tau.tex
% Requires \usepackage{tikz}

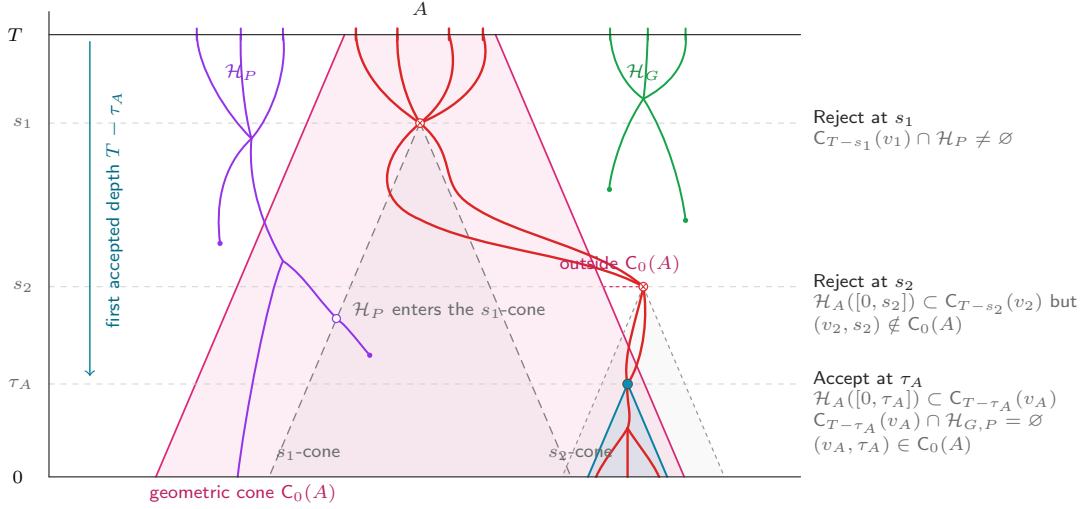
\begin{figure}[t]
\centering

\definecolor{taupink}{HTML}{DB2777}
\definecolor{taured}{HTML}{DC2626}
\definecolor{taugreen}{HTML}{16A34A}
\definecolor{taupurple}{HTML}{9333EA}
\definecolor{taucyan}{HTML}{0891B2}

\begin{tikzpicture}[
    scale=1.3,
    x=.64cm,
    y=.45cm,
    line cap=round,
    line join=round,
    axis/.style={
        black!82,
        line width=.45pt
    },
    guide/.style={
        black!18,
        line width=.35pt,
        dash pattern=on 2pt off 2.5pt
    },
    ambient/.style={
        taupink,
        line width=.65pt
    },
    failedcone/.style={
        black!50,
        line width=.5pt,
        dash pattern=on 3pt off 2.4pt
    },
    secondcone/.style={
        black!42,
        line width=.45pt,
        dash pattern=on 1.2pt off 1.8pt
    },
    protected/.style={
        taucyan!90!black,
        line width=.75pt
    },
    redhist/.style={
        taured,
        line width=.9pt
    },
    greenhist/.style={
        taugreen,
        line width=.75pt
    },
    purplehist/.style={
        taupurple,
        line width=.75pt
    },
    badtip/.style={
        circle,
        draw=taured,
        fill=white,
        minimum size=3.2pt,
        inner sep=0pt,
        line width=.45pt
    },
    goodtip/.style={
        circle,
        draw=black!70,
        fill=taucyan,
        minimum size=3.6pt,
        inner sep=0pt,
        line width=.35pt
    },
    witness/.style={
        circle,
        draw=taupurple,
        fill=white,
        minimum size=3pt,
        inner sep=0pt,
        line width=.45pt
    },
    pinkgap/.style={
        taupink,
        line width=.45pt,
        dash pattern=on 1pt off 1.5pt
    },
    depth/.style={
        taucyan!90!black,
        line width=.55pt
    },
    lab/.style={
        font=\tiny\sffamily,
        text=black!58,
        inner sep=1pt
    },
    pinklab/.style={
        font=\tiny\sffamily,
        text=taupink!85!black,
        inner sep=1pt
    },
    historylabel/.style={
        font=\tiny\sffamily,
        inner sep=1pt
    },
    callhead/.style={
        font=\tiny\sffamily,
        text=black!85,
        inner sep=1pt
    },
    note/.style={
        font=\tiny\sffamily,
        text=black!58,
        inner sep=1pt
    }
]

% ================================================================
% Ambient cone C_0(A)
% ================================================================

\path[
    fill=taupink,
    fill opacity=.08
]
    (-1.2,10) --
    (-4.2,0)  --
    ( 4.2,0)  --
    ( 1.2,10) -- cycle;

\draw[ambient]
    (-1.2,10) -- (-4.2,0)
    ( 1.2,10) -- ( 4.2,0);

% ================================================================
% Candidate cones
% ================================================================

% Cone from s_1=(0,8).
\path[
    fill=black,
    fill opacity=.030
]
    (0,8) -- (-2.40,0) -- (2.40,0) -- cycle;

\draw[failedcone]
    (0,8) -- (-2.40,0)
    (0,8) -- ( 2.40,0);

% Cone from s_2=(3.55,4.30).
\path[
    fill=black,
    fill opacity=.025
]
    (3.55,4.30) -- (2.26,0) -- (4.84,0) -- cycle;

\draw[secondcone]
    (3.55,4.30) -- (2.26,0)
    (3.55,4.30) -- (4.84,0);

% Accepted cone below tau_A=(3.30,2.10).
\path[
    fill=taucyan,
    fill opacity=.10
]
    (3.30,2.10) -- (2.67,0) -- (3.93,0) -- cycle;

\draw[protected]
    (3.30,2.10) -- (2.67,0)
    (3.30,2.10) -- (3.93,0);

% ================================================================
% Horizontal time guides
% ================================================================

\draw[guide] (-5.9,8)   -- (6.05,8);
\draw[guide] (-5.9,4.3) -- (6.05,4.3);
\draw[guide] (-5.9,2.1) -- (6.05,2.1);

% ================================================================
% Green information-percolation history on the right
% ================================================================
\begin{scope}[xshift=-8.5mm]
\foreach \x in {4.35,4.95,5.55}{
    \draw[greenhist] (\x,10.14) -- (\x,9.88);
}

\draw[greenhist]
    (4.35,10)
        .. controls (4.38,9.35) and (4.60,8.82)
        .. (4.88,8.55);

\draw[greenhist]
    (4.95,10)
        .. controls (4.93,9.30) and (4.91,8.82)
        .. (4.88,8.55);

\draw[greenhist]
    (5.55,10)
        .. controls (5.60,9.30) and (5.25,8.82)
        .. (4.88,8.55);

% Branches leave the new merging point (4.88,8.55).
\draw[greenhist]
    (4.88,8.55)
        .. controls (4.45,7.65) and (4.32,6.55)
        .. (4.34,6.50);

\draw[greenhist]
    (4.88,8.55)
        .. controls (5.28,7.65) and (5.50,6.45)
        .. (5.55,5.80);

\fill[taugreen]
    (4.34,6.50) circle[radius=.75pt];

\fill[taugreen]
    (5.55,5.80) circle[radius=.75pt];

\node[historylabel,text=taugreen!85!black]
    at (4.88,9.15)
    {$\mathcal{H}_{G}$};
\end{scope}
% ================================================================
% Purple information-percolation history
% ================================================================

\foreach \x in {-3.55,-2.85,-2.18}{
    \draw[purplehist] (\x,10.14) -- (\x,9.88);
}

\draw[purplehist]
    (-3.55,10)
        .. controls (-3.53,9.25) and (-3.08,8.20)
        .. (-2.68,7.65);

\draw[purplehist]
    (-2.85,10)
        .. controls (-2.87,9.22) and (-2.79,8.20)
        .. (-2.68,7.65);

\draw[purplehist]
    (-2.18,10)
        .. controls (-2.10,9.25) and (-2.34,8.20)
        .. (-2.68,7.65);

% One Purple branch dies early.
\draw[purplehist]
    (-2.68,7.65)
        .. controls (-3.17,7.00) and (-3.25,6.10)
        .. (-3.18,5.28);

\fill[taupurple]
    (-3.18,5.28) circle[radius=.75pt]; 

% Surviving branch continues to the final split.
\draw[purplehist]
    (-2.68,7.65)
        .. controls (-2.75,6.55) and (-2.45,5.55)
        .. (-2.18,4.88);

% Left branch reaches time 0 without entering the s_1-cone.
\draw[purplehist]
    (-2.18,4.88)
        .. controls (-2.42,4.12) and (-2.79,1.35)
        .. (-2.90,0);

% Right branch enters the s_1-cone and dies above tau_A.
\draw[purplehist]
    (-2.18,4.88)
        .. controls (-1.82,4.55) and (-1.55,3.83)
        .. (-1.22,3.45)
        .. controls (-1.05,3.20) and (-.90,2.91)
        .. (-.80,2.75);

\node[witness] at (-1.33,3.58) {}; %white doc in H_P

\fill[taupurple]
    (-.80,2.75) circle[radius=.75pt]; %purple dot in H_P

\node[lab,anchor=south west]
    at (-1.13,3.53)
    {$\mathcal{H}_{P}$ enters the $s_1$-cone};

\node[historylabel,text=taupurple!88!black]
    at (-2.84,9.15)
    {$\mathcal{H}_{P}$};

% ================================================================
% Red information-percolation history of A
% ================================================================

\foreach \x in {-1.02,-.36,.46,1.00}{
    \draw[redhist] (\x,10.14) -- (\x,9.88);
}

\draw[redhist]
    (-1.02,10)
        .. controls (-.98,9.20) and (-.56,8.52)
        .. (0,8);

\draw[redhist]
    (-.36,10)
        .. controls (-.38,9.24) and (-.22,8.50)
        .. (0,8);

\draw[redhist]
    (.46,10)
        .. controls (.63,9.18) and (.45,8.52)
        .. (0,8);

\draw[redhist]
    (1.00,10)
        .. controls (1.18,9.15) and (.74,8.45)
        .. (0,8);

% From s_1 to the outside tip s_2.
\draw[redhist]
    (0,8)
        .. controls (-.68,7.18) and (-.72,6.24)
        .. (0,5.82)
        .. controls (.94,5.12) and (2.66,5.02)
        .. (3.55,4.30);

\draw[redhist]
    (0,8)
        .. controls (.62,7.25) and (.27,6.65)
        .. (.80,6.15)
        .. controls (1.47,5.47) and (2.86,5.18)
        .. (3.55,4.30);

% Below s_2, both red branches stay in the s_2-cone.
\draw[redhist]
    (3.55,4.30)
        .. controls (3.48,3.72) and (3.20,2.72)
        .. (3.30,2.10);

\draw[redhist]
    (3.55,4.30)
        .. controls (3.62,3.72) and (3.58,2.72)
        .. (3.30,2.10);

% Below tau_A, every red branch stays in the accepted cone.
\draw[redhist]
    (3.30,2.10)
        .. controls (3.27,1.72) and (3.38,1.42)
        .. (3.30,1.10);

\draw[redhist]
    (3.30,1.10)
        .. controls (3.15,.75) and (2.90,.34)
        .. (2.80,0);

\draw[redhist]
    (3.30,1.10)
        .. controls (3.30,.75) and (3.30,.34)
        .. (3.30,0);

\draw[redhist]
    (3.30,1.10)
        .. controls (3.45,.75) and (3.70,.34)
        .. (3.80,0);

\node[historylabel,text=black!85]
    at (0,10.58)
    {$A$};

% ================================================================
% Test markers
% ================================================================

\node[badtip] at (0,8) {};
\node[badtip] at (3.55,4.30) {};
\node[goodtip] at (3.30,2.10) {};

% Tiny crosses inside the rejected points.
\draw[taured,line width=.3pt]
    (-.035,7.945) -- (.035,8.055)
    (.035,7.945) -- (-.035,8.055);

\draw[taured,line width=.3pt]
    (3.515,4.245) -- (3.585,4.355)
    (3.585,4.245) -- (3.515,4.355);

% Make s_2 visibly outside the pink region.
\draw[pinkgap]
    (2.91,4.30) -- (3.47,4.30);

\node[pinklab,anchor=south]
    at (3.18,4.53)
    {outside $\mathsf{C}_0(A)$};

% ================================================================
% Time axis
% ================================================================

\draw[axis] (-5.90,0) -- (-5.90,10.52);
\draw[axis] (-5.90,0) -- (6.05,0);
\draw[axis] (-5.90,10) -- (6.05,10);

\node[font=\tiny\sffamily,anchor=east]
    at (-6.15,10)
    {$T$};

\node[font=\tiny\sffamily,anchor=east]
    at (-6.15,0)
    {$0$};

\node[lab,anchor=east]
    at (-6.10,8)
    {$s_1$};

\node[lab,anchor=east]
    at (-6.10,4.3)
    {$s_2$};

\node[lab,anchor=east]
    at (-6.10,2.1)
    {$\tau_A$};

% Cone labels.
\node[pinklab,anchor=north west]
    at (-4.35,-.10)
    {geometric cone $\mathsf{C}_0(A)$};

\node[lab,anchor=south]
    at (-1.78,.32)
    {$s_1$-cone};

\node[lab,anchor=south]
    at (2.55,.32)
    {$s_2$-cone};

% ================================================================
% First accepted depth -- on the left
% ================================================================

\draw[depth,->]
    (-5.25,9.84) -- (-5.25,2.28);

\node[
    lab,
    text=taucyan!75!black,
    rotate=90
]
    at (-4.86,6.05)
    {first accepted depth $T-\tau_A$};

% ================================================================
% Explanations immediately beside the dotted levels
% ================================================================

\node[callhead,anchor=west]
    at (6.20,8.10)
    {Reject at $s_1$};

\node[note,anchor=west]
    at (6.20,7.62)
    {$\mathsf{C}_{T-s_1}(v_1)
      \cap\mathcal{H}_{P}\neq\varnothing$};

\node[callhead,anchor=west]
    at (6.20,4.40)
    {Reject at $s_2$};

\node[note,anchor=west]
    at (6.20,3.92)
    {$\mathcal{H}_{A}([0,s_2])
      \subset\mathsf{C}_{T-s_2}(v_2)$ but};

\node[note,anchor=west]
    at (6.20,3.42){$(v_2,s_2)\notin\mathsf{C}_0(A)$};

\node[callhead,anchor=west]
    at (6.20,2.20)
    {Accept at $\tau_A$};

\node[note,anchor=west]
    at (6.20,1.72)
    {$\mathcal{H}_{A}([0,\tau_A])
      \subset\mathsf{C}_{T-\tau_A}(v_A)$};

\node[note,anchor=west]
    at (6.20,1.22)
    {$\mathsf{C}_{T-\tau_A}(v_A)
      \cap\mathcal{H}_{G,P}=\varnothing$};

\node[note,anchor=west]
    at (6.20,0.72)
    {$(v_A,\tau_A)\in\mathsf{C}_0(A)$};

\end{tikzpicture}

\caption{
The time \(s_1\) is ruled out because a Purple history enters the corresponding
cone. The time \(s_2\) is ruled out because, although the lower Red history is
confined to its cone, the tip \((v_2,s_2)\) lies outside \(\mathsf C_0(A)\).
At the selected time \(\tau_A\), the tip lies in \(\mathsf C_0(A)\), the lower
Red history remains confined, and the accepted cone is disjoint from all
outside histories.
}
\end{figure}

If \(\tau_A=0\), declare \(A\) Strong Red. Henceforth assume that
\(\tau_A>0\), and set
\[
v_A:=v_{A,\tau_A},
\qquad
\mathsf C_A:=\mathsf C_{T-\tau_A}(v_A).
\]
We call \((v_A,\tau_A)\) the \emph{separation tip} and leave the graphical
marks in \(\mathsf C_A\) unrevealed.

The key feature of the separation rule is that, conditional on the accepted
tip, the lower-cone construction is independent of the revealed exterior.
More precisely, once the tip and the revealed exterior are fixed, its
conditional law is that of an independent graphical construction conditioned
only on cone confinement and survival to time \(0\).

We first check that geometric regularity imposes no additional condition on
the lower-cone marks. If \((v,s)\in\mathsf C_0(A)\), then every
\((y,u)\in\mathsf C_{T-s}(v)\) satisfies
\[
\dist(y,A)
\le c_{\rm cone}(s-u)+c_{\rm cone}(T-s)
\le c_{\rm cone}T
\le R/4.
\]
In particular, suppose that \(A'\ne A\) is geometrically regular Red and that
its history meets \(\mathsf C_A\). At an intersection point \((y,u)\),
geometric regularity of \(A'\) gives \(\dist(y,A')\le R/3\), and hence
\[
\dist(A,A')\le R/4+R/3<R,
\]
contrary to geometric regularity of \(A\). A Blue history is empty at and
below time \(T-1\), whereas \(\mathsf C_A\) lies below
\(\tau_A\le T-1\). The defining disjointness from
\(\mathcal H_{G,P}\) therefore yields
$\mathsf C_A\cap\mathcal H_A^-=\eset$.
The same \(R/4\) bound shows that the lower-cone marks cannot trigger the
spatial-span criterion for \(A\) to be Purple. In view of the displayed
disjointness, the nearby-Red-cluster criterion is determined entirely by the
histories outside \(\mathsf C_A\). Thus conditioning on \(A\) being
geometrically regular does not further bias the lower-cone marks.

For \(v\in\Lambda\) and \(0<t\le T\), let \(\mathcal H^{v,t}\) be the backward
history started from \(\{v\}\) at time \(t\) using an independent copy of the
graphical construction on \([0,t]\), and define
\[
\mathsf{Bot}_t(v)
:=
\left\{
\mathcal H^{v,t}\subseteq\mathsf C_{T-t}(v)
\ \text{and}\
\mathcal H^{v,t}(0)\ne\varnothing
\right\},
\qquad
M_t:=\mathbb P\bigl(\mathsf{Bot}_t(v)\bigr).
\]
We set $\mathsf{Bot}_0(v):=\Omega$ and $M_0:=1$.
By translation invariance, \(M_t\) does not depend on \(v\).
The preceding properties identify the conditional law below the separation
tip: the lower-cone marks have the law of an independent
graphical construction conditioned only on
\(\mathsf{Bot}_{\tau_A}(v_A)\). In particular, this conditional law does not
otherwise depend on the revealed exterior.

For an initial configuration \(\sigma_0\in[q]^\Lambda\), use that same
independent graphical construction to generate \(\sigma_t(v)\), and define
\[
\mathsf S(t,v,\sigma_0)(c)
:=
\mathbb P_{\sigma_0}\bigl(
\sigma_t(v)=c
\,\big|\,
\mathsf{Bot}_t(v)
\bigr),
\qquad c\in[q],
\]
and set
\[
p_t(v)
:=
q\inf_{\sigma_0\in[q]^\Lambda}
\min_{c\in[q]}\mathsf S(t,v,\sigma_0)(c),
\qquad p_0(v):=0.
\]
There is a probability measure \(\mathsf R(t,v,\sigma_0)\) on \([q]\) such that
\[
\mathsf S(t,v,\sigma_0)
=
p_t(v)\operatorname{Unif}([q])
+
\bigl(1-p_t(v)\bigr)\mathsf R(t,v,\sigma_0).
\]

For every \((v,t)\in\Lambda\times(0,T]\), let
\(J_{(v,t)}\sim\operatorname{Unif}[0,1]\), with these variables mutually
independent and independent of the graphical construction. Set
\(J_A:=J_{(v_A,\tau_A)}\), and $J_{(v,0)}:=1$. If
\(J_A\le p_{\tau_A}(v_A)\), declare \(A\) Yellow and sample the
separation-tip color from \(\operatorname{Unif}([q])\). Otherwise, declare
\(A\) Strong Red and sample the separation-tip color from
\(\mathsf R(\tau_A,v_A,\sigma_0)\). Together with the declaration for
\(\tau_A=0\), this partitions every geometrically regular Red cluster into
Yellow or Strong Red.

Subsection~\ref{subsec:cone-confined} bounds the unconditioned residual mass
\[
\mathbb P_{\sigma_0}\left(
\mathsf{Bot}_t(v),\,
J_{(v,t)}>p_t(v)
\right)
=M_t\bigl(1-p_t(v)\bigr),
\]
and Section~\ref{sec:cutoff} controls the backward delay \(T-\tau_A\).

% Figure source: tikz-IP.tex

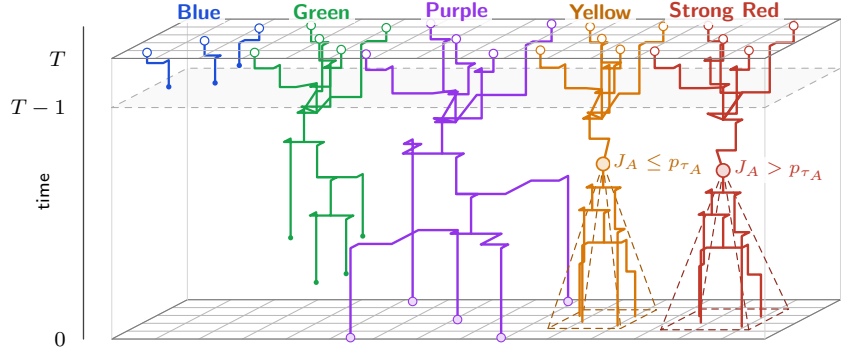
\begin{figure}[t]
\centering
\begingroup

\definecolor{ipblue}{HTML}{1D4ED8}
\definecolor{ipgreen}{HTML}{16A34A}
\definecolor{ippurple}{HTML}{9333EA}
\definecolor{ipyellow}{HTML}{D97706}
\definecolor{ipred}{HTML}{C0392B}

% Draw the common spatial lattice at height #1.
\newcommand{\IPplane}[1]{%
    \foreach \i in {0,...,15}{
        \draw[ipgrid] (\i,0,#1)--(\i,5,#1);
    }
    \foreach \j in {0,...,5}{
        \draw[ipgrid] (0,\j,#1)--(15,\j,#1);
    }
    \draw[ipframe]
        (0,0,#1)--(15,0,#1)--(15,5,#1)--(0,5,#1)--cycle;
}

\begin{tikzpicture}[
    scale=.8,
    x={(0.72cm,0cm)},
    y={(0.25cm,0.13cm)},
    z={(0cm,0.58cm)},
    line cap=round,
    line join=round,
    ipgrid/.style={
        black!28,
        line width=.25pt
    },
    ipframe/.style={
        black!52,
        line width=.4pt
    },
    ipvertical/.style={
        black!16,
        line width=.3pt
    },
    ipguide/.style={
        black!32,
        line width=.4pt,
        dash pattern=on 2.5pt off 2.5pt
    },
    bluehist/.style={
        ipblue,
        line width=.85pt
    },
    greenhist/.style={
        ipgreen,
        line width=.85pt
    },
    purplehist/.style={
        ippurple,
        line width=.9pt
    },
    yellowhist/.style={
        ipyellow,
        line width=.9pt
    },
    redhist/.style={
        ipred,
        line width=.9pt
    },
    coneedge/.style={
        line width=.45pt,
        dash pattern=on 2.5pt off 2pt
    },
    topnode/.style={
        circle,
        fill=white,
        minimum size=3.2pt,
        inner sep=0pt,
        line width=.45pt
    },
    tipnode/.style={
        circle,
        minimum size=5pt,
        inner sep=0pt,
        line width=.65pt
    },
    classlabel/.style={
        font=\scriptsize\sffamily\bfseries,
        fill=white,
        fill opacity=.85,
        text opacity=1,
        inner sep=1.5pt
    },
    note/.style={
        font=\tiny\sffamily,
        fill=white,
        fill opacity=.82,
        text opacity=1,
        inner sep=1pt
    }
]

% ================================================================
% One common space--time slab
% ================================================================

\IPplane{0}

% Faint vertical sides emphasize that all histories occupy one space.
\draw[ipvertical] (0,0,0)--(0,0,8);
\draw[ipvertical] (15,0,0)--(15,0,8);
\draw[ipvertical] (15,5,0)--(15,5,8);
\draw[ipvertical] (0,5,0)--(0,5,8);

% The level T-1.
\path[fill=black,fill opacity=.025]
    (0,0,6.6)--(15,0,6.6)--(15,5,6.6)--(0,5,6.6)--cycle;

\draw[ipguide]
    (0,0,6.6)--(15,0,6.6)--(15,5,6.6)
    --(0,5,6.6)--cycle;

% Time axis.
\draw[black!75,line width=.5pt]
    (-.65,0,0)--(-.65,0,8.9);

\node[font=\scriptsize,anchor=east] at (-.82,0,8) {$T$};
\node[font=\scriptsize,anchor=east] at (-.82,0,6.6) {$T-1$};
\node[font=\scriptsize,anchor=east] at (-.82,0,0) {$0$};

\node[
    font=\tiny\sffamily,
    rotate=90,
    anchor=south
] at (-1.20,0,4.2) {time};

% ================================================================
% BLUE CLUSTERS
% ================================================================

\draw[bluehist]
    (.55,.75,8)
    --(.55,.75,7.70)
    --(.92,.75,7.70)
    --(.92,1.12,7.70)
    --(.92,1.12,6.94);

\draw[bluehist]
    (1.35,2.25,8)
    --(1.35,2.25,7.58)
    --(1.72,2.25,7.58)
    --(1.72,1.88,7.58)
    --(1.72,1.88,6.88);

\draw[bluehist]
    (2.05,3.85,8)
    --(2.05,3.85,7.66)
    --(1.72,3.85,7.66)
    --(1.72,3.48,7.66)
    --(1.72,3.48,7.02);

\fill[ipblue] (.92,1.12,6.94) circle (1.25pt);
\fill[ipblue] (1.72,1.88,6.88) circle (1.25pt);
\fill[ipblue] (1.72,3.48,7.02) circle (1.25pt);

% ================================================================
% GREEN CLUSTER
% ================================================================

% Main Green backbone.
\draw[greenhist]
    (3.90,2.50,8)
    --(3.90,2.50,7.38)
    --(4.28,2.50,7.38)
    --(4.28,2.12,7.38)
    --(4.28,2.12,6.72)
    --(3.82,2.12,6.72)
    --(3.82,2.55,6.72)
    --(3.82,2.55,5.88)
    --(3.38,2.55,5.88)
    --(3.38,2.98,5.88)
    --(3.38,2.98,4.95)
    --(3.84,2.98,4.95)
    --(3.84,2.52,4.95)
    --(3.84,2.52,4.05)
    --(4.30,2.52,4.05)
    --(4.30,2.10,4.05)
    --(4.30,2.10,3.05)
    --(3.82,2.10,3.05)
    --(3.82,2.52,3.05)
    --(3.82,2.52,1.05);

% Upper branches.
\draw[greenhist]
    (3.02,.75,8)
    --(3.02,.75,7.55)
    --(3.45,.75,7.55)
    --(3.45,1.18,7.55)
    --(3.45,1.18,6.92)
    --(3.82,1.18,6.92)
    --(3.82,2.55,6.72);

\draw[greenhist]
    (4.88,1.12,8)
    --(4.88,1.12,7.48)
    --(4.48,1.12,7.48)
    --(4.48,1.55,7.48)
    --(4.48,1.55,6.28)
    --(3.82,1.55,6.28)
    --(3.82,2.55,5.88);

\draw[greenhist]
    (3.12,4.20,8)
    --(3.12,4.20,7.52)
    --(3.52,4.20,7.52)
    --(3.52,3.78,7.52)
    --(3.52,3.78,6.22)
    --(3.38,3.78,6.22)
    --(3.38,2.98,5.88);

\draw[greenhist]
    (4.95,3.90,8)
    --(4.95,3.90,7.62)
    --(4.52,3.90,7.62)
    --(4.52,3.48,7.62)
    --(4.52,3.48,6.12)
    --(3.82,3.48,6.12)
    --(3.82,2.55,5.88);

% Lower branches, all terminating above time 0.
\draw[greenhist]
    (3.38,2.98,4.95)
    --(2.92,2.98,4.95)
    --(2.92,3.42,4.95)
    --(2.92,3.42,2.12);

\draw[greenhist]
    (3.84,2.52,4.05)
    --(4.75,2.52,4.05)
    --(4.75,2.92,4.05)
    --(4.75,2.92,2.28);

\draw[greenhist]
    (4.30,2.10,3.05)
    --(4.82,2.10,3.05)
    --(4.82,1.62,3.05)
    --(4.82,1.62,1.50);

\foreach \p in {
    (3.82,2.52,1.05),
    (2.92,3.42,2.12),
    (4.75,2.92,2.28),
    (4.82,1.62,1.50)
}{
    \fill[ipgreen] \p circle (1.2pt);
}

% ================================================================
% PURPLE CLUSTER
% ================================================================

% Main Purple backbone.
\draw[purplehist]
    (7.05,2.48,8)
    --(7.05,2.48,7.30)
    --(7.48,2.48,7.30)
    --(7.48,2.05,7.30)
    --(7.48,2.05,6.52)
    --(7.02,2.05,6.52)
    --(7.02,2.52,6.52)
    --(7.02,2.52,5.62)
    --(6.55,2.52,5.62)
    --(6.55,2.98,5.62)
    --(6.55,2.98,4.62)
    --(7.05,2.98,4.62)
    --(7.05,2.48,4.62)
    --(7.05,2.48,3.58)
    --(7.55,2.48,3.58)
    --(7.55,2.02,3.58)
    --(7.55,2.02,2.55)
    --(7.08,2.02,2.55)
    --(7.08,2.48,2.55)
    --(7.08,2.48,0);

% Upper Purple branches.
\draw[purplehist]
    (5.65,.55,8)
    --(5.65,.55,7.55)
    --(6.10,.55,7.55)
    --(6.10,1.00,7.55)
    --(6.10,1.00,6.78)
    --(7.02,1.00,6.78)
    --(7.02,2.52,6.52);

\draw[purplehist]
    (8.55,.62,8)
    --(8.55,.62,7.48)
    --(8.12,.62,7.48)
    --(8.12,1.08,7.48)
    --(8.12,1.08,6.05)
    --(7.02,1.08,6.05)
    --(7.02,2.52,5.62);

\draw[purplehist]
    (5.78,4.45,8)
    --(5.78,4.45,7.55)
    --(6.22,4.45,7.55)
    --(6.22,4.02,7.55)
    --(6.22,4.02,6.15)
    --(6.55,4.02,6.15)
    --(6.55,2.98,5.62);

\draw[purplehist]
    (8.58,4.38,8)
    --(8.58,4.38,7.50)
    --(8.15,4.38,7.50)
    --(8.15,3.92,7.50)
    --(8.15,3.92,6.08)
    --(7.02,3.92,6.08)
    --(7.02,2.52,5.62);

% Widely separated lower branches.
\draw[purplehist]
    (6.55,2.98,4.62)
    --(5.65,2.98,4.62)
    --(5.65,4.20,4.62)
    --(5.25,4.20,4.62)
    --(5.25,4.78,4.62)
    --(5.25,4.78,0);

\draw[purplehist]
    (7.05,2.48,3.58)
    --(8.18,2.48,3.58)
    --(8.18,4.18,3.58)
    --(8.82,4.18,3.58)
    --(8.82,4.78,3.58)
    --(8.82,4.78,0);

\draw[purplehist]
    (7.55,2.02,2.55)
    --(8.25,2.02,2.55)
    --(8.25,.65,2.55)
    --(8.88,.65,2.55)
    --(8.88,.18,2.55)
    --(8.88,.18,0);

\draw[purplehist]
    (7.08,2.48,2.55)
    --(6.18,2.48,2.55)
    --(6.18,.72,2.55)
    --(5.42,.72,2.55)
    --(5.42,.18,2.55)
    --(5.42,.18,0);

\foreach \p in {
    (7.08,2.48,0),
    (5.25,4.78,0),
    (8.82,4.78,0),
    (8.88,.18,0),
    (5.42,.18,0)
}{
    \node[
        circle,
        draw=ippurple,
        fill=ippurple!18,
        minimum size=3.2pt,
        inner sep=0pt
    ] at \p {};
}

% ================================================================
% YELLOW CLUSTER
% ================================================================

\coordinate (Ytip) at (10.40,2.55,4.42);

% Accepted cone.
\draw[coneedge,ipyellow!75!black] (Ytip)--(9.55,1.35,0);
\draw[coneedge,ipyellow!75!black] (Ytip)--(11.22,1.35,0);
\draw[coneedge,ipyellow!75!black] (Ytip)--(11.22,3.72,0);
\draw[coneedge,ipyellow!75!black] (Ytip)--(9.55,3.72,0);

\draw[coneedge,ipyellow!75!black]
    (9.55,1.35,0)--(11.22,1.35,0)
    --(11.22,3.72,0)--(9.55,3.72,0)--cycle;

% History above the Yellow separation tip.
\draw[yellowhist]
    (10.38,2.52,8)
    --(10.38,2.52,7.32)
    --(10.82,2.52,7.32)
    --(10.82,2.08,7.32)
    --(10.82,2.08,6.55)
    --(10.38,2.08,6.55)
    --(10.38,2.52,6.55)
    --(10.38,2.52,5.72)
    --(9.95,2.52,5.72)
    --(9.95,2.95,5.72)
    --(9.95,2.95,5.05)
    --(10.40,2.95,5.05)
    --(Ytip);

\draw[yellowhist]
    (9.48,1.05,8)
    --(9.48,1.05,7.50)
    --(9.90,1.05,7.50)
    --(9.90,1.48,7.50)
    --(9.90,1.48,6.82)
    --(10.38,1.48,6.82)
    --(10.38,2.52,6.55);

\draw[yellowhist]
    (11.28,1.15,8)
    --(11.28,1.15,7.46)
    --(10.88,1.15,7.46)
    --(10.88,1.58,7.46)
    --(10.88,1.58,6.10)
    --(10.38,1.58,6.10)
    --(10.38,2.52,5.72);

\draw[yellowhist]
    (9.58,4.05,8)
    --(9.58,4.05,7.55)
    --(9.98,4.05,7.55)
    --(9.98,3.65,7.55)
    --(9.98,3.65,6.18)
    --(9.95,3.65,6.18)
    --(9.95,2.95,5.72);

\draw[yellowhist]
    (11.25,4.08,8)
    --(11.25,4.08,7.55)
    --(10.85,4.08,7.55)
    --(10.85,3.65,7.55)
    --(10.85,3.65,6.18)
    --(10.38,3.65,6.18)
    --(10.38,2.52,5.72);

% Confined Yellow history below the tip.
\draw[yellowhist]
    (Ytip)
    --(10.40,2.55,3.75)
    --(10.02,2.55,3.75)
    --(10.02,2.92,3.75)
    --(10.02,2.92,3.02)
    --(10.42,2.92,3.02)
    --(10.42,2.55,3.02)
    --(10.42,2.55,2.18)
    --(10.05,2.55,2.18)
    --(10.05,2.18,2.18)
    --(10.05,2.18,0);

\draw[yellowhist]
    (10.02,2.92,3.02)
    --(9.72,2.92,3.02)
    --(9.72,3.25,3.02)
    --(9.72,3.25,1.48)
    --(9.72,2.95,1.48)
    --(9.72,2.95,0);

\draw[yellowhist]
    (10.42,2.55,2.18)
    --(10.82,2.55,2.18)
    --(10.82,2.88,2.18)
    --(10.82,2.88,1.28)
    --(11.02,2.88,1.28)
    --(11.02,2.88,0);

\draw[yellowhist]
    (10.40,2.55,3.75)
    --(10.80,2.55,3.75)
    --(10.80,2.18,3.75)
    --(10.80,2.18,2.65)
    --(11.02,2.18,2.65)
    --(11.02,1.72,2.65)
    --(11.02,1.72,0);

\node[
    tipnode,
    draw=ipyellow,
    fill=ipyellow!18
] at (Ytip) {};

% ================================================================
% STRONG RED CLUSTER
% ================================================================

\coordinate (Rtip) at (13.18,2.48,4.25);

% Accepted cone.
\draw[coneedge,ipred!75!black] (Rtip)--(12.20,1.18,0);
\draw[coneedge,ipred!75!black] (Rtip)--(14.22,1.18,0);
\draw[coneedge,ipred!75!black] (Rtip)--(14.22,3.82,0);
\draw[coneedge,ipred!75!black] (Rtip)--(12.20,3.82,0);

\draw[coneedge,ipred!75!black]
    (12.20,1.18,0)--(14.22,1.18,0)
    --(14.22,3.82,0)--(12.20,3.82,0)--cycle;

% History above the Strong Red separation tip.
\draw[redhist]
    (13.18,2.48,8)
    --(13.18,2.48,7.38)
    --(12.75,2.48,7.38)
    --(12.75,2.88,7.38)
    --(12.75,2.88,6.62)
    --(13.18,2.88,6.62)
    --(13.18,2.48,6.62)
    --(13.18,2.48,5.70)
    --(13.62,2.48,5.70)
    --(13.62,2.05,5.70)
    --(13.62,2.05,4.92)
    --(13.18,2.05,4.92)
    --(Rtip);

\draw[redhist]
    (12.15,.92,8)
    --(12.15,.92,7.52)
    --(12.58,.92,7.52)
    --(12.58,1.35,7.52)
    --(12.58,1.35,6.85)
    --(13.18,1.35,6.85)
    --(13.18,2.48,6.62);

\draw[redhist]
    (14.28,1.02,8)
    --(14.28,1.02,7.45)
    --(13.85,1.02,7.45)
    --(13.85,1.45,7.45)
    --(13.85,1.45,6.10)
    --(13.18,1.45,6.10)
    --(13.18,2.48,5.70);

\draw[redhist]
    (12.25,4.18,8)
    --(12.25,4.18,7.55)
    --(12.68,4.18,7.55)
    --(12.68,3.75,7.55)
    --(12.68,3.75,6.18)
    --(13.18,3.75,6.18)
    --(13.18,2.48,5.70);

\draw[redhist]
    (14.22,4.15,8)
    --(14.22,4.15,7.52)
    --(13.82,4.15,7.52)
    --(13.82,3.72,7.52)
    --(13.82,3.72,6.20)
    --(13.18,3.72,6.20)
    --(13.18,2.48,5.70);

% Confined Strong Red history below the tip.
\draw[redhist]
    (Rtip)
    --(13.18,2.48,3.62)
    --(12.78,2.48,3.62)
    --(12.78,2.88,3.62)
    --(12.78,2.88,2.90)
    --(13.18,2.88,2.90)
    --(13.18,2.48,2.90)
    --(13.18,2.48,2.05)
    --(12.80,2.48,2.05)
    --(12.80,2.10,2.05)
    --(12.80,2.10,0);

\draw[redhist]
    (12.78,2.88,2.90)
    --(12.48,2.88,2.90)
    --(12.48,3.22,2.90)
    --(12.48,3.22,1.42)
    --(12.48,2.92,1.42)
    --(12.48,2.92,0);

\draw[redhist]
    (13.18,2.48,2.05)
    --(13.62,2.48,2.05)
    --(13.62,2.85,2.05)
    --(13.62,2.85,1.20)
    --(13.92,2.85,1.20)
    --(13.92,2.85,0);

\draw[redhist]
    (13.18,2.48,3.62)
    --(13.58,2.48,3.62)
    --(13.58,2.08,3.62)
    --(13.58,2.08,2.58)
    --(13.92,2.08,2.58)
    --(13.92,1.58,2.58)
    --(13.92,1.58,0);

\node[
    tipnode,
    draw=ipred,
    fill=ipred!18
] at (Rtip) {};

% ================================================================
% Common top lattice and top-set markers
% ================================================================

\IPplane{8}

% Blue top sets.
\foreach \p in {
    (.55,.75,8),
    (1.35,2.25,8),
    (2.05,3.85,8)
}{
    \node[topnode,draw=ipblue] at \p {};
}

% Green top set.
\foreach \p in {
    (3.90,2.50,8),
    (3.02,.75,8),
    (4.88,1.12,8),
    (3.12,4.20,8),
    (4.95,3.90,8)
}{
    \node[topnode,draw=ipgreen] at \p {};
}

% Purple top set.
\foreach \p in {
    (7.05,2.48,8),
    (5.65,.55,8),
    (8.55,.62,8),
    (5.78,4.45,8),
    (8.58,4.38,8)
}{
    \node[topnode,draw=ippurple] at \p {};
}

% Yellow top set.
\foreach \p in {
    (10.38,2.52,8),
    (9.48,1.05,8),
    (11.28,1.15,8),
    (9.58,4.05,8),
    (11.25,4.08,8)
}{
    \node[topnode,draw=ipyellow] at \p {};
}

% Strong Red top set.
\foreach \p in {
    (13.18,2.48,8),
    (12.15,.92,8),
    (14.28,1.02,8),
    (12.25,4.18,8),
    (14.22,4.15,8)
}{
    \node[topnode,draw=ipred] at \p {};
}

% ================================================================
% Direct labels
% ================================================================

\node[classlabel,text=ipblue]
    at (1.15,2.45,8.78) {Blue};

\node[classlabel,text=ipgreen]
    at (3.95,2.50,8.78) {Green};

\node[classlabel,text=ippurple]
    at (7.05,2.50,8.78) {Purple};

\node[classlabel,text=ipyellow!88!black]
    at (10.35,2.50,8.78) {Yellow};

\node[classlabel,text=ipred]
    at (13.20,2.50,8.78) {Strong Red};

\node[
    note,
    text=ipyellow!88!black,
    anchor=west
] at (10.55,2.55,4.42)
    {$J_A\le p_{\tau_A}$};

\node[
    note,
    text=ipred,
    anchor=west
] at (13.35,2.48,4.25)
    {$J_A>p_{\tau_A}$};

\end{tikzpicture}

\caption{
Blue histories die before time \(T-1\), while the Green history does not reach
time \(0\). The Purple history reaches time \(0\) but is geometrically
irregular. The Yellow and Strong Red histories are isolated and geometrically
regular; their accepted cones are shown by dashed lines.
}
\label{fig:IP-common-space}

\endgroup
\end{figure}

\subsection{Reduction to Purple/Strong-Red overlap}\label{subsec:overlap-reduction}

We use the following \(q\)-ary version of the overlap argument in Proposition~3.2 of \cite{MillerPeres}.

\begin{lemma}\label{lem:MP-overlap}

Let \(W\) be a finite set, and let \(\nu\) be the uniform product measure on
\([q]^W\). Suppose that a probability measure \(\mu\) on \([q]^W\) is
generated as follows. First sample a random set \(S\subseteq W\), and then
sample the spins on \(S\) from an arbitrary law that may depend on \(S\).
Conditional on \(S\) and on these spins, sample the spins on \(W\setminus S\)
independently and uniformly from \([q]\). Then

\[
\|\mu-\nu\|_{L^2(\nu)}^2
\le
\mathbb E\big[q^{|S\cap S'|}\big]-1,
\]
where \(S'\) is an independent copy of \(S\).
\end{lemma}

By Cauchy--Schwarz, the \(L^2\)-bound in Lemma~\ref{lem:MP-overlap}
implies the total-variation bound

\[
\|\mu-\nu\|_{\TV}
=
\frac12\int \left|\frac{d\mu}{d\nu}-1\right|\,d\nu
\le
\frac12
\left\|
\frac{d\mu}{d\nu}-1
\right\|_{L^2(\nu)}=
\frac12\|\mu-\nu\|_{L^2(\nu)}.
\]

Fix an initial condition \(x_0\in[q]^\Lambda\). Let \(\pi\) be the stationary
Potts measure, and take \(X_0^\pi\sim\pi\) independently of the graphical
construction and all auxiliary randomness. Using the same graphical
construction, run one chain \(X_t^{x_0}\) from \(x_0\) and another
\(X_t^\pi\) from \(X_0^\pi\). By stationarity, \(X_t^\pi\sim\pi\). For
\(U\subseteq\Lambda\), let \(\nu_U\) denote the product-uniform measure on
\([q]^U\).

\begin{prop}[Purple/Strong-Red overlap reduction]\label{prop:IP-overlap}

Set \(V_{\mathrm{P}}^{(1)}:=V_{\mathrm{P}}\), and, conditional on
\(\mathcal H_{G,Y}\vee\mathcal H_{\mathrm{SR}}\), let
\(V_{\mathrm{P}}^{(2)}\) be an independent copy of
\(V_{\mathrm{P}}^{(1)}\). Similarly, set
\(V_{\mathrm{SR}}^{(1)}:=V_{\mathrm{SR}}\), and, conditional on
\(\mathcal H_{G,Y}\vee\mathcal H_{\mathrm{P}}\), let
\(V_{\mathrm{SR}}^{(2)}\) be an independent copy of
\(V_{\mathrm{SR}}^{(1)}\). Then, for every \(x_0\in[q]^\Lambda\),

\[
\left\|
\Law(X_t^{x_0})-\pi
\right\|_{\mathrm{TV}}
\le
\left(
\mathbb E\left[
q^{|V_{\mathrm{P}}^{(1)}\cap V_{\mathrm{P}}^{(2)}|}-1
\right]
\right)^{1/2}
+
\left(
\mathbb E\left[
q^{|V_{\mathrm{SR}}^{(1)}\cap V_{\mathrm{SR}}^{(2)}|}-1
\right]
\right)^{1/2}.
\]
Consequently,
\[
d_\Lambda(t)
\le
\left(
\mathbb E\left[
q^{|V_{\mathrm{P}}^{(1)}\cap V_{\mathrm{P}}^{(2)}|}-1
\right]
\right)^{1/2}
+
\left(
\mathbb E\left[
q^{|V_{\mathrm{SR}}^{(1)}\cap V_{\mathrm{SR}}^{(2)}|}-1
\right]
\right)^{1/2}.
\]
\end{prop}

\begin{proof}
Conditioning on \(\mathcal H_{G,Y}\) and using convexity of total variation gives
\[
\bigl\|\Law(X_t^{x_0})-\pi\bigr\|_{\TV}
=
\bigl\|\Law(X_t^{x_0})-\Law(X_t^\pi)\bigr\|_{\TV}
\le
\mathbb E\left[
\left\|
\Law(X_t^{x_0}\mid\mathcal H_{G,Y})
-
\Law(X_t^\pi\mid\mathcal H_{G,Y})
\right\|_{\TV}
\right].
\]
By construction, the spins on $V_{G,Y}$ at time $t$ are \(\mathcal H_{G,Y}\)-measurable, and their conditional law does not depend on the initial condition. These coordinates therefore form a common conditional
factor, so
\[
\left\|
\Law(X_t^{x_0}\mid\mathcal H_{G,Y})
-
\Law(X_t^\pi\mid\mathcal H_{G,Y})
\right\|_{\TV}
=
\left\|
\Law(X_t^{x_0}(W)\mid\mathcal H_{G,Y})
-
\Law(X_t^\pi(W)\mid\mathcal H_{G,Y})
\right\|_{\TV}.
\]

For each initial condition \(\xi\in[q]^\Lambda\), construct an auxiliary
configuration \(Y_t^\xi\) on \(W\) from \(X_t^\xi(W)\) by replacing the spins
on \(V_{\mathrm{P}}\) with fresh independent uniform spins and leaving the
spins on \(W\setminus V_{\mathrm{P}}\) unchanged. The fresh spins are also
independent of all other randomness. When \(\xi=X_0^\pi\), write
\(Y_t^\pi:=Y_t^{X_0^\pi}\).
We apply the triangle inequality along the interpolation
\[
X_t^{x_0}(W)
\longrightarrow
Y_t^{x_0}
\longrightarrow
Y_t^\pi
\longrightarrow
X_t^\pi(W).
\]

We first bound the term comparing \(X_t^{x_0}\) and \(Y_t^{x_0}\). Another
application of convexity, now conditioning further on
\(\mathcal H_{\mathrm{SR}}\), gives
\[
\begin{aligned}
&\left\|
\Law(X_t^{x_0}(W)\mid\mathcal H_{G,Y})
-
\Law(Y_t^{x_0}(W)\mid\mathcal H_{G,Y})
\right\|_{\TV} \\
&\qquad\le
\mathbb E\left[
\left\|
\Law(X_t^{x_0}(W)\mid\mathcal H_{G,Y}\vee\mathcal H_{\mathrm{SR}})
-
\Law(Y_t^{x_0}(W)\mid\mathcal H_{G,Y}\vee\mathcal H_{\mathrm{SR}})
\right\|_{\TV}
\,\middle|\,
\mathcal H_{G,Y}
\right].
\end{aligned}
\]

Here the inner quantity is measurable with respect to \(\mathcal H_{G,Y}\vee\mathcal H_{\mathrm{SR}}\), so the outer conditional expectation averages over the remaining randomness in \(\mathcal H_{\mathrm{SR}}\) given \(\mathcal H_{G,Y}\).
Conditional on
\(\mathcal H_{G,Y}\vee\mathcal H_{\mathrm{SR}}\), the Strong Red spins depend
only on their own histories and on the initial configuration at the bottoms of
those histories. Since \(Y_t^{x_0}\) modifies only the disjoint Purple set, the
Strong Red coordinates form a common conditional factor of the two laws. On
\(W\setminus V_{\mathrm{SR}}\), meanwhile, \(Y_t^{x_0}\) is product-uniform:
the Blue coordinates are product-uniform, and the Purple coordinates have been
replaced by fresh product-uniform randomness. Therefore,
\[
\begin{aligned}
&\left\|
\Law(X_t^{x_0}(W)\mid\mathcal H_{G,Y}\vee\mathcal H_{\mathrm{SR}})
-
\Law(Y_t^{x_0}(W)\mid\mathcal H_{G,Y}\vee\mathcal H_{\mathrm{SR}})
\right\|_{\TV}\\
&\qquad=
\left\|
\Law(X_t^{x_0}(W\setminus V_{\mathrm{SR}})
\mid\mathcal H_{G,Y}\vee\mathcal H_{\mathrm{SR}})
-
\nu_{W\setminus V_{\mathrm{SR}}}
\right\|_{\TV}.
\end{aligned}
\]
Conditional on \(\mathcal H_{G,Y}\vee\mathcal H_{\mathrm{SR}}\), we may apply
Lemma~\ref{lem:MP-overlap} on \(W\setminus V_{\mathrm{SR}}\), with exceptional
set \(S=V_{\mathrm{P}}\). Indeed, only the Purple coordinates may retain
information from the initial condition; conditional on \(V_{\mathrm{P}}\), the
remaining Blue coordinates are independent uniform spins. Hence,
\[
\begin{aligned}
&\left\|
\Law(X_t^{x_0}(W\setminus V_{\mathrm{SR}})
\mid\mathcal H_{G,Y}\vee\mathcal H_{\mathrm{SR}})
-
\nu_{W\setminus V_{\mathrm{SR}}}
\right\|_{\TV}\\
&\qquad\le
\frac12
\left(
\mathbb E\left[
q^{|V_{\mathrm{P}}^{(1)}\cap V_{\mathrm{P}}^{(2)}|}-1
\,\middle|\,
\mathcal H_{G,Y}\vee\mathcal H_{\mathrm{SR}}
\right]
\right)^{1/2}.
\end{aligned}
\]

The same argument bounds the term comparing \(Y_t^\pi\) and \(X_t^\pi\):
condition first on \(X_0^\pi=\xi\), apply the preceding fixed-initial-state
estimate, and then average over \(\xi\sim\pi\). Thus, conditional on
\(\mathcal H_{G,Y}\), the two outer triangle terms together are bounded by

\[
\mathbb E\left[
\left(
\mathbb E\left[
q^{|V_{\mathrm{P}}^{(1)}\cap V_{\mathrm{P}}^{(2)}|}-1
\,\middle|\,
\mathcal H_{G,Y}\vee\mathcal H_{\mathrm{SR}}
\right]
\right)^{1/2}
\,\middle|\,
\mathcal H_{G,Y}
\right].
\]

It remains to compare \(Y_t^{x_0}\) and \(Y_t^\pi\). Applying convexity once
more, this time by conditioning further on \(\mathcal H_{\mathrm{P}}\), gives

\[
\begin{aligned}
&\left\|
\Law(Y_t^{x_0}(W)\mid\mathcal H_{G,Y})
-
\Law(Y_t^\pi(W)\mid\mathcal H_{G,Y})
\right\|_{\TV} \\
&\qquad\le
\mathbb E\left[
\left\|
\Law(Y_t^{x_0}(W)\mid\mathcal H_{G,Y}\vee\mathcal H_{\mathrm{P}})
-
\Law(Y_t^\pi(W)\mid\mathcal H_{G,Y}\vee\mathcal H_{\mathrm{P}})
\right\|_{\TV}
\,\middle|\,
\mathcal H_{G,Y}
\right].
\end{aligned}
\]
Given \(\mathcal H_{G,Y}\vee\mathcal H_{\mathrm{P}}\), the set \(V_{\mathrm{P}}\) is fixed, and the spins \(Y_t^\xi(V_{\mathrm{P}})\) are independent product-uniform spins for every initial condition \(\xi\). Thus the \(V_{\mathrm{P}}\)-coordinates form a common conditional factor, so
\[
\begin{aligned}
&\left\|
\Law(Y_t^{x_0}(W)\mid\mathcal H_{G,Y}\vee\mathcal H_{\mathrm{P}})
-
\Law(Y_t^\pi(W)\mid\mathcal H_{G,Y}\vee\mathcal H_{\mathrm{P}})
\right\|_{\TV} \\
&\qquad=
\left\|
\Law(Y_t^{x_0}(W\setminus V_{\mathrm{P}})
\mid\mathcal H_{G,Y}\vee\mathcal H_{\mathrm{P}})
-
\Law(Y_t^\pi(W\setminus V_{\mathrm{P}})
\mid\mathcal H_{G,Y}\vee\mathcal H_{\mathrm{P}})
\right\|_{\TV}.
\end{aligned}
\]
Now
\(W\setminus V_{\mathrm{P}}=V_{\mathrm{Blue}}\sqcup V_{\mathrm{SR}}\), and
only the Strong Red coordinates may retain information from the initial
condition. Conditional on \(V_{\mathrm{SR}}\), the Blue coordinates are
independent uniform spins. 
Therefore, applying Lemma
\ref{lem:MP-overlap} conditionally to each of the two laws and using the triangle inequality through
\(\nu_{W\setminus V_{\mathrm P}}\), we get

\[
\begin{aligned}
&\left\|
\Law(Y_t^{x_0}(W\setminus V_{\mathrm{P}})
\mid\mathcal H_{G,Y}\vee\mathcal H_{\mathrm{P}})
-
\Law(Y_t^\pi(W\setminus V_{\mathrm{P}})
\mid\mathcal H_{G,Y}\vee\mathcal H_{\mathrm{P}})
\right\|_{\TV}\\
&\qquad\le
\left(
\mathbb E\left[
q^{|V_{\mathrm{SR}}^{(1)}\cap V_{\mathrm{SR}}^{(2)}|}-1
\,\middle|\,
\mathcal H_{G,Y}\vee\mathcal H_{\mathrm{P}}
\right]
\right)^{1/2}.
\end{aligned}
\]

Combining the three triangle bounds and then averaging over
\(\mathcal H_{G,Y}\) yields
\[
\bigl\|\Law(X_t^{x_0})-\pi\bigr\|_{\TV}
\le
\mathbb E[\Delta_{\mathrm{P}}]
+
\mathbb E[\Delta_{\mathrm{SR}}],
\]
where, for brevity,
\[
\begin{aligned}
\Delta_{\mathrm{P}}
&:=
\left(
\mathbb E\left[
q^{|V_{\mathrm{P}}^{(1)}\cap V_{\mathrm{P}}^{(2)}|}-1
\,\middle|\,
\mathcal H_{G,Y}\vee\mathcal H_{\mathrm{SR}}
\right]
\right)^{1/2},\\
\Delta_{\mathrm{SR}}
&:=
\left(
\mathbb E\left[
q^{|V_{\mathrm{SR}}^{(1)}\cap V_{\mathrm{SR}}^{(2)}|}-1
\,\middle|\,
\mathcal H_{G,Y}\vee\mathcal H_{\mathrm{P}}
\right]
\right)^{1/2}.
\end{aligned}
\]
Since the square-root function is concave, Jensen's inequality and the tower
property give

\[
\mathbb E[\Delta_{\mathrm{P}}]
\le
\left(
\mathbb E\left[
q^{|V_{\mathrm{P}}^{(1)}\cap V_{\mathrm{P}}^{(2)}|}-1
\right]
\right)^{1/2},
\qquad
\mathbb E[\Delta_{\mathrm{SR}}]
\le
\left(
\mathbb E\left[
q^{|V_{\mathrm{SR}}^{(1)}\cap V_{\mathrm{SR}}^{(2)}|}-1
\right]
\right)^{1/2}.
\]

This proves the claimed bound for the arbitrary initial condition \(x_0\).
Taking the maximum over \(x_0\in[q]^\Lambda\) gives the stated bound on
\(d_\Lambda(t)\).

\end{proof}

In Subsection~\ref{subsec:cutoff-upper}, we will establish, at terminal time $t$, the estimates
\[
\mathbb E\left[
q^{|V_{\mathrm{P}}^{(1)}\cap V_{\mathrm{P}}^{(2)}|}-1
\right]
=
o\bigl(|\Lambda|e^{-2\mathfrak{r}t}\bigr),
\qquad
\mathbb E\left[
q^{|V_{\mathrm{SR}}^{(1)}\cap V_{\mathrm{SR}}^{(2)}|}-1
\right]
=
O\bigl(|\Lambda|e^{-2\mathfrak{r}t}\bigr).
\]
Choose \(t_\star\) so that
\(|\Lambda|e^{-2\mathfrak{r}t_\star}=1\), and set \(T:=t_\star+s\). Evaluating
the preceding estimates at \(t=T\) gives

\[
\mathbb E\left[
q^{|V_{\mathrm{P}}^{(1)}\cap V_{\mathrm{P}}^{(2)}|}-1
\right]
=
o(e^{-2\mathfrak{r}s}),
\qquad
\mathbb E\left[
q^{|V_{\mathrm{SR}}^{(1)}\cap V_{\mathrm{SR}}^{(2)}|}-1
\right]
=
O(e^{-2\mathfrak{r}s}).
\]

Taking square roots and applying Proposition~\ref{prop:IP-overlap} yields the
desired upper bound on the total-variation distance in
Theorem~\ref{thm:main}.

%%%%%%%%%%%%%%%%%%%%%%%%%%%%%%%%%%%%%%%%%%%%%%%%%%%%%%%%%%%%%%%%%%%%%%%%%%%%%%%%%%%%%%%%%%%%%%%%%%%%%%%%%%%%%%%%%%%%%%%%%%%%%%%%%%%%%%%%%%%%%%%%%%%%%%%%%%%%%%%%%%%%%%%%%%%%%%%%%%%%%%%%%%%%%%%%%%%%%%%%%%%%%%%%%%%%%%%%%%%%%%%%%%%%%%%%%%%%%%%%

\medskip

\section{Sausage decomposition and magnetization}\label{sec:sausage}

In this section we develop the sausage estimates used in the cutoff proof and in the extremality results. We first construct a regeneration, or sausage, decomposition of a one-site backward history. This gives a signed renewal kernel \(\IK(z,t)\) and identifies the decay rate $\mathfrak{r}$. We then state the Fourier input needed to control the signed convolution powers and derive an exponential \(\ell^1\)-bound on \(\IK\). Finally, we convert these kernel estimates into magnetization bounds and a cone-confined estimate for Strong Red mass.

We begin with the i.i.d.\ graphical construction on \(\mathbb Z^d\). All renewal objects below---the sausage law \(\mu\), the signed weight \(p\), and the influence kernel \(\IK\)---are defined in this infinite-volume construction. For finite-volume applications, the kernel is periodized in the extremality proof, while the cutoff estimates are transferred to the torus by local couplings in regions where the quotient map is injective.

Throughout this section and Appendix~\ref{sec:appendix}, we fix
\begin{equation}\label{eq:pstop-choice}
    \varepsilon:=\frac18,
    \qquad
    p_{\mathrm{stop}}:=\frac12,
    \qquad
    \eta:=1-\log(1-p_{\mathrm{stop}})=1+\log 2.
\end{equation}
In particular,
\[
    1<\eta=1+\log 2<\frac74=2(1-\varepsilon).
\]

\subsection{Sausage renewal decomposition}
\label{subsec:sausage-def}

Fix a space-time point \((v,T)\in\mathbb Z^d\times[0,\infty)\), and expose its backward
history $\widehat{\mathcal H}(s):=\mathcal H_v(T-s)$, where \(s\) denotes backward time. Thus \(s=0\)
corresponds to real time \(T\). Let \(N_s:=|\widehat{\mathcal H}(s)|\) and, on \(\{N_s=1\}\), let \(V_s\) be
the unique vertex in \(\widehat{\mathcal H}(s)\).

We decompose the backward history into independent regeneration blocks, called \emph{sausages}
(see Figure~\ref{fig:phase_mechanics}).
Each sausage begins at an integer backward time when the cluster is a singleton.
From that time we run two competing mechanisms:
(i) the next Poisson ring at the current vertex; and
(ii) an independent ``stop coin'' checked at integer times.
If the coin succeeds before any ring, the sausage is a vertical one-lineage block.
If a ring occurs first, the history may branch; coin checks are suspended during the ensuing excursion,
and the sausage ends at the first later integer time at which the history has returned to a singleton.
If the history dies out before returning to a singleton, we record a cemetery outcome.

% ============================
% Sausages / regeneration blocks
% ============================
\begin{defn}[Sausages]\label{def:sausages}
Set \(\tau_0:=0\). Let \((A_n)_{n\ge1}\) be i.i.d. \(\mathrm{Bernoulli}(p_{\mathrm{stop}})\),
independent of the graphical construction. For \(k\ge1\), on
\(\{\tau_{k-1}<\infty,\ N_{\tau_{k-1}}=1\}\), define
\begin{align*}
\mathsf r_k&:=\inf\{s>\tau_{k-1}:\ \text{the Poisson clock at }V_{\tau_{k-1}}\text{ rings at backward time }s\},\\
\mathsf c_k&:=\inf\{n\in\mathbb Z_{\ge1}:\ n>\tau_{k-1},\ A_n=1\},
\end{align*}
and set
\[
\tau_k :=
\begin{cases}
\mathsf c_k, & \mathsf c_k\le \mathsf r_k,\\[2pt]
\inf\{n\in\mathbb Z_{\ge1}:\ n>\mathsf r_k,\ N_n=1\}, & \mathsf r_k<\mathsf c_k,
\end{cases}
\qquad(\inf\emptyset:=\infty).
\]
On \(\{\tau_k<\infty\}\), let \(x_k\) be the restriction of the backward history to the slab
\([\tau_{k-1},\tau_k]\), recentered by the translation sending
\((V_{\tau_{k-1}},\tau_{k-1})\) to \((0,0)\). Let \(\Xi\) be the set of all such recentered slabs, and adjoin a cemetery symbol \(\dagger\). On \(\{\tau_k=\infty\}\), set \(x_k:=\dagger\), and write \(\Xi^\dagger:=\Xi\cup\{\dagger\}\).

For $x\in\Xi$, define its length, displacement, and number of branch-out events:
\[
\ell(x):=\tau_k-\tau_{k-1}\in\mathbb{Z}_{\ge1},
\quad
y(x):=V_{\tau_k}-V_{\tau_{k-1}}\in\mathbb{Z}^d,
\quad
b(x):=\#\{\text{branch-out events inside }x\}\in\mathbb Z_{\ge 0},
\]
and set $\ell(\dagger)=0$, $y(\dagger)=0$, $b(\dagger)=0$.

Fix a distinguished color, say $1\in[q]$. Define the signed influence
\[
p(x):=\frac{q}{q-1}\left(
\mathbb P\!\left(
\sigma_{T-\tau_{k-1}}(V_{\tau_{k-1}})=1
\,\middle|\,
\sigma_{T-\tau_k}(V_{\tau_k})=1,\ x_k=x
\right)-\frac1q
\right)\in\Big[-\frac1{q-1},\,1\Big],
\quad p(\dagger):=0.
\]
\end{defn}

By the strong Markov property at the regeneration times and by independence of the stop coins, up to and including the first occurrence of $\dagger$, the sequence $(x_k)_{k\ge1}$ is i.i.d.\ with common law $\mu$ on $\Xi^\dagger$.

We use the notation
\[
    B_\infty(u,k):=\{z\in\mathbb Z^d:\|z-u\|_\infty\le k\}.
\]

%=======================================
%=============Sausage def Fig=============
%=======================================
\begin{figure}[htbp]
\centering
\begin{tikzpicture}[
    x=0.45cm, y=0.45cm, 
    grid/.style={gray!40, dashed, line width=0.45pt},
    hist/.style={black, line width=1.3pt, line cap=round, line join=round},
    death/.style={circle, fill=black, inner sep=1.0pt},
    regen line/.style={draw=red!80!black, thick},
    pinch point/.style={circle, fill=red!80!black, inner sep=1.8pt},
    label/.style={font=\scriptsize}, 
    annotation/.style={font=\tiny, align=left, text=blue!70!black} 
]

  % --- Time Grid ---
  \foreach \t in {1,...,16}{
    \draw[grid] (-4.5,\t) -- (5,\t);
    \node[left, gray!80, font=\tiny] at (-4.5,\t) {\t};
  }

  % =======================================================
  % SAUSAGE NAMES ON THE LEFT 
  % =======================================================
  \draw[decorate, decoration={brace, amplitude=6pt}, thick, black] 
    (-6, 10) -- (-6, 15) node[midway, left=8pt, align=right, font=\scriptsize\bfseries] 
    {$k$-th\\Sausage};

  \draw[decorate, decoration={brace, amplitude=6pt}, thick, black] 
    (-6, 2) -- (-6, 10) node[midway, left=8pt, align=right, font=\scriptsize\bfseries] 
    {$(k\!+\!1)$-th\\Sausage};

  % =======================================================
  % k-th SAUSAGE: ONLY COIN-TOSSING PHASE 
  % =======================================================
  \draw[regen line] (-4.5,15) -- (4,15) node[right, text=red!80!black, font=\tiny] {$\tau_{k-1}$};

  \draw[hist] (0,15) -- (0,10);
  \node[pinch point, label={[above, yshift=2pt] $V_{\tau_{k-1}}$}] at (0,15) {};
  \node[pinch point] at (0,10) {};
  \draw[regen line] (-4.5,10) -- (4,10) node[right, text=red!80!black, font=\tiny] {$\tau_{k}$};

  \node[annotation] at (2.7, 14) {Check $A=0$ (Fail)};
  \node[annotation] at (2.7, 13) {Check $A=0$ (Fail)};
  \node[annotation] at (2.7, 12) {Check $A=0$ (Fail)};
  \node[annotation] at (2.7, 11) {Check $A=0$ (Fail)};

  % FIXED: Added \tiny before \bfseries to prevent size reset
  \node[annotation, text=red!80!black, font=\tiny\bfseries] at (-0.4, 10.3) {Check $A=1$ (Success!)};

  \draw[->, gray, thin] (1.1, 14) -- (0.2, 14);
  \draw[->, gray, thin] (1.1, 13) -- (0.2, 13);
  \draw[->, gray, thin] (1.1, 12) -- (0.2, 12);
  \draw[->, gray, thin] (1.1, 11) -- (0.2, 11);

  \draw[decorate, decoration={brace, amplitude=5pt}, thick, blue!70!black] 
    (5.5, 15) -- (5.5, 10) node[midway, right=8pt, align=left, font=\scriptsize] 
    {Coin-tossing\\phase only\\$\mathsf c_k \le \mathsf r_k$};

  % =======================================================
  % (k+1)-th SAUSAGE: COIN-TOSSING + EXCURSION 
  % =======================================================
  \draw[hist] (0,10) -- (0,8.4);

  \node[annotation] at (2.7, 9) {Check $A=0$ (Fail)};
  \draw[->, gray, thin] (1.1, 9) -- (0.2, 9);

  \draw[decorate, decoration={brace, amplitude=3pt}, thick, blue!70!black] 
    (5.5, 10) -- (5.5, 8.4) node[midway, right=6pt, align=left, font=\scriptsize] 
    {Coin-tossing phase};

  \draw[->, red!70!black, thick] (-2.5, 8.4) node[left, font=\tiny, align=right] {$\mathsf r_{k+1}$: Clock rings before check $\mathsf c_{k+1}$} -- (-0.2, 8.4);

  \draw[hist] (-1,8.4) -- (1,8.4);
  \draw[hist] (1,8.4) -- (1,6.5);
  \draw[hist] (0,6.5) -- (2,6.5);
  \draw[hist] (2,6.5) -- (2,3.8); \node[death] at (2,3.8) {}; 
  \draw[hist] (0,6.5) -- (0,5.1); \node[death] at (0,5.1) {}; 
  \draw[hist] (-1,8.4) -- (-1,3.8);
  \draw[hist] (-2,3.8) -- (0,3.8);
  \draw[hist] (-2,3.8) -- (-2,2.8); \node[death] at (-2,2.8) {}; 
  \draw[hist] (0,3.8) -- (0,2.4);

  \draw[dashed, red!70!black, thick] (-4.5, 2.8) -- (0.6, 2.8);
  \node[annotation, text=red!70!black] at (2.5, 2.8) {$X_t$ returns to $1$};

  \draw[hist] (0,2.4) -- (0,2);
  \node[pinch point] at (0,2) {};
  \draw[regen line] (-4.5,2) -- (4,2) node[right, text=red!80!black, font=\tiny] {$\tau_{k+1}$};

  % FIXED: Added \tiny before \bfseries to prevent size reset
  \node[annotation, text=red!80!black, font=\tiny\bfseries] at (1.0, 1.6) {Wait for next integer};

  \draw[decorate, decoration={brace, amplitude=5pt}, thick, blue!70!black] 
    (5.5, 8.4) -- (5.5, 2) node[midway, right=8pt, align=left, font=\scriptsize] 
    {Excursion phase\\$\mathsf r_{k+1} < \mathsf c_{k+1}$};

  \draw[decorate, decoration={brace, mirror, amplitude=4pt}, thick, gray] 
    (-2.2, 8.4) -- (-2.2, 2.8) node[midway, left=3pt, align=right, font=\tiny] 
    {Cluster size\\$X_t > 1$};

\end{tikzpicture}
\caption{Mechanics of the two-phase regeneration rule. 
\textbf{Top ($k$-th sausage):} $\mathsf c_k\le \mathsf r_k$ (coin succeeds before the next ring), so $\tau_k=\mathsf c_k$.
\textbf{Bottom ($(k+1)$-th sausage):} $\mathsf r_{k+1}<\mathsf c_{k+1}$, so an excursion begins at $\mathsf r_{k+1}$ and $\tau_{k+1}$ is the first later integer time with $X_t=1$.
}
\label{fig:phase_mechanics}
\end{figure}
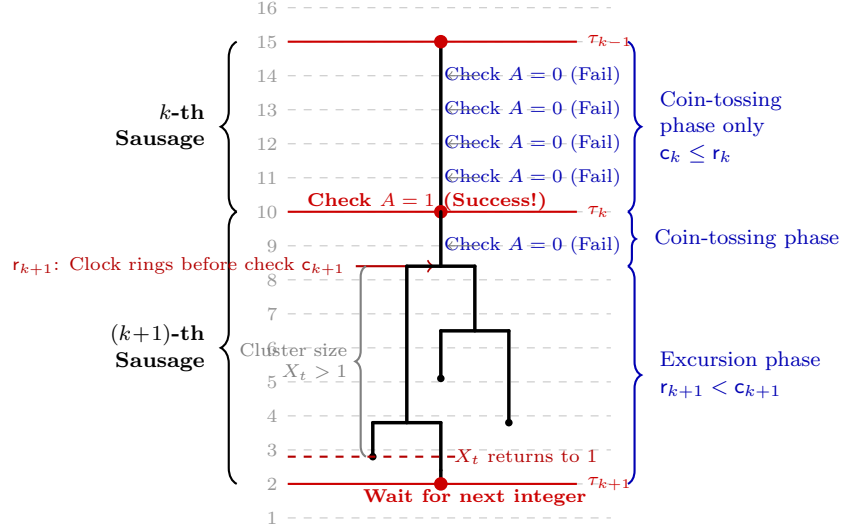
%=======================================
%=========Sausage Fig end=============
%=======================================

In the high-temperature  regime, the backward cluster spends most of its time at size one. The coin-toss filter prevents cutting at every such visit, producing i.i.d.\ renewal blocks.
In the influence expansions we sum only over genuine sausages $x\in\Xi$, since the cemetery outcome contributes zero.

\begin{rmk}[Signed weights]
For \(q\ge3\), the Potts heat-bath dynamics is not monotone, and the influence \(p(x)\) need
not be nonnegative. The point of the sausage decomposition is that, although individual weights
may be signed, their renewal structure can still be analyzed through Fourier estimates for signed
convolution kernels.
\end{rmk}

We record the elementary sausages that determine the leading behavior of the renewal weights.

\begin{lemma}[Canonical sausages]\label{lem:canonical-sausages}
Let
\[
\Xi_i:=\{x\in\Xi:\ b(x)=i\},
\qquad
\Xi_{\ge3}:=\{x\in\Xi:\ b(x)\ge3\}.
\]
Then:

\begin{enumerate}
\item If \(x\in\Xi_0\), then \(y(x)=0\) and \(p(x)=1\).

\item If \(x\in\Xi_1\), then \(p(x)=p_1\), where 
\[
p_1=\frac1{2d}+O_{d,q}(\beta)
=\frac1{2d}+O_{d,q}(\alpha).
\]
Thus, after decreasing \(\beta_0(d,q)\), we have \(1/(4d)<p_1<1\).

\item If \(x\in\Xi_2\), then
$p(x)\in\{p_1,\ p_1^2\}$,
and in particular
\begin{equation}\label{eq:p-two-branch}
0\le p(x)\le p_1.
\end{equation}
\end{enumerate}
\end{lemma}

\begin{proof}
If \(b(x)=0\), the history remains a single vertical line throughout the sausage. Hence
\(y(x)=0\), and conditioning the bottom spin to be \(1\) forces the top spin to be \(1\), so \(p(x)=1\).

If $b(x)=1$, there is exactly one non-oblivious update on the unique ancestral line. Conditional on the event \(x_k=x\), this update sees one distinguished neighbor forced to be color \(1\), while the other \(2d-1\) neighbors are i.i.d.\ uniform. The conditional probability of color \(1\) under the residual law is
\[
F_{\mathrm{res}}(\beta)
:=
\frac{F(\beta)-(1-\alpha)/q}{\alpha},
\qquad
F(\beta):=\frac{e^{\beta(N_1+1)}}{e^{\beta(N_1+1)}+\sum_{i=2}^q e^{\beta N_i}},
\]
where
$(N_1,\dots,N_q)\sim\mathrm{Mult}(2d-1;1/q,\dots,1/q)$.
Since \((1-\alpha)/q+\alpha/q=1/q\), it follows that
\[
p_1
=
\frac{q}{q-1}\left(\mathbb E[F_{\mathrm{res}}(\beta)]-\frac1q\right)
=
\frac{q}{(q-1)\alpha}
\left(\mathbb E[F(\beta)]-\frac1q\right).
\]
Since \(F(0)=1/q\), 
and
\[
F'(0)=\frac{q(N_1+1)-2d}{q^2},
\qquad
\mathbb E[F'(0)]=\frac{q-1}{q^2},
\]
while the softmax Hessian is uniformly bounded by \(2(2d)^2\) on \([0,1]\), Taylor's theorem
gives 
\[
\frac{1}{\alpha(\beta)}
\left(
\frac{\beta}{q}-\frac{q}{q-1}(2d)^2\beta^2
\right)
\le
p_1
\le
\frac{1}{\alpha(\beta)}
\left(
\frac{\beta}{q}+\frac{q}{q-1}(2d)^2\beta^2
\right).
\]
Finally, \(\alpha(\beta)=2d\beta/q+O_{d,q}(\beta^2)\), which gives the stated expansion of \(p_1\).

If \(b(x)=2\), the sausage starts and ends with a singleton. There is a unique surviving
ancestral line. All side branches die at oblivious updates before the end of the sausage and
therefore contribute only uniform inputs. Along the surviving line there are either one or two
non-oblivious updates, contributing respectively \(p_1\) or \(p_1^2\). This proves the claim.
\end{proof}

%=========================================
%=============Sausage pic=================
%=========================================
\begin{figure}[htbp]
\centering

% Define global tikz styles to keep the code clean
\tikzset{
  grid/.style={gray!50, dashed, line width=0.45pt},
  hist/.style={black, line width=1.1pt, line cap=round, line join=round},
  nodept/.style={circle, fill=black, inner sep=0.9pt},
  death/.style={circle, fill=black, inner sep=0.9pt},
  label/.style={font=\scriptsize}
}

% =======================================================
% (a) Example 1: Straight line (length 6, Time 10 to 4)
% =======================================================
\begin{subfigure}[b]{0.22\textwidth}
\centering
\begin{tikzpicture}[x=0.45cm, y=0.6cm]
  \foreach \t in {4,...,10}{
    \draw[grid] (-1.5,\t) -- (1.5,\t);
    \node[left, gray!80, font=\tiny] at (-1.5,\t) {\t};
  }

  \draw[hist] (0,10) -- (0,4);
  \node[nodept] at (0,10) {};
  \node[nodept] at (0,4) {};

  \node[label, anchor=south, yshift=2pt] at (0,10) {start};

  % l(x) and p(x) next to each other
  \node[label] at (0,3.2) {$\ell(x)=6, \ p(x)=1$};

  % y(x) label at the end point
  \node[label, anchor=west] at (0.3, 4) {$y(x)=0$};
\end{tikzpicture}
\caption{$x\in\Xi_0$}
\end{subfigure}%
\hfill
% =======================================================
% (b) Example 2: Simple branch (length 5, Time 8 to 3)
% =======================================================
\begin{subfigure}[b]{0.22\textwidth}
\centering
\begin{tikzpicture}[x=0.45cm, y=0.6cm]
  \foreach \t in {3,...,8}{
    \draw[grid] (-1.5,\t) -- (1.5,\t);
    \node[left, gray!80, font=\tiny] at (-1.5,\t) {\t};
  }

  \draw[hist] (0,8) -- (0,6.7);
  \node[nodept] at (0,8) {};
  \node[nodept] at (0,6.7) {};
  \node[label, anchor=south, yshift=2pt] at (0,8) {start};

  % Branching out at t=6.7
  \draw[hist] (-1,6.7) -- (1,6.7);
  \draw[hist] (-1,6.7) -- (-1,3.6); \node[death] at (-1,3.6) {}; % Left dies
  \draw[hist] (1,6.7) -- (1,3);     \node[nodept] at (1,3) {};   % Right survives

  % l(x) and p(x) next to each other
  \node[label] at (0,2.2) {$\ell(x)=5, \ p(x)=p_1$};

  % y(x) label at the end point
  \node[label, anchor=west] at (1.2, 3) {$y(x)=1$};
\end{tikzpicture}
\caption{$x\in\Xi_1$}
\end{subfigure}%
\hfill
% =======================================================
% (c) Example 3: Two-level branch (length 5, Time 12 to 7)
% =======================================================
\begin{subfigure}[b]{0.22\textwidth}
\centering
\begin{tikzpicture}[x=0.45cm, y=0.6cm]
  \foreach \t in {7,...,12}{
    \draw[grid] (-1.5,\t) -- (2.5,\t);
    \node[left, gray!80, font=\tiny] at (-1.5,\t) {\t};
  }

  \draw[hist] (0,12) -- (0,11.4);
  \node[nodept] at (0,12) {};
  \node[nodept] at (0,11.4) {};
  \node[label, anchor=south, yshift=2pt] at (0,12) {start};

  % First split at t=11.4
  \draw[hist] (-1,11.4) -- (1,11.4);
  \draw[hist] (-1,11.4) -- (-1,9.3); \node[death] at (-1,9.3) {}; % Left dies

  % Right continues and splits again at t=10.7
  \draw[hist] (1,11.4) -- (1,10.7);
  \node[nodept] at (1,10.7) {};
  \draw[hist] (0,10.7) -- (2,10.7);

  \draw[hist] (2,10.7) -- (2,7.7); \node[death] at (2,7.7) {}; % Right dies
  \draw[hist] (0,10.7) -- (0,7);   \node[nodept] at (0,7) {};   % Left survives

  % l(x) and p(x) next to each other
  \node[label] at (0.5,6.2) {$\ell(x)=5, \ p(x)=p_1^2$};

  % y(x) label at the end point
  \node[label, anchor=west] at (0.3, 7) {$y(x)=0$};
\end{tikzpicture}
\caption{$x\in\Xi_2$}
\end{subfigure}%
\hfill
% =======================================================
% (d) Example 4: Complex branch (length 5, Time 6 to 1)
% =======================================================
\begin{subfigure}[b]{0.28\textwidth}
\centering
\begin{tikzpicture}[x=0.45cm, y=0.6cm]
  \foreach \t in {1,...,6}{
    \draw[grid] (-3.5,\t) -- (2.5,\t);
    \node[left, gray!80, font=\tiny] at (-3.5,\t) {\t};
  }

  \draw[hist] (0,6) -- (0,5.5);
  \node[nodept] at (0,6) {};
  \node[nodept] at (0,5.5) {};
  \node[label, anchor=south, yshift=2pt] at (0,6) {start};

  % Top split at t=5.5
  \draw[hist] (-1,5.5) -- (1,5.5);

  % --- Right Subtree (starts at 1) ---
  \draw[hist] (1,5.5) -- (1,3.0);
  \node[nodept] at (1,3.0) {};
  \draw[hist] (0,3.0) -- (2,3.0);
  \draw[hist] (0,3.0) -- (0,2.5); \node[death] at (0,2.5) {};
  \draw[hist] (2,3.0) -- (2,1.7); \node[death] at (2,1.7) {};

  % --- Left Subtree (starts at -1) ---
  \draw[hist] (-1,5.5) -- (-1,4.8);
  \node[nodept] at (-1,4.8) {};
  \draw[hist] (-2,4.8) -- (0,4.8);

  % Inner Left (starts at 0)
  \draw[hist] (0,4.8) -- (0,3.8);
  \node[nodept] at (0,3.8) {};
  \draw[hist] (-1,3.8) -- (1,3.8);
  \draw[hist] (-1,3.8) -- (-1,3.4); \node[death] at (-1,3.4) {};

  % Outer Left (starts at -2)
  \draw[hist] (-2,4.8) -- (-2,2.2);
  \node[nodept] at (-2,2.2) {};
  \draw[hist] (-3,2.2) -- (-1,2.2);
  \draw[hist] (-3,2.2) -- (-3,1.5); \node[death] at (-3,1.5) {};

  % Final survivor reaches the integer line t=1 at x=-1
  \draw[hist] (-1,2.2) -- (-1,1);   \node[nodept] at (-1,1) {}; 

  % Only l(x) as requested
  \node[label] at (-0.5,0.2) {$\ell(x)=5$};

  % y(x) label at the end point
  \node[label, anchor=west] at (-0.7, 1) {$y(x)=-1$};
\end{tikzpicture}
\caption{$x\in\Xi_{\geq 3}$}
\end{subfigure}

\caption{Four examples of 1D schematic sausages representing the backward history cluster on $\mathbb{Z}$.}
\label{fig:sausages_backward_shifted_labels}
\end{figure}
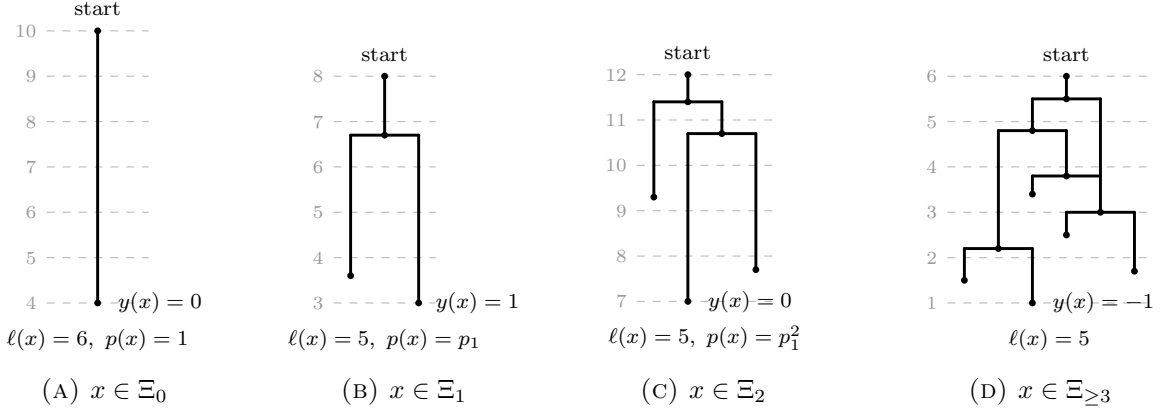
%=========================================
%===========Sausage pic end=================
%=========================================

\subsection{Branch-out and length estimates}
\label{subsec:branch-length}

We next collect the estimates on the length and complexity of a sausage.
These will show that sausages with many branch-outs are exponentially rare in the number of branch-outs, which is
what lets the higher-order terms \(Z^{(2)},Z^{(\ge3)}\) in Lemma~\ref{lem:Zp-bound} be treated as
corrections to the dominant one-branch term, and what later controls the Strong-Red mass in
Section~\ref{sec:cutoff}. Let \((N_s)_{s\ge0}\)
be the continuous-time branching process in which each particle lives an exponential time of rate
\(1\), and at death produces no offspring with probability \(1-\alpha\), and one child at each of
its \(2d\) neighboring sites with probability \(\alpha\). Its mean offspring number is $m:=2d\alpha<1$.
We work in the subcritical regime \(m<1\). Let \(J_s\) be the number of branch-out events up to
time \(s\).

\begin{lem}[Branch-out penalty]\label{lem:branch-penalty}
There exists a constant \(C_{\mathrm{bp}}\) such that the following holds. 
For $0<\alpha\le 1/({2C_{\mathrm{bp}}})$, set $\vartheta:=C_{\mathrm{bp}}\alpha\in(0,1/2]$. 
Then, for every integer \(t\ge1\) and every \(k\ge0\),
\begin{equation}\label{eq:staytwoJ}
\mathbb P\bigl(
N_1>1,\ldots,N_{t-1}>1,\ N_t\geq 1,\ J_t\ge k
\bigr)
\le
e^{1-\varepsilon}\,
\vartheta^k e^{-2(1-\varepsilon)t}.
\end{equation}
\end{lem}

\begin{proof}

For \(s\ge1\), \(\theta\ge1\), and \(t\ge0\), define
\[
u(s,\theta,t):=\mathbb E\big[s^{N_t}\theta^{J_t}\big].
\]
For a short time interval \([0,h]\), with probability \(1-h+o(h)\) nothing happens, with
probability \((1-\alpha)h+o(h)\) the initial particle dies with no offspring, and with probability
\(\alpha h+o(h)\) it branches to \(2d\) offspring, contributing an extra factor \(\theta\). By the
branching property, this gives
\[
u(s,\theta,t+h)
=
(1-h)u(s,\theta,t)
+
h\Bigl((1-\alpha)+\alpha\theta\,u(s,\theta,t)^{2d}\Bigr)
+
o(h).
\]
Subtracting \(u(s,\theta,t)\), dividing by \(h\), and letting \(h\downarrow0\), we obtain
\[
\partial_t u(s,\theta,t)
=
-u(s,\theta,t)+(1-\alpha)+\alpha\theta\,u(s,\theta,t)^{2d},
\qquad
u(s,\theta,0)=s.
\]

Choose \(\lambda>0\) so that
\[
    e^\lambda(1-e^{-\varepsilon})>1-e^{-1},
\]
and set \(s_\star:=e^{\lambda+1-\varepsilon}\).
For \(\alpha\) sufficiently small one may choose
\(\theta=\vartheta^{-1}\), with \(\vartheta=O(\alpha)\), so that
\[
    (1-\alpha)+\alpha\theta s_\star^{2d}-s_\star\le0.
\]
Indeed,
we can set $\vartheta:=C_{\mathrm{bp}}\alpha\in(0,1/2]$ with $C_{\mathrm{bp}}:=\max\Big\{\frac{s_\star^{2d}}{s_\star-1},\ \frac{(1-e^{-1})s_\star^{2d}}{e^\lambda(1-e^{-\varepsilon})-(1-e^{-1})}\Big\}$, which is available by taking $\alpha$ small enough.
Then,
\[
u(s_\star,\theta,t)\le s_\star
\]
for all $t\geq 0$ by a barrier condition for the ODE. Moreover, by integrating
the equation on \([0,1]\), we have
\[
u(s_\star ,\theta,1)\le 1+(s_\star-1)e^{-1}+\alpha(1-e^{-1})\big(\theta s_\star^{2d}-1\big),
\]
which implies
\[
    u(s_\star,\theta,1)\le e^\lambda.
\]

For integer \(t\ge0\), define
\[
    W_t:=\sum_{i=1}^t N_i,
    \qquad
    M_t:=\exp\{\lambda N_t+(1-\varepsilon)W_t\}\theta^{J_t}.
\]
The estimate \(u(s_\star,\theta,1)\le e^\lambda\) implies that \((M_t)_{t\in\mathbb N}\) is a
supermartingale and $\mathbb{E}[M_t]\le \mathbb{E}[M_0]= e^\lambda$.

On the event
\(\{N_1>1,\ldots,N_{t-1}>1,\ N_t\geq1,\ J_t\ge k\}\), one has \(W_t\ge2t-1\) and
\(\theta^{J_t}\ge\theta^k\). Therefore
\[
    \mathbb P\bigl(N_1>1,\ldots,N_{t-1}>1,\ N_t\geq1,\ J_t\ge k\bigr)
    \le
     e^{1-\varepsilon}\theta^{-k}e^{-2(1-\varepsilon)t}.
\]
Since \(\theta^{-1}=\vartheta\), this proves \eqref{eq:staytwoJ}.

\end{proof}

The stop-coin construction gives a geometric-exponential tail for vertical sausages, while Lemma~\ref{lem:branch-penalty} shows that branching sausages are smaller by powers of \(\alpha\).

\begin{prop}\label{prop:sausagetail}
Assume \(0<\alpha\le(2C_{\mathrm{bp}})^{-1}\) and recall
\(\vartheta=C_{\mathrm{bp}}\alpha\). There are constants
\(c_{\mathrm{len}},C_{\mathrm{len}},C_{\mathrm{jnt}}<\infty\) such that for all integers $\ell\ge1$ and $b\ge1$,
\begin{equation}\label{eq:mu-joint}
\mu\big(\ell(x_1)=\ell,\ b(x_1)\ge b\big)\le C_{\mathrm{jnt}}\,\vartheta^{\,b}\,e^{-\eta\ell}.
\end{equation}
Moreover, for every integer $\ell\ge1$,
\begin{equation}\label{eq:len-pmf}
c_{\mathrm{len}}\,e^{-\eta \ell}
\ \le\ \mu(\ell(x_1)=\ell)
\ \le\ c_{\mathrm{len}}\,e^{-\eta \ell}\,\bigl(1+C_{\mathrm{len}}\,\alpha\bigr),
\end{equation}
and hence
\begin{equation}\label{eq:len-tail}
\mu(\ell(x_1)>t)\ \le\ 
c_{\mathrm{tail}}\,e^{-\eta t},
\qquad
c_{\mathrm{tail}}:=\frac{c_{\mathrm{len}}\,e^{-\eta}(1+C_{\mathrm{len}}\alpha)}{1-e^{-\eta}}.
\end{equation}
\end{prop}

\begin{proof}
For the lower bound in \eqref{eq:len-pmf}, observe that if there is no clock ring up to time \(\ell\) and the stop coin first succeeds at time \(\ell\), then \(\ell(x_1)=\ell\) and \(b(x_1)=0\). This event has probability
\[
e^{-\ell}(1-p_{\mathrm{stop}})^{\ell-1}p_{\mathrm{stop}}
=
\frac{p_{\mathrm{stop}}}{1-p_{\mathrm{stop}}}
\bigl(e^{-1}(1-p_{\mathrm{stop}})\bigr)^\ell
=
c_{\mathrm{len}}e^{-\eta\ell},
\qquad
c_{\mathrm{len}}:=\frac{p_{\mathrm{stop}}}{1-p_{\mathrm{stop}}}.
\]

For \eqref{eq:mu-joint}, let \(a\in\{0,\ldots,\ell-1\}\) be the last integer time before
the first clock ring. 

The probability of no ring and no successful stop coin up to \(a\) is
\(e^{-\eta a}\). Conditional on this, the remainder of the sausage is dominated by the branching
process in Lemma~\ref{lem:branch-penalty}. On
\(\{\ell(x_1)=\ell,\ b(x_1)\ge b\}\), this domination gives
\(N_1>1,\ldots,N_{\ell-a-1}>1\), \(N_{\ell-a}\ge1\), and
\(J_{\ell-a}\ge b\). Thus

\[
\mu(\ell(x_1)=\ell,\ b(x_1)\ge b)
\le
\sum_{a=0}^{\ell-1}
e^{-\eta a}\,
e^{1-\varepsilon}\vartheta^b e^{-2(1-\varepsilon)(\ell-a)}
\le
C_{\mathrm{jnt}}\vartheta^b e^{-\eta\ell},\qquad
C_{\mathrm{jnt}}:=\frac{e^{1-\varepsilon}}{e^{2(1-\varepsilon)-\eta}-1},
\]
because \(\eta<2(1-\varepsilon)\). Taking \(b=1\) and adding the \(b=0\) contribution gives the
upper bound in \eqref{eq:len-pmf}. Summing over \(\ell>t\) gives
\eqref{eq:len-tail}.

\end{proof}

The next lemma supplies the matching lower bounds showing that one- and two-branch sausages occur with the expected orders in \(\alpha\).

\begin{lemma}[Lower bounds for $\Xi_1$ and $\Xi_2$]\label{lem:X1X2-lower-bound}
Assume $\alpha\le \tfrac12$.
There exist explicit constants $C_1=C_1(d,p_{\mathrm{stop}})>0$ and $C_2=C_2(d,p_{\mathrm{stop}})>0$ such that
for every integer $l\ge1$,
\begin{align}
\mu\big(\ell(x_1)=l,\ b(x_1)=1\big) &\ge C_1\,\alpha\,e^{-\eta l}, \label{eq:Xi-1-lower}\\
\mu\big(\ell(x_1)=l,\ b(x_1)=2\big) &\ge C_2\,\alpha^2\,e^{-\eta l}. \label{eq:Xi-2-lower}
\end{align}
\end{lemma}

\begin{proof}
Set
\[
c_*:=\frac{1-e^{-1/4}}{2}>0.
\]
Since \(e^{-\eta}=e^{-1}(1-p_{\mathrm{stop}})\),
\begin{equation}\label{eq:base-factor}
\mathbb P\big(\text{no stop at }1,\dots,\ell-1,\ \text{no ring on }[0,\ell-1)\big)
=
e^\eta e^{-\eta\ell}.
\end{equation}

\smallskip
\noindent\textbf{Case \(b(x_1)=1\).}
Let \(E^{(1)}_\ell\) be the event that

\begin{enumerate}
\item there is no stop success at times \(1,\dots,\ell-1\) and no ring on \([0,\ell-1)\);
\item the first ring in \(I_1:=[\ell-1,\ell-\tfrac34]\) occurs and is non-oblivious;
\item one distinguished child has no ring up to time \(\ell\);
\item each of the other \(2d-1\) children has a first ring before time \(\ell-\tfrac12\), and that first ring is oblivious.
\end{enumerate}

Then \(E^{(1)}_\ell\subset\{\ell(x_1)=\ell,\ b(x_1)=1\}\): there is exactly one non-oblivious update, and since it occurs after time \(\ell-1\), the first integer time \(>\mathsf r_1\) at which the history returns to a singleton is \(\ell\).

By \eqref{eq:base-factor}, item (1) has probability \(e^\eta e^{-\eta\ell}\). Conditioned on (1), item (2) has probability at least \((1-e^{-1/4})\alpha\). Given the first ring time \(s\in I_1\), item (3) has probability at least \(e^{-1}\), while each side child satisfies item (4) with probability at least
\[
(1-e^{-1/4})(1-\alpha)\ge c_*
\qquad(\alpha\le\tfrac12).
\]
Hence
\[
\mu\big(\ell(x_1)=\ell,\ b(x_1)=1\big)\ge \mathbb P(E^{(1)}_\ell)
\ge
C_1\,\alpha\,e^{-\eta\ell},
\]
with
\[
C_1:=e^\eta(1-e^{-1/4})e^{-1}c_*^{\,2d-1}.
\]

\smallskip
\noindent\textbf{Case \(b(x_1)=2\).}
Let \(E^{(2)}_\ell\) be the event that

\begin{enumerate}
\item there is no stop success at times \(1,\dots,\ell-1\) and no ring on \([0,\ell-1)\);
\item the first ring in \(I_1:=[\ell-1,\ell-\tfrac34]\) occurs and is non-oblivious;
\item one distinguished child has no ring up to time \(\ell-\tfrac12\), while each of the other \(2d-1\) children has a first ring before time \(\ell-\tfrac12\), and that first ring is oblivious;
\item the distinguished child has a first ring in \(I_2:=[\ell-\tfrac12,\ell-\tfrac14]\), and it is non-oblivious;
\item one distinguished grandchild has no ring up to time \(\ell\), while each of the other \(2d-1\) grandchildren has a first ring before time \(\ell\), and that first ring is oblivious.
\end{enumerate}

Then \(E^{(2)}_\ell\subset\{\ell(x_1)=\ell,\ b(x_1)=2\}\): exactly two non-oblivious updates occur, and again the first integer time after the first branch-out at which the history is a singleton is \(\ell\).

Using \eqref{eq:base-factor}, item (1) has probability \(e^\eta e^{-\eta\ell}\), items (2) and (4) contribute at least \((1-e^{-1/4})\alpha\) each, item (3) contributes at least
\[
e^{-1/2}c_*^{\,2d-1},
\]
and item (5) contributes at least
\[
e^{-1/2}c_*^{\,2d-1}.
\]
Therefore
\[
\mu\big(\ell(x_1)=\ell,\ b(x_1)=2\big)\ge \mathbb P(E^{(2)}_\ell)
\ge
C_2\,\alpha^2\,e^{-\eta\ell},
\]
with
\[
C_2:=e^\eta(1-e^{-1/4})^2e^{-1}c_*^{\,2(2d-1)}.
\]
This proves \eqref{eq:Xi-1-lower} and \eqref{eq:Xi-2-lower}.
\end{proof}

%===========================IK

\subsection{The signed renewal kernel and the decay rate}
\label{subsec:signed-kernel}
The renewal analysis below introduces an exponential-tilt parameter
\(\kappa\in(-1,0)\) for the sausage law, defines the decay rate
\(\mathfrak r:=1+\kappa\), and proves, as \(\beta\downarrow0\),
\[
-\kappa=\alpha+O_{d,q}(\alpha^2).
\]
Consequently,
\[
-\kappa
=\Theta_{d,q}(\alpha)
=\Theta_{d,q}(\beta).
\]

For $z\in\mathbb Z^d$ and integer $t\ge0$, define the influence kernel
\begin{equation}\label{eq:IK-def}
\IK(z,t):=
\sum_{r\ge 1}\ \sum_{x_1,\ldots,x_r\in\Xi}
\Big(\prod_{i=1}^{r}\mu(x_i)p(x_i)\Big)\,
\mathbf 1\Big\{\sum_{i=1}^r y(x_i)=z,\ \sum_{i=1}^r \ell(x_i)=t\Big\},
\end{equation}
and set $\IK(z,0):=\mathbf 1\{z=0\}$.
Define the length-weights
\begin{equation*}
w_\ell:=\sum_{x\in\Xi:\ \ell(x)=\ell}e^{\ell}p(x)\mu(x),
\qquad
Z_p:=\sum_{\ell\ge 1}w_\ell=\mathbb E_\mu[e^{\ell(x)}p(x)].
\end{equation*}

\begin{lemma}\label{lem:Zp-bound}
For every \(\ell\ge1\),
\begin{equation}\label{eq:w-ell}
    \frac{p_{\mathrm{stop}}}{2}(1-p_{\mathrm{stop}})^{\ell-1}
    \le
    w_\ell
    \le
    \frac{3p_{\mathrm{stop}}}{2}(1-p_{\mathrm{stop}})^{\ell-1}.
\end{equation}
Moreover,
\begin{equation}\label{eq:Zp}
    Z_p
    =
    1+\frac{\alpha}{p_{\mathrm{stop}}}+O_{d,q}(\alpha^2)=1+2\alpha+O_{d,q}(\alpha^2).
\end{equation}
In particular, when $\alpha$ is small enough, \(1<Z_p\le2\).
\end{lemma}

\begin{proof}
For \eqref{eq:w-ell}, the contribution from \(b(x)=0\) is exact:
\[
e^\ell
\sum_{x:\,\ell(x)=\ell,\ b(x)=0}p(x)\mu(x)
=
p_{\mathrm{stop}}(1-p_{\mathrm{stop}})^{\ell-1}.
\]
On the other hand, \(|p(x)|\le1\) and \eqref{eq:mu-joint} with \(b=1\) give
\[
e^\ell
\sum_{x:\,\ell(x)=\ell,\ b(x)\ge1}|p(x)|\mu(x)
\le
C_{\mathrm{jnt}}C_{\mathrm{bp}}\alpha
(1-p_{\mathrm{stop}})^\ell.
\]
After decreasing \(\beta_0(d,q)\), the last expression is at most
$\frac{p_{\mathrm{stop}}}{2}
(1-p_{\mathrm{stop}})^{\ell-1}$,
which proves \eqref{eq:w-ell}.

We now prove \eqref{eq:Zp}. The total contribution from \(b(x)=0\) is
\[
\sum_{\ell\ge1}
p_{\mathrm{stop}}(1-p_{\mathrm{stop}})^{\ell-1}
=1.
\]
Therefore, decomposing according to the number of branch-outs,
\[
Z_p-1=Z^{(1)}+Z^{(2)}+Z^{(\ge3)}.
\]

We first compute \(Z^{(1)}\). Condition on the first ring occurring at
time \(a+u\), where \(a\in\mathbb Z_{\ge0}\) and \(u\in(0,1)\).
The density of this event, together with failed stop coins before the
ring and a non-oblivious mark at the ring, is
\[
(1-p_{\mathrm{stop}})^a e^{-(a+u)}\alpha\,du.
\]
After the branch-out there are \(2d\) children. For \(j\ge1\), set
\[
    B_0(u):=0,
    \qquad
    B_j(u):=
    \bigl(1-e^{-(j-u)}\bigr)^{2d-1}.
\]
The probability that \(a+j\) is the first integer checkpoint with
exactly one surviving child is
\[
    2d\,e^{-(j-u)}
    \bigl(B_j(u)-B_{j-1}(u)\bigr).
\]
Indeed, choose the surviving child, require it to have no ring by time
\(j-u\), and require all other \(2d-1\) children to have rung by time
\(j-u\), but not all by the preceding checkpoint.

For the sausage to have exactly one branch-out, these \(2d-1\) rings
must all be oblivious, contributing the factor
\((1-\alpha)^{2d-1}\). Therefore
\begin{align*}
&\mathbb E_\mu\!\left[
e^{\ell(x_1)}\mathbf 1_{\{b(x_1)=1\}}
\right]\\
&\quad=
\alpha(1-\alpha)^{2d-1}
\sum_{a\ge0}(1-p_{\mathrm{stop}})^a
\int_0^1
\sum_{j\ge1}
e^{a+j}e^{-(a+u)}
2d\,e^{-(j-u)}
\bigl(B_j(u)-B_{j-1}(u)\bigr)\,du\\
&\quad=
2d\alpha(1-\alpha)^{2d-1}
\sum_{a\ge0}(1-p_{\mathrm{stop}})^a
\int_0^1
\sum_{j\ge1}
\bigl(B_j(u)-B_{j-1}(u)\bigr)\,du.
\end{align*}
Since \(B_j(u)\to1\), the sum over \(j\) telescopes to \(1\).
Consequently,
\[
\mathbb E_\mu\!\left[
e^{\ell(x_1)}\mathbf 1_{\{b(x_1)=1\}}
\right]
=
\frac{2d}{p_{\mathrm{stop}}}
\alpha(1-\alpha)^{2d-1}.
\]
Thus
\[
Z^{(1)}
=
p_1\,
\mathbb E_\mu\!\left[
e^{\ell(x_1)}\mathbf 1_{\{b(x_1)=1\}}
\right]
=
\frac{2d\,p_1}{p_{\mathrm{stop}}}
\alpha(1-\alpha)^{2d-1}
=
\frac{\alpha}{p_{\mathrm{stop}}}
+O_{d,q}(\alpha^2),
\]
where we used $p_1=\frac1{2d}+O_{d,q}(\alpha)$ from Lemma~\ref{lem:canonical-sausages} (2).

Finally, Lemma~\ref{lem:canonical-sausages} and
\eqref{eq:mu-joint} give
\[
0\le Z^{(2)}
\le
\sum_{\ell\ge1}e^\ell
\mu\bigl(\ell(x_1)=\ell,\ b(x_1)\ge2\bigr)
=
O_{d,q}(\alpha^2),
\]
and
\[
\bigl|Z^{(\ge3)}\bigr|
\le
\sum_{\ell\ge1}e^\ell
\mu\bigl(\ell(x_1)=\ell,\ b(x_1)\ge3\bigr)
=
O_{d,q}(\alpha^3).
\]
Therefore
\[
Z_p
=
1+\frac{\alpha}{p_{\mathrm{stop}}}
+O_{d,q}(\alpha^2)
=
1+2\alpha+O_{d,q}(\alpha^2),
\]
which proves \eqref{eq:Zp}. In particular, after decreasing
\(\beta_0(d,q)\), one has \(1<Z_p\le2\).
\end{proof}

For each \(\ell\ge1\), define the signed displacement kernel
\begin{equation*}
\varphi_\ell(z):=\frac{\sum_{x:\ \ell(x)=\ell}\mu(x)\,p(x)\,\mathbf 1\{y(x)=z\}}
{\sum_{x:\ \ell(x)=\ell}\mu(x)\,p(x)}\,,
\qquad z\in\mathbb Z^d.
\end{equation*}
By \eqref{eq:w-ell}, the denominator is positive. Also $\sum_{z\in\mathbb{Z}^d}\varphi_\ell(z)=1$.

Define the length-biased law
\[
\tilde{\mathbb P}(L=\ell):=\frac{w_\ell}{Z_p},
\qquad \ell\ge1.
\]
By \eqref{eq:w-ell} and \(Z_p\le2\),
\begin{equation}\label{eq:L-tilde-bound}
\frac{p_{\mathrm{stop}}}{4}(1-p_{\mathrm{stop}})^{\ell-1}
\le
\tilde{\mathbb P}(L=\ell)
\le
\frac{3p_{\mathrm{stop}}}{2}(1-p_{\mathrm{stop}})^{\ell-1}.
\end{equation}
Let \(L_1,L_2,\dots\) be i.i.d.\ with law \(\tilde{\mathbb P}\), and set \(S_r:=L_1+\cdots+L_r\). Write $\tilde{\mathbb E}$ for expectation under $\tilde{\mathbb P}$.
Given a length profile $\vs=(s_1,\dots,s_t)$ where $s_i:=\#\{k:L_k=i\}$ counts the number of length $i$ sausages,
subject to
$\sum_{i=1}^t s_i=r$ and $\sum_{i=1}^t i\,s_i=t$,
we call \(\vs\) an \emph{admissible profile} of length \(r\) and depth \(t\).
Define the mixed convolution kernel and its conditional expectation
\[
\Phi_{\vs}
:=\varphi_1^{\conv s_1}\conv\cdots\conv \varphi_t^{\conv s_t},
\qquad
K(z;r,t):=\tilde{\mathbb E}\!\left[\Phi_{\mathbf s}(z)\ \big|\ S_r=t\right].
\]
Then regrouping \eqref{eq:IK-def} by lengths gives
\begin{equation}\label{eq:IK-length-decomp}
\IK(z,t)=e^{-t}\sum_{r\ge1}Z_p^r\,\tilde{\mathbb P}(S_r=t)\,K(z;r,t).
\end{equation}

We now identify the exponential tilt that absorbs the factor \(Z_p^r\)
in \eqref{eq:IK-length-decomp}.

\begin{lem}[Renewal tilt and decay rate]\label{lem:kappa bound}
There exists a unique \(\kappa\in(-1,0)\) such that
\begin{equation}\label{eq:kappa}
    \tilde{\mathbb E}[e^{\kappa L_1}]=Z_p^{-1}.
\end{equation}
Set
\[
\mathfrak r:=1+\kappa\in(0,1).
\]
Then
\[
\kappa=-\alpha+O_{d,q}(\alpha^2),
\qquad
\mathfrak r=1-\alpha+O_{d,q}(\alpha^2).
\]
Consequently, after decreasing \(\beta_0(d,q)\),
\[
0<-\kappa<2d\alpha,
\qquad
1-2d\alpha<\mathfrak r<1.
\]
\end{lem}

\begin{proof}
Set
\[
f(\theta):=\tilde{\mathbb E}[e^{\theta L_1}],
\qquad
\theta<-\log(1-p_{\mathrm{stop}}).
\]
By \eqref{eq:L-tilde-bound}, \(L_1\ge1\) and \(L_1\) has exponential tails, so \(f\) is finite and \(C^1\) on
\((-\infty,-\log(1-p_{\mathrm{stop}}))\). Since \(L_1\ge1\) a.s., \(f\) is strictly increasing. Also
\(f(0)=1\), and \(\lim_{\theta\to-\infty}f(\theta)=0\).

By \eqref{eq:Zp} and Lemma~\ref{lem:Zp-bound}, $Z_p^{-1}\in[1/2,1)$. On the other hand,
\[
f(-1)=\tilde{\mathbb E}[e^{-L_1}]\le e^{-1}<\frac12\le Z_p^{-1}.
\]
Hence there exists a unique \(\kappa\in(-1,0)\) such that \(f(\kappa)=Z_p^{-1}\), proving \eqref{eq:kappa}.

For the asymptotic estimate, define
\[
G_\alpha(\theta):=\sum_{\ell\ge1}w_\ell e^{\theta\ell}=Z_p f(\theta),\qquad
G_\alpha(\kappa)=1,\qquad
G_\alpha(0)=Z_p.
\]
The non-branching weights are
$w_\ell^{(0)}
:=
p_{\mathrm{stop}}(1-p_{\mathrm{stop}})^{\ell-1}$.
Consequently,
\[
G_\alpha'(0)
=
\sum_{\ell\ge1}\ell w_\ell
=
\sum_{\ell\ge1}
\ell p_{\mathrm{stop}}
(1-p_{\mathrm{stop}})^{\ell-1}
+O_{d,q}(\alpha)
=
\frac1{p_{\mathrm{stop}}}+O_{d,q}(\alpha).
\]
Also, \eqref{eq:w-ell} implies \(G_\alpha''\) is uniformly bounded on \([-1,0]\). 
By the mean-value theorem, for some \(\xi\in(\kappa,0)\),
\begin{equation}\label{eq:kappa-mvt}
    Z_p-1
    =
    G_\alpha(0)-G_\alpha(\kappa)
    =
    (-\kappa)G_\alpha'(\xi).
\end{equation}
Since \(w_1\ge p_{\mathrm{stop}}/2\) by \eqref{eq:w-ell}, $G_\alpha'(\xi)
    \ge
    e^{-1}w_1
    \ge
    \frac{p_{\mathrm{stop}}}{2e}$.
Together with \(Z_p-1=O_{d,q}(\alpha)\), this first gives
\[
    -\kappa=O_{d,q}(\alpha).
\]
Hence \(|\xi|=O_{d,q}(\alpha)\), and the uniform bound on
\(G_\alpha''\) yields
\[
    G_\alpha'(\xi)
    =
    G_\alpha'(0)+O_{d,q}(|\xi|)
    =
    \frac1{p_{\mathrm{stop}}}
    +O_{d,q}(\alpha).
\]
Substituting this and \eqref{eq:Zp} into
\eqref{eq:kappa-mvt}, we obtain
\[
\begin{aligned}
-\kappa
=
\frac{Z_p-1}{G_\alpha'(\xi)}
=
\frac{\alpha/p_{\mathrm{stop}}+O_{d,q}(\alpha^2)}
     {1/p_{\mathrm{stop}}+O_{d,q}(\alpha)}=
\alpha+O_{d,q}(\alpha^2).
\end{aligned}
\]
\end{proof}

Let \(\tilde{\mathbb P}_\kappa\) denote the \(\kappa\)-tilt of \(\tilde{\mathbb P}\):
\[
\tilde{\mathbb P}_\kappa(L=\ell)
:=
\frac{e^{\kappa\ell}\tilde{\mathbb P}(L=\ell)}
{\tilde{\mathbb E}[e^{\kappa L_1}]},
\qquad \ell\ge1.
\]
Then \eqref{eq:IK-length-decomp} becomes

\begin{equation}\label{eq:IK-tilt}
\IK(z,t)=e^{-(1+\kappa)t}\sum_{r\ge 1}\tilde{\mathbb P}_\kappa(S_r=t)\,K(z;r,t),
\qquad S_r:=\sum_{i=1}^r L_i,\qquad t\ge1.
\end{equation}

\begin{lem}\label{lem:LDP}
Let $\bar\ell_\kappa:=\tilde{\mathbb E}_\kappa[L_1]\in(0,\infty)$.
Fix $\delta\in(0,\bar\ell_\kappa/2)$ and define
\begin{equation*} 
a_-:=\frac{1}{\bar\ell_\kappa+\delta},
\qquad
a_+:=\frac{1}{\bar\ell_\kappa-\delta},
\qquad
\mathrm{SP}_t:=\big\{r\in\mathbb N:\ a_-t\le r\le a_+t\big\}.
\end{equation*}
Then there exists \(c_{\mathrm{ld}}(\delta)>0\) such that, for all large \(t\),
\[
    \sum_{r\notin\mathrm{SP}_t}
    \widetilde{\mathbb P}_\kappa(S_r=t)
    \le e^{-c_{\mathrm{ld}}t}.
\]
Moreover,
\[
\sum_{r\ge1}\tilde{\mathbb P}_\kappa(S_r=t)\longrightarrow \frac{1}{\bar\ell_\kappa}
\qquad\text{as }t\to\infty.
\]
In particular, for every fixed $\delta\in(0,\bar\ell_\kappa/2)$ there exists $t_{\mathrm{ren}}(\delta)\in\mathbb N$ such that for all $t\ge t_{\mathrm{ren}}(\delta)$,
\begin{equation}\label{eq:renewal-window}
\sum_{r\in\mathrm{SP}_t}\tilde{\mathbb P}_\kappa(S_r=t)\ge \frac{1}{3\bar\ell_\kappa}.
\end{equation}
\end{lem}
\begin{proof}
Since \(\tilde{\mathbb P}_\kappa\) has exponential moments in a neighborhood of \(0\), its log-moment generating function
\[
\Lambda_\kappa(\theta):=\log \tilde{\mathbb E}_\kappa[e^{\theta L_1}]
\]
is finite and \(C^2\) near \(0\), with \(\Lambda_\kappa'(0)=\bar\ell_\kappa\).

Choose \(\theta_+>0\) so small that \(\Lambda_\kappa(\theta_+)\le \theta_+(\bar\ell_\kappa+\delta/2)\). Then for \(r<a_-t\),
\[
\tilde{\mathbb P}_\kappa(S_r=t)\le \tilde{\mathbb P}_\kappa(S_r\ge t)
\le
\exp\!\Big(-\theta_+ t+r\Lambda_\kappa(\theta_+)\Big)
\le
\exp\!\Big(-\frac{\theta_+\delta}{2(\bar\ell_\kappa+\delta)}\,t\Big).
\]
Similarly, choose \(\theta_->0\) so small that \(\Lambda_\kappa(-\theta_-)\le -\theta_-(\bar\ell_\kappa-\delta/2)\). Then for \(r>a_+t\),
\[
\tilde{\mathbb P}_\kappa(S_r=t)\le \tilde{\mathbb P}_\kappa(S_r\le t)
\le
\exp\!\Big(\theta_- t+r\Lambda_\kappa(-\theta_-)\Big)
\le
\exp\!\Big(-\frac{\theta_-\delta}{2(\bar\ell_\kappa-\delta)}\,t\Big).
\]
This proves the first claim.

Also \(\tilde{\mathbb P}_\kappa(L_1=1)>0\) by \eqref{eq:L-tilde-bound}, so the arithmetic renewal theorem applies and yields
\[
\sum_{r\ge0}\tilde{\mathbb P}_\kappa(S_r=t)\longrightarrow \frac{1}{\bar\ell_\kappa}.
\]
Since \(t\ge1\), the \(r=0\) term vanishes, proving the convergence. Combining this with the large-deviation estimate yields \eqref{eq:renewal-window}.
\end{proof}

Remark that
\[
\bar\ell_\kappa
=
Z_p\tilde{\mathbb E}[L_1e^{\kappa L_1}]
\le
Z_p\tilde{\mathbb E}[L_1]
\le
2\sum_{\ell\ge1}\ell\cdot \frac{3p_{\mathrm{stop}}}{2}(1-p_{\mathrm{stop}})^{\ell-1}
=
\frac{3}{p_{\mathrm{stop}}},
\]
using \eqref{eq:L-tilde-bound} and \(Z_p\le2\). 
Consequently, by the renewal convergence in Lemma~\ref{lem:LDP}, there exists
\(t_{\mathrm{ren}}<\infty\) such that, for all integers \(t\ge t_{\mathrm{ren}}\),
\begin{equation}\label{eq:renewal-lb-ps}
\sum_{r\ge1}\tilde{\mathbb P}_\kappa(S_r=t)
\ge \frac{1}{2\bar\ell_\kappa}
\ge \frac{p_{\mathrm{stop}}}{6}.
\end{equation}

\subsection{Fourier input and the \texorpdfstring{\(\ell^1\)}{l1} bound}
\label{subsec:fourier-input}

The required Fourier estimates for the individual signed displacement
kernels are established in Lemma~\ref{lem:varphi-input} of the
Appendix~\ref{sec:appendix}. That lemma gives uniform moment
and spectral-gap estimates, together with the expansion
\[
\log \widehat\varphi_\ell(\xi)
=
-\frac{\sigma_\ell}{2}|\xi|^2
+\mathfrak a_\ell|\xi|^4
+\mathfrak b_\ell\sum_{j=1}^d\xi_j^4
+O(|\xi|^6),
\qquad \xi\to0,
\]
uniformly in \(\ell\), where there exist constants
$\sigma_\pm>0$, $A_4>0$ such that 
\[
0<\sigma_-\le \sigma_\ell\le\sigma_+,
\qquad
|\mathfrak a_\ell|+|\mathfrak b_\ell|\le A_4.
\]
Although $\varphi_\ell$ may be signed, it provides the contractive Fourier control
needed for local limit expansions and pointwise bounds on convolutions to prove the Proposition below.

\smallskip
For a length profile $\vs=(s_1,\ldots,s_t)$ with $\sum s_i=r$ and $\sum i s_i=t$, define the effective coefficients
\begin{equation*}
\sigma_\vs:=\frac1r\sum_{i=1}^t s_i\sigma_i,\qquad
\mathfrak{a}_\vs:=\frac1r\sum_{i=1}^t s_i\mathfrak{a}_i,\qquad
\mathfrak{b}_\vs:=\frac1r\sum_{i=1}^t s_i\mathfrak{b}_i.
\end{equation*}
Then $\sigma_-\le \sigma_{\vs}\le \sigma_+$ and $\abs{\mathfrak{a}_{\vs}}+\abs{\mathfrak{b}_{\vs}}\le A_4$. Let
\begin{equation*}
P_{\vs}(\xi):=\frac{\sigma_{\vs}}{2}\abs{\xi}^2,
\qquad
Q_{4,\vs}(\xi)
:=\mathfrak{a}_\vs|\xi|^4+\mathfrak{b}_\vs\sum_{j=1}^d \xi_j^4,
\end{equation*}
and let the corresponding Gaussian kernel be
\begin{equation*}
H_{\vs}(x)
:=\frac{1}{(2\pi r\sigma_\vs)^{d/2}}
\exp\!\Big(-\frac{|x|^2}{2r\sigma_\vs}\Big),
\qquad x\in\mathbb{Z}^d.
\end{equation*}

For a polynomial $Q(\xi)=\sum_\alpha c_\alpha \xi^\alpha$, define $Q(i\nabla_x):=\sum_\alpha c_\alpha i^{\abs{\alpha}}\partial_x^\alpha$.

\begin{prop}\label{prop:Fourier-six-exp}
There exists \(C_{\mathrm F}<\infty\) such that, for every admissible profile \(\vs\) and every \(x\in\mathbb Z^d\),
\begin{equation}\label{eq:main-expansion}
\Phi_\vs(x)
=
H_\vs(x)
+
rQ_{4,\vs}(i\nabla_x)H_\vs(x)
+
\mathrm{Err}_\vs(x),
\end{equation}
where
\begin{equation}\label{eq:ErrBounds}
\|\mathrm{Err}_\vs\|_{\ell^1(\mathbb Z^d)}
\le C_{\mathrm F}r^{-2},
\qquad
\|\mathrm{Err}_\vs\|_{\ell^\infty(\mathbb Z^d)}
\le C_{\mathrm F}r^{-2-d/2}.
\end{equation}
\end{prop}
\noindent
The proof is given in \hyperref[sec:appendix]{the Appendix}.

\begin{lemma}\label{lem:derivative-bound}
There exists \(C_{\mathrm{der}}<\infty\), depending only on \(d,\sigma_\pm,A_4\), such that for every admissible profile \(\mathbf s\),
\[
\big|rQ_{4,\mathbf s}(i\nabla_x)H_{\mathbf s}(x)\big|
\le
\frac{C_{\mathrm{der}}}{r}\left(1+\frac{|x|^4}{r^2}\right)H_{\mathbf s}(x)
\qquad\text{for all }x\in\mathbb Z^d.
\]
\end{lemma}

\begin{proof}
Every fourth-order derivative of the Gaussian kernel \(H_{\mathbf s}\) is a polynomial of degree \(4\) in
\(x/\sqrt r\), multiplied by \(H_{\mathbf s}\), with coefficients bounded uniformly for
\(\sigma_{\mathbf s}\in[\sigma_-,\sigma_+]\). The coefficients of \(Q_{4,\mathbf s}\) are bounded by \(A_4\), and the prefactor \(r\) cancels one power of \(r^{-1}\) from the four derivatives. This yields the stated bound.
\end{proof}

\begin{lemma}[Uniform \(\ell^1\) bound on \(\Phi_{\mathbf s}\)]\label{lem:Phi-l1}
There exists \(C_{\ell^1}<\infty\) such that for every admissible profile \(\mathbf s\),
\[
\|\Phi_{\mathbf s}\|_{\ell^1(\mathbb Z^d)}\le C_{\ell^1}.
\]
Consequently, for every \(r\) and \(t\),
\[
\|K(\cdot;r,t)\|_{\ell^1(\mathbb Z^d)}\le C_{\ell^1}.
\]
\end{lemma}

\begin{proof}
From Proposition~\ref{prop:Fourier-six-exp},
\[
\|\Phi_{\mathbf s}\|_1
\le
\|H_{\mathbf s}\|_1+\|rQ_{4,\mathbf s}(i\nabla)H_{\mathbf s}\|_1+\|\mathrm{Err}_{\mathbf s}\|_1.
\]
The error term is bounded by \(C_{\mathrm F}r^{-2}\). By Lemma~\ref{lem:derivative-bound},
\[
\|rQ_{4,\mathbf s}(i\nabla)H_{\mathbf s}\|_1
\le
\frac{C_{\mathrm{der}}}{r}
\sum_{x\in\mathbb Z^d}\left(1+\frac{|x|^4}{r^2}\right)H_{\mathbf s}(x).
\]
The right-hand side is uniformly bounded, since it is a Gaussian fourth moment and
\(\sigma_{\mathbf s}\in[\sigma_-,\sigma_+]\).
Also $\|H_{\vs}\|_1$ is uniformly bounded.
Thus $\|\Phi_{\vs}\|_1\le C_{\ell^1}$ for some constant depending only on $(d,\sigma_\pm,A_4,C_{\mathrm F},C_{\mathrm{der}})$.

Finally,
\[
\|K(\cdot;r,t)\|_1
=
\sum_x\left|\tilde{\mathbb E}\!\left[\Phi_{\mathbf s}(x)\,\middle|\,S_r=t\right]\right|
\le
\tilde{\mathbb E}\!\left[\|\Phi_{\mathbf s}\|_1\,\middle|\,S_r=t\right]
\le
C_{\ell^1}.
\]
\end{proof}

\begin{lemma}[Uniform \(\ell^1\) bound for \(\IK\)]\label{lem:IK-l1}
There exists \(C_{\IK,1}<\infty\) such that for all integers \(t\ge0\),
\begin{equation*}
\sum_{z\in\mathbb Z^d}|\IK(z,t)|\le C_{\IK,1}e^{-(1+\kappa)t}.
\end{equation*}
\end{lemma}

\begin{proof}
The case \(t=0\) is immediate. For \(t\ge1\), \eqref{eq:IK-tilt} and Lemma~\ref{lem:Phi-l1} give
\[
\sum_{z\in\mathbb Z^d}|\IK(z,t)|
\le
e^{-(1+\kappa)t}\sum_{r\ge1}\tilde{\mathbb P}_\kappa(S_r=t)\|K(\cdot;r,t)\|_1
\le
C_{\ell^1}e^{-(1+\kappa)t}\sum_{r\ge1}\tilde{\mathbb P}_\kappa(S_r=t).
\]
By Lemma~\ref{lem:LDP}, the renewal mass \(\sum_{r\ge1}\tilde{\mathbb P}_\kappa(S_r=t)\) converges, hence is bounded to prove the claim.
\end{proof}

\smallskip

\subsection{Magnetization bounds}
\label{subsec:mag-analysis}

Let \(\sigma_t\) denote the Potts configuration at real time \(t\). For an initial condition
\(\sigma_0\in[q]^{\mathbb Z^d}\) and a color \(c\in[q]\), define the one-site bias on $\mathbb{Z}^d$, which correspond to \eqref{eq:magnetization-def}:
\[
    \mathfrak m_t(\sigma_0,c)
    :=
    \mathbb P_{\sigma_0}(\sigma_t(o)=c)-\frac1q.
\]

To account for the boundary at real time \(0\), we isolate the bottom-most, possibly partial,
sausage. Let \(x_0\) denote the full sausage whose slab contains time \(0\), and let
\(s\in[0,\ell(x_0))\) be the cut depth inside \(x_0\), measured in backward time from the top of
\(x_0\). Define
\[
\nu_{x_0,\sigma_0,c}(z,s)
:=
\mathbb P\bigl(\sigma_s(z)=c \,\big|\, \sigma_0,\ \text{terminal sausage geometry is }x_0,\ \text{cut depth is }s\bigr),
\]
(see Figure~\ref{fig:nu_definition}), and the centered, rescaled bias
\begin{equation*}
g_{x_0,\sigma_0,c}(z,s):=\frac{q}{q-1}\Big(\nu_{x_0,\sigma_0,c}(z,s)-\frac1q\Big)\in\Big[-\frac1{q-1},\,1\Big].
\end{equation*}
The quantity \(g_{x_0,\sigma_0,c}(z,s)\) is the interface bias fed from the truncated
terminal sausage at depth \(s\) into the upper influence kernel in \eqref{eq:mag-sausage}; the only
general fact used below is that \(|g_{x_0,\sigma_0,c}(z,s)|\le1\) for all \(z,s\). For later
use, let
\[
\tilde b_s(x_0):=\#\{\text{branch-out events in the restriction of }x_0\text{ to depth }[0,s]\}.
\]
For example, if \(\tilde b_s(x_0)=0\), then the truncated terminal sausage contains no non-oblivious update and
therefore acts as the identity at depth \(s\): for monochromatic \(\mathfrak c\equiv c\), $g_{x_0,\mathfrak c,c}(z,s)=1$ for every $z\in\mathbb{Z}^d$.

%===============================================
% =========Terminal sausage
% ==============================================
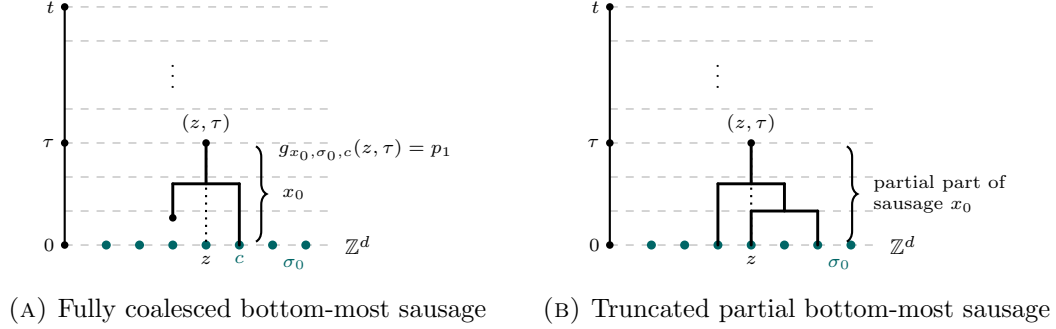
\begin{figure}[htbp]
\centering
% LEFT PANEL: Full terminal sausage
% ==============================================
\begin{subfigure}[b]{0.46\textwidth}
\centering
\begin{tikzpicture}[
    x=0.55cm, y=0.45cm,
    every node/.style={font=\tiny},
    grid/.style={gray!50, dashed, line width=0.5pt},
    hist/.style={black, line width=1.1pt, line cap=round, line join=round},
    death/.style={circle, fill=black, inner sep=1pt},
    nodept/.style={circle, fill=black, inner sep=1pt},
    initpt/.style={circle, fill=teal!80!black, inner sep=1.2pt}
]
  % Grid
  \foreach \y in {0, 1, 2, 3, 4, 6, 7} {
    \draw[grid] (-0.5, \y) -- (6, \y);
  }
  \node at (2.1, 5.2) {$\vdots$};

  % Y-axis
  \draw[thick] (-0.5, 0) -- (-0.5, 7);
  \node[left] at (-0.5, 0) {$0$};
  \node[left] at (-0.5, 3) {$\tau$};
  \node[left] at (-0.5, 7) {$t$};
  \node[nodept] at (-0.5, 0) {};
  \node[nodept] at (-0.5, 3) {};
  \node[nodept] at (-0.5, 7) {};

  % Spatial Axis Label (Z^d)
  \node[right, font=\scriptsize] at (6, 0) {$\mathbb{Z}^d$};

  % Init Config (sigma_0)
  \foreach \x in {0.5, 1.3, 2.1, 2.9, 3.7, 4.5, 5.3} {
    \node[initpt] at (\x, 0) {};
  }
  \node[below right, text=teal!80!black] at (4.5, -0.1) {$\sigma_0$};

  \node[below, text=teal!80!black] at (3.7,0) {$c$};

  % Root Point (z, \tau)
  \coordinate (ZTau) at (2.9, 3);
  \node[above] at (ZTau) {$(z,\tau)$};
  \node[nodept] at (ZTau) {};
  \draw[thick, dotted] (ZTau) -- (2.9, 0) node[below] {$z$};

  % Branching Process
  \draw[hist] (ZTau) -- (2.9, 1.8);
  \draw[hist] (2.1, 1.8) -- (3.7, 1.8);

  % Left leg (dies)
  \draw[hist] (2.1, 1.8) -- (2.1, 0.8); 
  \node[death] at (2.1, 0.8) {};

  % Right leg (survives to 0)
  \draw[hist] (3.7, 1.8) -- (3.7, 0);

  % Brace for x_0
  \draw[decorate, decoration={brace, mirror, amplitude=4pt}, thick]
    (4.1, 0.1) -- (4.1, 2.9) node[midway, right=6pt] {$x_0$};

  % Equation Text (Updated to p_1 + 1/q)
  \node[right] at (4.4, 2.8) {$g_{x_0,\sigma_0,c}(z,\tau)= p_1$};

\end{tikzpicture}
\caption{Fully coalesced bottom-most sausage}
\end{subfigure}%
\hspace{-0.02\textwidth}%
% RIGHT PANEL: Partial terminal sausage
% =============================================
\begin{subfigure}[b]{0.46\textwidth}
\centering
\begin{tikzpicture}[
    x=0.55cm, y=0.45cm,
    every node/.style={font=\tiny},
    grid/.style={gray!50, dashed, line width=0.5pt},
    hist/.style={black, line width=1.1pt, line cap=round, line join=round},
    death/.style={circle, fill=black, inner sep=1pt},
    nodept/.style={circle, fill=black, inner sep=1pt},
    initpt/.style={circle, fill=teal!80!black, inner sep=1.2pt}
]
  % Grid
  \foreach \y in {0, 1, 2, 3, 4, 6, 7} {
    \draw[grid] (-0.5, \y) -- (6, \y);
  }
  \node at (2.1, 5.2) {$\vdots$};

  % Y-axis
  \draw[thick] (-0.5, 0) -- (-0.5, 7);
  \node[left] at (-0.5, 0) {$0$};
  \node[left] at (-0.5, 3) {$\tau$};
  \node[left] at (-0.5, 7) {$t$};
  \node[nodept] at (-0.5, 0) {};
  \node[nodept] at (-0.5, 3) {};
  \node[nodept] at (-0.5, 7) {};

  % Spatial Axis Label (Z^d)
  \node[right, font=\scriptsize] at (6, 0) {$\mathbb{Z}^d$};

  % Init Config (sigma_0)
  \foreach \x in {0.5, 1.3, 2.1, 2.9, 3.7, 4.5, 5.3} {
    \node[initpt] at (\x, 0) {};
  }
  \node[below right, text=teal!80!black] at (4.5, -0.1) {$\sigma_0$};

  % Root Point (z, \tau)
  \coordinate (ZTau) at (2.9, 3);
  \node[above] at (ZTau) {$(z,\tau)$};
  \node[nodept] at (ZTau) {};
  \draw[thick, dotted] (ZTau) -- (2.9, 0) node[below] {$z$};

  % Branching Process
  \draw[hist] (ZTau) -- (2.9, 1.8);
  \draw[hist] (2.1, 1.8) -- (3.7, 1.8);

  % Left outer leg
  \draw[hist] (2.1, 1.8) -- (2.1, 0);

  % Right branch splits again
  \draw[hist] (3.7, 1.8) -- (3.7, 1.0);
  \draw[hist] (2.9, 1.0) -- (4.5, 1.0);

  % Inner left leg
  \draw[hist] (2.9, 1.0) -- (2.9, 0);
  % Outer right leg
  \draw[hist] (4.5, 1.0) -- (4.5, 0);

  % Brace for partial x_0
  \draw[decorate, decoration={brace, mirror, amplitude=4pt}, thick]
    (5.2, 0.1) -- (5.2, 2.9) node[midway, right=6pt, align=left] {partial part of\\sausage $x_0$};

\end{tikzpicture}
\caption{Truncated partial bottom-most sausage}
\end{subfigure}

\caption{Effect of the initial configuration $\sigma_0$ through the bottom-most sausage $x_0$ (the slab intersecting time $0$).
\textbf{Left:} The history has coalesced to a singleton before depth $\tau$, so the terminal contribution reduces to the usual one-lineage bias (here $g=p_1$).
\textbf{Right:} The boundary at time $0$ truncates the slab before coalescence, producing a partial sausage in which several sites of $\sigma_0$ can influence $(z,\tau)$.
}
\label{fig:nu_definition}
\end{figure}
%=========Terminal sausage end============
% ==============================================

For \(t\ge0\), write
\[
n_t:=\lfloor t\rfloor,
\qquad
\theta_t:=t-n_t\in[0,1).
\]

By the strong Markov property at the regeneration times, the complete
sausages above the terminal partial sausage form an i.i.d.\ concatenation.
By color symmetry, the transition kernel associated with a complete sausage
\(x\) preserves the uniform law and acts as multiplication by \(p(x)\) on 
centered color functions, while the terminal partial sausage contributes
\(g_{x_0,\sigma_0,c}\). Since \(g\) is normalized by the factor \(q/(q-1)\),
converting back to the unscaled bias \(\mathfrak m_t\) contributes the factor
\((q-1)/q\). Summing over the total length and displacement of the complete
sausages therefore gives, for every \(t\ge0\),
\begin{equation}\label{eq:mag-sausage}
    \mathfrak m_t(\sigma_0,c)
    =
    \frac{q-1}{q}
    \sum_{\tau=0}^{n_t}
    \sum_{x_0\in\Xi}
    \mu(x_0)\mathbf 1_{\{\ell(x_0)>\theta_t+\tau\}}
    \sum_{z\in\mathbb Z^d}
    \IK(z,n_t-\tau)
    g_{x_0,\sigma_0,c}(z,\theta_t+\tau).
\end{equation}
The spatial sum is absolutely convergent by
Lemma~\ref{lem:IK-l1}.

\begin{lemma}[Terminal-sausage tails]\label{lem:terminal-sausage-tails}
For every integer \(\tau\ge0\), every \(\theta\in[0,1)\), and every \(j\ge1\),
\[
\mu\bigl(\ell(x_0)>\tau+\theta,\ \tilde b_{\tau+\theta}(x_0)=0\bigr)
\ge c_0e^{-\eta\tau},
\]
and
\[
\mu\bigl(\ell(x_0)>\tau+\theta,\ \tilde b_{\tau+\theta}(x_0)\ge j\bigr)
\le
\frac{C_{\mathrm{jnt}}e^{-\eta}}{1-e^{-\eta}}\,
\vartheta^{\,j}e^{-\eta\tau},
\]
with $c_0:=
    \frac{c_{\mathrm{len}}e^{-\eta}}{1-e^{-\eta}}.$ In particular,
\[
\mu\bigl(\ell(x_0)>\tau+\theta,\ \tilde b_{\tau+\theta}(x_0)\ge1\bigr)
\le
\frac{C_{\mathrm{jnt}}C_{\mathrm{bp}}e^{-\eta}}{1-e^{-\eta}}\,
\alpha\,e^{-\eta\tau}.
\]
\end{lemma}

\begin{proof}
Since \(\ell(x_0)\in\mathbb Z_{\ge1}\), the event \(\{\ell(x_0)>\tau+\theta\}\) is exactly
\(\{\ell(x_0)\ge\tau+1\}\). For \(\tilde b_{\tau+\theta}(x_0)=0\), the proof of
Proposition~\ref{prop:sausagetail} shows that
\[
\mu\bigl(\ell(x_0)=\ell,\ \tilde b_{\tau+\theta}(x_0)=0\bigr)=c_{\mathrm{len}}e^{-\eta\ell}
\qquad(\ell\ge \tau+1).
\]
Summing over \(\ell\ge\tau+1\) yields the first identity. Likewise, summing the joint estimate
\eqref{eq:mu-joint} over \(\ell\ge\tau+1\) gives the second display.
\end{proof}

\begin{prop}[Monochromatic lower bound]\label{prop:all-one-magnetization-bound}
After decreasing \(\beta_0(d,q)\), there exist constants
\(c_{\mathrm{all}}>0\)  and \(t_{\mathrm{all}}<\infty\),
depending only on \(d,q,\beta\), such that, for the monochromatic initial configuration \(\mathfrak c\equiv c\),
\[
    \mathfrak m_t(\mathfrak c,c)
    \ge
    c_{\mathrm{all}}e^{-(1+\kappa)t}
    \qquad\text{for all }t\ge t_{\mathrm{all}}.
\]
\end{prop}

\begin{proof}
Fix \(t\ge0\), and for \(0\le \tau\le n_t\) write
\[
s_\tau:=\theta_t+\tau.
\]
We split \eqref{eq:mag-sausage} according to whether the truncated terminal sausage up to depth
\(s_\tau\) branches:
\[
\mathfrak m_t(\mathfrak c,c)
=
\frac{q-1}{q}\bigl(\Sigma_0(t)+\Sigma_{\ge1}(t)\bigr).
\]

If \(\tilde b_{s_\tau}(x_0)=0\), then the truncated terminal sausage is vertical and acts as the
identity at depth \(s_\tau\). Since the initial condition is monochromatic, we therefore have
\[
g_{x_0,\mathfrak c,c}(z,s_\tau)=1
\qquad\text{for every }z\in\mathbb Z^d.
\]
Hence
\[
\Sigma_0(t)
=
\sum_{\tau=0}^{n_t}\mu(\ell(x_0)> s_\tau,\ \tilde b_{s_\tau}(x_0)=0)\sum_{z\in\mathbb Z^d}\IK(z,n_t-\tau).
\]
By Lemma~\ref{lem:terminal-sausage-tails},
\[
\mu(\ell(x_0)> s_\tau,\ \tilde b_{s_\tau}(x_0)=0)=c_0e^{-\eta\tau}.
\]
For \(s\ge1\), \(\sum_z K(z;r,s)=1\), so by \eqref{eq:IK-tilt},
\[
\sum_{z\in\mathbb Z^d}\IK(z,s)
=
e^{-(1+\kappa)s}\sum_{r\ge1}\tilde{\mathbb P}_\kappa(S_r=s).
\]
Together with \(\sum_z\IK(z,0)=1\), this shows that every summand in
\(\Sigma_0(t)\) is nonnegative. Keeping only the \(\tau=0\) term and using
\eqref{eq:renewal-lb-ps}, for all \(n_t\ge t_{\mathrm{ren}}\),
\begin{equation}\label{eq:Sigma0-lb}
\Sigma_0(t)\ge c_0e^{-(1+\kappa)n_t}
\sum_{r\ge1}\tilde{\mathbb P}_\kappa(S_r=n_t)
\ge \frac{c_0p_{\mathrm{stop}}}{6}e^{-(1+\kappa)n_t}.
\end{equation}

For the remainder, use \(|g_{x_0,\mathfrak c,c}|\le1\) together with
Lemma~\ref{lem:terminal-sausage-tails}:
\begin{align*}
\Sigma_{\ge1}(t)
&\ge
-
\sum_{\tau=0}^{n_t}\mu(\ell(x_0)> s_\tau,\ \tilde b_{s_\tau}(x_0)\ge1)
\sum_{z\in\mathbb Z^d}|\IK(z,n_t-\tau)| \\
&\ge
-
\frac{C_{\mathrm{jnt}}C_{\mathrm{bp}}e^{-\eta}}{1-e^{-\eta}}\,
\alpha\sum_{\tau=0}^{n_t}e^{-\eta\tau}\sum_{z\in\mathbb Z^d}|\IK(z,n_t-\tau)|.
\end{align*}
Using Lemma~\ref{lem:IK-l1}, we obtain
\begin{align}\label{eq:Sigma-ge1-lb}
\Sigma_{\ge1}(t)
\ge
-
C_{\mathrm{neg}}\alpha\,e^{-(1+\kappa)n_t}.
\end{align}

Combining \eqref{eq:Sigma0-lb} and \eqref{eq:Sigma-ge1-lb},
\[
\mathfrak m_t(\mathfrak c,c)
\ge
\frac{q-1}{q}
e^{-(1+\kappa)n_t}
\left(\frac{c_0p_{\mathrm{stop}}}{6}-C_{\mathrm{neg}}\alpha\right).
\]
Decreasing \(\beta_0\) so that
\(\alpha\) is small enough, the negative term is at most half of the positive contribution.
Absorbing the factor \((q-1)/q\) into the constant, we obtain
\[
    \mathfrak m_t(\mathfrak c,c)
    \ge c e^{-(1+\kappa)n_t}
    \ge c_{\mathrm{all}}e^{-(1+\kappa)t}
\]
for all sufficiently large \(t\).
\end{proof}

\begin{prop}[Uniform exponential bound on one-site bias]\label{prop:magnetization-bound}
There exists \(C_{\mathrm m}<\infty\) such that, for every initial condition
\(\sigma_0\), every color \(c\), and every \(t\ge0\),
\[
    |\mathfrak m_t(\sigma_0,c)|
    \le
    C_{\mathrm m}e^{-(1+\kappa)t}.
\]
\end{prop}

\begin{proof}
Fix \(t\ge0\). By \eqref{eq:mag-sausage},
\[
|\mathfrak m_t(\sigma_0,c)|
\le
\sum_{\tau=0}^{n_t}\mu(\ell(x_0)\ge \tau+\theta_t)\sum_{z\in\mathbb Z^d}|\IK(z,n_t-\tau)|.
\]
Since \(\ell(x_0)\in\mathbb Z_{\ge1}\), the event \(\{\ell(x_0)>\tau+\theta_t\}\) is contained in
\(\{\ell(x_0)>\tau\}\). Therefore Proposition~\ref{prop:sausagetail} and Lemma~\ref{lem:IK-l1}
yield
\begin{align*}
|\mathfrak m_t(\sigma_0,c)|
&\le
\sum_{\tau=0}^{n_t}c_{\mathrm{tail}}e^{-\eta\tau}C_{\IK,1}e^{-(1+\kappa)(n_t-\tau)}\\
&=
c_{\mathrm{tail}}C_{\IK,1}e^{-(1+\kappa)n_t}
\sum_{\tau=0}^{n_t}e^{-(\eta-(1+\kappa))\tau}
\le
C\,e^{-(1+\kappa)n_t}.
\end{align*}
Since \(n_t\ge t-1\), this is at most \(C_{\mathrm m}e^{-(1+\kappa)t}\).
\end{proof}

\paragraph{Transfer to a finite torus.}
The two propositions above concern the i.i.d.\ graphical construction on
\(\mathbb Z^d\). To transfer them to the logarithmic time scale on
\(\Lambda_n\), we first record a spatial-range estimate for the non-coalescing
branching walk from Subsection~\ref{subsec:branch-length}. Let \(G\) be either
\(\mathbb Z^d\) or a torus \(\Lambda_n\), write \(d_G\) for its graph distance,
and let \(Z_u^x(y)\) be the number of particles at \(y\) at elapsed time \(u\)
when the walk starts from one particle at \(x\). Recall that its mean offspring
number is \(m=2d\alpha<1\).

\begin{lemma}[Deviation for the dominating branching walk]
\label{lem:brw-deviation}
Fix \(\theta>0\) such that \(m\cosh\theta<1\), and, for \(c\ge0\), set
\[
        \lambda(c):=1+\theta c-m\cosh\theta>0.
\]
Then, uniformly over \(G=\mathbb Z^d\) and \(G=\Lambda_n\), over \(x\in G\), and over
\(r\ge0\),
\[
\mathbb P\left(
    \exists u\ge r,\ \exists y\in G:
    d_G(x,y)>cu,\ Z_u^x(y)>0
\right)
\le
2^d e^{-\lambda(c)r}.
\]
For \(G=\Lambda_n\), lift every genealogical path from a fixed lift
\(\tilde x\) of \(x\), and denote the resulting particle counts on
\(\mathbb Z^d\) by \(\widetilde Z_u^{\tilde x}(z)\); for
\(G=\mathbb Z^d\), use the process itself. Then, for every \(R\ge0\),
\[
\mathbb P\left(
    \exists u\ge0,\ \exists z\in\mathbb Z^d:
    \widetilde Z_u^{\tilde x}(z)>0,\
    \|z-\tilde x\|_\infty>R
\right)
\le
2^d e^{-\theta R}.
\]
\end{lemma}

\begin{proof}
Use the lifted particle counts from the statement. For each
\(\varepsilon\in\{\pm1\}^d\), define
\[
W_u^\varepsilon
:=
\sum_{z\in\mathbb Z^d}
\widetilde Z_u^{\tilde x}(z)
e^{\theta\varepsilon\cdot(z-\tilde x)},
\qquad
\mathcal M_u^\varepsilon
:=
e^{-(m\cosh\theta-1)u}W_u^\varepsilon.
\]
The many-to-one formula \cite{ManyToOne} gives
\[
\mathbb E[W_u^\varepsilon]
=
e^{(m-1)u}
\mathbb E_{\tilde x}
\left[e^{\theta\varepsilon\cdot(S_u^{(m)}-\tilde x)}\right]
=
e^{(m\cosh\theta-1)u},
\]
where \(S^{(m)}\) is continuous-time simple random walk of total jump rate \(m\).
The branching property therefore makes \((\mathcal M_u^\varepsilon)_{u\ge0}\) a nonnegative
mean-one martingale.

Suppose that the event in the statement occurs at time \(u\), and choose a lifted particle
\(z\) above \(y\). Then
\[
        |z-\tilde x|_1\ge d_G(x,y)>cu.
\]
Choose \(\varepsilon_j=1\) when \(z_j-\tilde x_j\ge0\) and
\(\varepsilon_j=-1\) otherwise. For this choice,
\[
\mathcal M_u^\varepsilon
\ge
e^{-(m\cosh\theta-1)u}e^{\theta|z-\tilde x|_1}
>
e^{\lambda(c)u}
\ge
e^{\lambda(c)r}.
\]
Ville's inequality followed by a union bound over the \(2^d\) choices of
\(\varepsilon\) proves the first claim.

For the second claim, suppose that a lifted particle \(z\) satisfies
\(\|z-\tilde x\|_\infty>R\) at time \(u\), and again choose
\(\varepsilon\) according to the coordinatewise signs of \(z-\tilde x\).
Since \(m\cosh\theta<1\),
\[
\mathcal M_u^\varepsilon
\ge
e^{(1-m\cosh\theta)u}e^{\theta|z-\tilde x|_1}
>
e^{\theta R}.
\]
Another application of Ville's inequality and the same union bound proves the
radial estimate. Both bounds are uniform in \(n\).
\end{proof}

In all subsequent applications of Lemma~\ref{lem:brw-deviation}, we take its
tilt parameter to be \(\theta=1\). After decreasing \(\beta_0(d,q)\) if
necessary, this choice satisfies \(m\cosh 1<1\), and we write
\[
        \lambda(c):=1+c-m\cosh 1.
\]

\begin{cor}[Magnetization bounds on the torus]
\label{cor:torus-magnetization-bounds}
Fix \(0<A<\infty\). For all sufficiently large \(n\), uniformly over
\(0\le t\le A\log n\), \(\sigma_0\in[q]^{\Lambda_n}\), and \(c\in[q]\),
\[
    \left|\mathfrak m_t^{(n)}(\sigma_0,c)\right|
    \le 2C_{\mathrm m}e^{-(1+\kappa)t}.
\]
Moreover, if \(\mathfrak c_n\equiv c\), then, uniformly over
\(t_{\mathrm{all}}\le t\le A\log n\),
\[
    \mathfrak m_t^{(n)}(\mathfrak c_n,c)
    \ge \frac{c_{\mathrm{all}}}{2}e^{-(1+\kappa)t}.
\]
\end{cor}

\begin{proof}
Let
\(\pi_n:\mathbb Z^d\to\Lambda_n\) be the quotient map, put
\(r_n:=\lfloor n/4\rfloor-1\), and let
\(\sigma_0^\sharp:=\sigma_0\circ\pi_n\) be the periodic lift of \(\sigma_0\).
Couple the torus dynamics and the infinite-volume dynamics started from
\(\sigma_0^\sharp\) by using the same graphical marks in
\(B_\infty(0,r_n+1)\). On this ball, the quotient map is a nearest-neighbor
graph isomorphism onto its image. The two backward explorations are therefore
identical up to their common first exit from \(B_\infty(0,r_n)\). Let
\(\mathcal E_{n,t}\) be the event that this common history exits before the
exploration reaches real time \(0\). On
\(\mathcal E_{n,t}^{\mathsf c}\), the two terminal spins agree.

The one-site history is dominated by the non-coalescing branching walk.
Consequently, the radial estimate in Lemma~\ref{lem:brw-deviation} gives,
uniformly in \(t\ge0\), \(\sigma_0\), and \(c\),
\[
\mathbb P(\mathcal E_{n,t})
\le
2^d e^{- r_n}.
\]
It follows that
\begin{equation}\label{eq:torus-Zd-mag-comparison}
\left|
\mathfrak m_t^{(n)}(\sigma_0,c)
-
\mathfrak m_t(\sigma_0^\sharp,c)
\right|
\le
2^d e^{- r_n}
\le
C_{\mathrm{lift}}e^{-c_{\mathrm{lift}}n}.
\end{equation}
Proposition~\ref{prop:magnetization-bound} now yields
\[
\left|\mathfrak m_t^{(n)}(\sigma_0,c)\right|
\le
C_{\mathrm m}e^{-(1+\kappa)t}
+
C_{\mathrm{lift}}e^{-c_{\mathrm{lift}}n}.
\]
The periodic lift of \(\mathfrak c_n\) is the monochromatic configuration
\(\mathfrak c\) on \(\mathbb Z^d\). Hence
Proposition~\ref{prop:all-one-magnetization-bound} and
\eqref{eq:torus-Zd-mag-comparison} give, for \(t\ge t_{\mathrm{all}}\),
\[
\mathfrak m_t^{(n)}(\mathfrak c_n,c)
\ge
c_{\mathrm{all}}e^{-(1+\kappa)t}
-
C_{\mathrm{lift}}e^{-c_{\mathrm{lift}}n}.
\]
For \(t\le A\log n\), one has
\(e^{-(1+\kappa)t}\ge n^{-A(1+\kappa)}\). Thus the exponentially small
coupling error is \(o(e^{-(1+\kappa)t})\), uniformly in this range, and can be
absorbed into the two main terms.
\end{proof}
\subsection{Cone-confined law at a separation tip}
\label{subsec:cone-confined}

We prove the one-site estimate required by the Yellow/Strong-Red splitting in
Subsection~\ref{subsec:yellowSR}. 

Fix \(\Lambda_n\) and a terminal horizon \(T\) such that
\begin{equation}\label{eq:cone-embedded}
2\bigl(\lceil c_{\rm cone}T\rceil+1\bigr)<n.
\end{equation}
For the cutoff application \(T=t_\star+s=O(\log|\Lambda_n|)\) for every fixed
offset \(s\), so \eqref{eq:cone-embedded} holds for all sufficiently large
\(n\).
Let \(\pi_n:\mathbb Z^d\to\Lambda_n\) denote the quotient map. Writing
\(R_T:=\lceil c_{\rm cone}T\rceil\), condition
\eqref{eq:cone-embedded} makes \(\pi_n\) injective on
\(\{z:|z|_1\le R_T+1\}\), and it identifies the full nearest-neighbor
neighborhood of every site in \(\{z:|z|_1\le R_T\}\) with its torus
neighborhood. Put
\[
\mathcal K(c):=
\{(z,u)\in\mathbb Z^d\times[0,\infty):|z|_1\le cu\}.
\]
For a sausage \(x\), let \(\operatorname{Supp}(x)\subset
\mathbb Z^d\times[0,\ell(x)]\) be its recentered space--time support. Given
\(\mathbf x=(x_1,\ldots,x_r)\), set \(L_0=0\), \(S_0=0\), and
\[
L_j:=\sum_{i=1}^j\ell(x_i),
\qquad
S_j:=\sum_{i=1}^j y(x_i).
\]
Its concatenated support in the lifted cone is
\[
\operatorname{Supp}(\mathbf x)
:=
\bigcup_{i=1}^r
\left\{
\bigl(S_{i-1}+z,L_{i-1}+u\bigr):
(z,u)\in\operatorname{Supp}(x_i)
\right\},
\]
and we define
\[
\operatorname{Conf}(\mathbf x)
:=
\mathbf 1\{\operatorname{Supp}(\mathbf x)\subset\mathcal K(c_{\rm cone})\},
\qquad
\operatorname{Conf}(\varnothing):=1.
\]

For \(\bar z\in\Lambda_n\) and \(1\le t\le T\), define
\[
\IK_{\Lambda_n}^{\rm cone}(\bar z,t)
:=
\sum_{r\ge1}
\sum_{x_1,\ldots,x_r\in\Xi}
\left(\prod_{i=1}^r\mu(x_i)p(x_i)\right)
\mathbf 1\{L_r=t,\ \pi_n(S_r)=\bar z\}
\operatorname{Conf}(x_1,\ldots,x_r),
\]
and set
\(\IK_{\Lambda_n}^{\rm cone}(\bar z,0):=\mathbf 1_{\{\bar z=\bar0\}}\).
By \eqref{eq:cone-embedded}, this is precisely the complete-sausage kernel of
the torus history confined to the cone: all graphical marks seen by that
history, including the update neighborhoods it queries, lie in the locally
identified region and therefore have the same joint law as the corresponding
marks in the i.i.d.\ \(\mathbb Z^d\) construction; in particular, history
coalescences coincide under this identification.

\begin{lemma}[Cone-restricted torus kernel]
\label{lem:IK-cone-l1}
After decreasing \(\beta_0(d,q)\), there is \(C_{\rm cone}<\infty\) such that, whenever
\eqref{eq:cone-embedded} holds,
\[
        \sum_{\bar z\in\Lambda_n}|\IK_{\Lambda_n}^{\rm cone}(\bar z,t)|
        \le
        C_{\rm cone}e^{-(1+\kappa)t}
        \qquad(0\le t\le T,\ t\in\mathbb Z).
\]
\end{lemma}

\begin{proof}
Set \(\delta_0:=(\eta-1)/8\). Note that $(1+c_{\rm cone})/2>1+2\delta_0$ as $c_{\rm cone}=2$.
We decrease \(\beta_0\) so that
\[
        m\cosh\theta<1,
        \qquad
        m\le(\eta-1)/2,
        \qquad
        m\cosh\theta/2<\delta_0,
\]
with $\theta=1$.
For the proof, let \(\IK^{\rm cone}(z,t)\) be the same signed sum as
\(\IK_{\Lambda_n}^{\rm cone}(\bar z,t)\), but with the endpoint condition
\(S_r=z\). A confined sequence of total depth \(t\le T\) has
\(|S_r|_1\le c_{\rm cone}T\). Hence \eqref{eq:cone-embedded} makes its endpoint
lift unique and gives
\begin{equation}\label{eq:cone-lift-l1}
\sum_{\bar z\in\Lambda_n}|\IK_{\Lambda_n}^{\rm cone}(\bar z,t)|
=
\sum_{z\in\mathbb Z^d}|\IK^{\rm cone}(z,t)|.
\end{equation}

For \(s\ge1\), define the lifted first-exit kernel by
\[
\begin{aligned}
\IK^{\rm exit}(z,s)
:={}&
\sum_{r\ge1}
\sum_{x_1,\ldots,x_r\in\Xi}
\left(\prod_{i=1}^r\mu(x_i)p(x_i)\right)
\mathbf 1\{L_r=s,\ S_r=z\}\\
&\qquad\times
\mathbf 1\{\operatorname{Conf}(x_1,\ldots,x_{r-1})=1,
\ \operatorname{Conf}(x_1,\ldots,x_r)=0\}.
\end{aligned}
\]
Splitting every non-confined sequence at the first sausage that leaves the cone gives, with
convolution on \(\mathbb Z^d\),
\[
\IK^{\rm cone}(\cdot,t)
=
\IK(\cdot,t)
-
\sum_{s=1}^t
\IK^{\rm exit}(\cdot,s)*\IK(\cdot,t-s).
\]
Consequently,
\begin{equation}\label{eq:cone-first-exit-convolution}
\|\IK^{\rm cone}(\cdot,t)\|_1
\le
\|\IK(\cdot,t)\|_1
+
\sum_{s=1}^t
\|\IK^{\rm exit}(\cdot,s)\|_1
\|\IK(\cdot,t-s)\|_1.
\end{equation}

We now bound the lifted first-exit kernel. Define
\[
U_j
:=
\mathbf 1_{\{j=0\}}
+
\sum_{r\ge1}
\sum_{x_1,\ldots,x_r\in\Xi}
\left(\prod_{i=1}^r\mu(x_i)\right)
\mathbf 1\{L_r=j\}.
\]
Because the history is
dominated by the non-coalescing branching walk, we have $U_j\leq e^{-(1-m)j}$. 
Take absolute values in \(\IK^{\rm exit}\) and use \(|p(x)|\le1\). The resulting
unsigned sum is the probability, in one i.i.d. sausage construction, of first leaving the cone
during the final sausage of a concatenation ending at depth \(s\). Let \(l\) be the length of
that final sausage. For \(l\ge s/2\), the bound on \(U_j\) and
\eqref{eq:len-pmf} give
\begin{align*}
\|\IK^{\rm exit}(\cdot,s)\|_{1;\,l\ge s/2}
&\le
\sum_{l=\lceil s/2\rceil}^s U_{s-l}\,\mu(\ell(x)=l)\\
&\le
C\sum_{l=\lceil s/2\rceil}^s
e^{-(1-m)(s-l)}e^{-\eta l}
\le
Ce^{-c_{\rm long}s},
\end{align*}
where $c_{\rm long}:=\frac{\eta+1-m}{2}$. 
Our choice of \(\beta_0\) ensures \(c_{\rm long}\ge1+2\delta_0\).

If \(l<s/2\), the first exit occurs during the final sausage, whose top is at depth
\(s-l>s/2\). The unsigned event just identified is therefore contained in the event that the
dominating branching walk has a particle outside its cone at some elapsed time \(u>s/2\).
Lemma~\ref{lem:brw-deviation} yields
\[
        \|\IK^{\rm exit}(\cdot,s)\|_{1;\,l<s/2}
        \le
        2^d e^{-\lambda(c_{\rm cone})s/2}.
\]
Our choices give \(\lambda(c_{\rm cone})/2>1+\delta_0\). Since \(\kappa<0\),
\begin{equation}\label{eq:cone-exit-bound}
        \|\IK^{\rm exit}(\cdot,s)\|_1
        \le
        Ce^{-((1+\kappa)+\delta_0)s}.
\end{equation}
For the unrestricted kernel, Lemma~\ref{lem:IK-l1} gives
\[
        \|\IK(\cdot,t)\|_1\le C_{\IK,1}e^{-(1+\kappa)t}.
\]
Substituting this and \eqref{eq:cone-exit-bound} into
\eqref{eq:cone-first-exit-convolution}, and then using
\eqref{eq:cone-lift-l1}, proves the result. The case \(t=0\) is immediate.
\end{proof}

We now prove the residual-mass estimate. Recall that
\(\mathcal H^{v,t}\) is the one-site history from \((v,t)\) generated by the
independent copy of the graphical construction introduced in
Subsection~\ref{subsec:yellowSR}, that
\(\mathsf{Bot}_t(v)\) is the event that this history reaches time \(0\)
without leaving its cone, and that
\(M_t=\mathbb P(\mathsf{Bot}_t(v))\).

\begin{prop}[Residual weight of an isolated cone]\label{prop:isolated-cone-residual}
Under \eqref{eq:cone-embedded}, there exists \(C<\infty\) such that, for every
initial configuration \(\sigma_0\in[q]^{\Lambda_n}\) and every
\((v,t)\in\Lambda_n\times[0,T]\),
\[
\mathbb{P}_{\sigma_0}\left(
\mathsf{Bot}_t(v),\,J_{(v,t)}>p_t(v)
\right)\leq C e^{-(1+\kappa)t}.
\]
\end{prop}

\begin{proof}

The case \(t=0\) follows from the preceding conventions after increasing \(C\) if necessary.
Assume henceforth that \(t>0\).
By translation invariance, take \(v=0\) and translate \(\sigma_0\) accordingly.
Let \(\sigma_0^\sharp:=\sigma_0\circ\pi_n\) be its periodic lift.

First observe that, for every \(a\in[q]\),
\[
M_t\left(\mathsf S(t,v,\sigma_0)(a)-\frac{1}{q}\right)
=
\mathbb E_{\sigma_0}\left[
\mathbf 1_{\mathsf{Bot}_t(v)}
\left(\mathbf 1_{\{\sigma_t(v)=a\}}-\frac1q\right)
\right].
\]
On \(\mathsf{Bot}_t(v)\), the surviving history has a unique terminal partial
sausage. If the complete sausages above it have total depth
\(n_t-\tau\), then the terminal cut depth is \(\theta_t+\tau\).
Under \eqref{eq:cone-embedded}, lift this truncated history to
\(\mathbb Z^d\) and complete its terminal sausage below time \(0\) with
independent marks, obtaining \(x_0\sim\mu\). After taking absolute values,
discard the terminal confinement condition and use
\(\lvert g_{x_0,\sigma_0^\sharp,a}\rvert\le1\). Thus
Proposition~\ref{prop:sausagetail} and Lemma~\ref{lem:IK-cone-l1} give

\[
M_t\left|
    {\mathsf S}(t,v,\sigma_0)(a)
    -\frac{1}{q}
    \right|
    \le
    \sum_{\tau=0}^{n_t}
    \mu(\ell(x_0)>\theta_t+\tau)
    \|\IK_{\Lambda_n}^{\rm cone}(\cdot,n_t-\tau)\|_1
    \leq
    C\sum_{\tau=0}^{n_t}
    e^{-\eta\tau}
    e^{-(1+\kappa)(n_t-\tau)}
    \leq C' e^{-(1+\kappa)t},
\]
because \(\eta>1>1+\kappa\) and \(n_t\ge t-1\). Therefore,
\[
\begin{aligned}
\mathbb{P}_{\sigma_0}\left(
\mathsf{Bot}_t(v),\,J_{(v,t)}>p_t(v)
\right)
=M_t\bigl(1-p_t(v)\bigr)
=M_t\left(
1-q\inf_{\xi_0\in[q]^{\Lambda_n}}
\min_{a\in[q]}\mathsf S(t,v,\xi_0)(a)
\right)
\leq qC' e^{-(1+\kappa)t}.
\end{aligned}
\]

\end{proof}

Collecting the preceding estimates, fix
\(\alpha_4=\alpha_4(d,q)>0\) sufficiently small that all the
smallness conditions in this section hold whenever
\(0<\alpha\le\alpha_4\).

%%%%%%%%%%%%%%%%%%%%%%%%%%%%%%%%%%%%%%%%%%%%%%%%%%%%%%%%%%%%%%%%%%%%%%%%%%%%%%%%%%%%%%%%%%%%%%%%%%%%%%%%%%%%%%%%%%%%%%%%%%%%%%%%%%%%%%%%%%%%%%%%%%%%%%%%%%%%%%%%%%%%%%%%%%%%%%%%%%%%%%%%%%%%%%%%%%%%%%%%%%%%%%%%%%%%%%%%%%%%%%%%%%%%%%%%%%%%%%%%

\medskip
\section{Cutoff with a constant window}\label{sec:cutoff}

We retain the restriction \(0<\alpha\le\alpha_4\). 
All auxiliary
penalty parameters used below (\(\lambda_{\rm P}\), \(\lambda_{\rm SR}\) and \(\lambda_0\)) are  fixed independently of \(\beta\). Once these parameters have been fixed, choose \(\beta_0=\beta_0(d,q)>0\)
sufficiently small that \(\alpha(\beta_0)\le\alpha_4\) and all the estimates in this section hold.

\subsection{Upper bound}\label{subsec:cutoff-upper}
Throughout this subsection, \(\Lambda=\Lambda_n\) and
\[
        T=t_\star+s,
        \qquad s\ge0,
        \qquad
        t_\star=\frac{1}{2(1+\kappa)}\log|\Lambda|.
\]

By Proposition~\ref{prop:IP-overlap}, it is enough to prove, after decreasing
\(\beta_0(d,q)\) if necessary, that
\[
\mathbb E\left[
q^{|V_{\mathrm P}^{(1)}\cap V_{\mathrm P}^{(2)}|}-1
\right]
=
o(e^{-2(1+\kappa)s}),
\qquad
\mathbb E\left[
q^{|V_{\mathrm{SR}}^{(1)}\cap V_{\mathrm{SR}}^{(2)}|}-1
\right]
=
O(e^{-2(1+\kappa)s}).
\]
The next two subsubsections establish these estimates for the Purple and
Strong--Red sets, respectively.

For every nonempty $A\subset\Lambda$, fix deterministically a connected set
$\operatorname{Hull}(A)\subset\Lambda$ of minimal cardinality such that
$A\subset\operatorname{Hull}(A)$, and set
\[
        W(A):=|\operatorname{Hull}(A)|.
\]
We abbreviate
\[
        W(A,B):=W(A\cup B).
\]

\subsubsection{Purple overlap} 
We prove the Purple-overlap estimate
\[
\mathbb E\left[
q^{|V_{\mathrm{P}}^{(1)}\cap V_{\mathrm{P}}^{(2)}|}-1
\right],
\]
where \(V_{\mathrm{P}}^{(1)}:=V_{\mathrm{P}}\) and \(V_{\mathrm{P}}^{(2)}\) is a conditionally independent copy of \(V_{\mathrm{P}}^{(1)}\) given
\(\mathcal H_{G,Y}\vee\mathcal H_{\mathrm{SR}}\).
Let
\[
        \mathcal K:=\mathcal H_{G,Y}\vee\mathcal H_{\mathrm{SR}} .
\]
We write $A\in\mathrm{Pur}^{(a)}$ when $A$ is the top set of a Purple cluster
in copy \(a\), and \(A\in\mathrm{Red}^{(a)}\) when \(A\) is the top set of a Red cluster in copy \(a\).
For a finite nonempty \(D\subseteq\Lambda\), define

\[
\mathsf T_D^{(a)}
        :=
        \{D\in\mathrm{Red}^{(a)}\}
        \cup
        \{D\subset V_{\mathrm{Blue}}^{(a)}\}.
\]
and for a finite set \(Y\subseteq\Lambda\), write
\[
        \mathcal H_D\rightsquigarrow Y
        \quad\Longleftrightarrow\quad
        Y\subseteq\operatorname{Span}(\mathcal H_D).
\]

\paragraph{The cost of bad Red geometry.}

Purple clusters are Red clusters whose geometry is bad. We begin with the estimates that quantify the relevant costs.

\begin{lemma}\label{lem:single_bound}
Fix \(\lambda_{\rm P}>0\). After decreasing \(\beta_0(d,q)\) if necessary, there
exists \(C_{\rm h}<\infty\) such that the following holds.

Let \(D\subset\Lambda\) be finite and nonempty, and let \(Y\subset\Lambda\) be
finite. For a realization \(X\) of
\(\mathcal H_D^-\), write \(X\sim\mathcal K\) if \(X\) is consistent with the
Green, Yellow, and Strong-Red data revealed by \(\mathcal K:=\mathcal H_{G,Y}\vee\mathcal H_{\mathrm{SR}}\). Then,
\[
        \Psi(D;Y):=\sup_{\mathcal H_D^-\sim\mathcal K}
        \mathbb P\left(
        D\in\mathrm{Red},\,
        \mathcal H_D\rightsquigarrow Y
        \,\middle|\,
        \mathcal H_D^-,\,
        \mathsf T_D
        \right)
        \le
        C_{\rm h} e^{-(1-m)T}e^{-\lambda_{\rm P} W(D,Y)} ,
\]
where the supremum is over conditioning events of positive probability.
\end{lemma}

\begin{proof}
If $D$ overlaps with the top set of a Green, Yellow, or Strong-Red
cluster revealed by \(\mathcal K\), $D$ cannot be a Red cluster, so the probability becomes zero. We only consider the cases when it does not overlap.

Fix \(\mathcal K\) and an outside transcript \(X\sim\mathcal K\). 
Conditional on \(\mathcal H_D^-=X\), the additional information in \(\mathcal K\) can only rule out realizations of the \(D\)-history that are incompatible with the revealed Yellow, and Strong-Red data. 
Every realization in which \(D\subset V_{\mathrm{Blue}}\) remains compatible with \(\mathcal K\) if it was compatible with $X$. 
However, it may restrict $D\in\Red$ because $D$ should be $R$-separated from realized Strong Red and Yellow set, and geometrically bad in addition to being Red. 
Therefore, on conditioning events of positive probability,
\[
\mathbb P\left(
D\in\mathrm{Red},\,\mathcal H_D\rightsquigarrow Y
\,\middle|\,
\mathcal H_D^-=X,\mathsf T_D
\right) 
\le
\frac{
\mathbb P\left(
D\in\mathrm{Red},\,\mathcal H_D\rightsquigarrow Y
\,\middle|\,
\mathcal H_D^-=X
\right)}
{
\mathbb P\left(
D\subset V_{\mathrm{Blue}}
\,\middle|\,
\mathcal H_D^-=X
\right)
}.
\]

The right-hand side is exactly the Red/Blue ratio bounded in Lemma 2.1 of \cite{Exposition}, with the extra support constraint \(\mathcal H_D\rightsquigarrow Y\). The extra condition \(\mathcal H_D\rightsquigarrow Y\) forces the spatial support of \(\mathcal H_D\) to contain a connected lattice animal spanning \(D\cup Y\), so the same lattice-animal enumeration and exponential-moment estimate as in
Lemma 2.1 gives the factor \(e^{-\lambda_{\rm P} W(D,Y)}\), after decreasing \(\beta_0(d,q)\). 
Finally, we get $e^{-(1-m)T}$ from the submultiplicativity of the magnetization  $m_{T}\leq m_0 e^{-(1-m)T}$ from (2.10) in \cite{Exposition} to prove the claim.
\end{proof}

\begin{lemma}[Multi-cluster history bound]
\label{lem:multi-red-bound}
Let \(a\in\{1,2\}\). 
Let \(D_1,\ldots,D_r\subseteq\Lambda\) be pairwise disjoint finite nonempty sets,
each of which is disjoint from every top set of a Green, Yellow, or Strong-Red
cluster revealed by \(\mathcal K\).
Let
\(Y_1,\ldots,Y_r\subseteq\Lambda\) be finite. Then under the condition of Lemma~\ref{lem:single_bound},
\[
\mathbb P\left(
        \bigcap_{i=1}^r
        \{D_i\in{\rm Red}^{(a)},\,
          \mathcal H_{D_i}^{(a)}\rightsquigarrow Y_i\}
        \,\middle|\,\mathcal K
\right)
\le
C_{\rm h}^r e^{-r(1-m)T}
\prod_{i=1}^r e^{-\lambda_{\rm P} W(D_i,Y_i)} .
\]
\end{lemma}

\begin{proof}
Set
\[
        E_i:=
        \{D_i\in{\rm Red}^{(a)},\,
          \mathcal H_{D_i}^{(a)}\rightsquigarrow Y_i\},
\]
and 
\[
        \mathscr F_i
        :=
        \sigma(\mathbf 1_{E_1},\ldots,\mathbf 1_{E_i}).
\]
We claim that
\[
\mathbb{P}(E_i\mid\mathcal K,\mathscr F_{i-1})\leq \Psi (D_i;Y_i).
\]
Indeed, since \(E_i\subseteq \mathsf T_{D_i}^{(a)}\), $\mathbb P(E_i\mid \mathcal K,\mathscr F_{i-1})
\le
\mathbb P(E_i\mid \mathcal K,\mathscr F_{i-1},\mathsf T_{D_i}^{(a)}).$
On \(\mathsf T_{D_i}^{(a)}\), the top histories started from
\(\Lambda\setminus D_i\) determine the events \(E_j\), \(j<i\), because the
sets \(D_j\) are disjoint from \(D_i\). Hence, after conditioning on
\(\mathcal H_{D_i}^{-,(a)}\), the sigma-field \(\mathscr F_{i-1}\) adds no
further information relevant to \(E_i\). Taking the supremum over all outside
transcripts compatible with \(\mathcal K\) gives
\[
\mathbb P(E_i\mid\mathcal K,\mathscr F_{i-1})
\le \Psi(D_i;Y_i).
\]

Iterating, we get
\[
\mathbb P\left(\bigcap_{i=1}^r E_i\,\middle|\,\mathcal K\right)
\le
\prod_{i=1}^r \Psi(D_i;Y_i).
\]
Applying Lemma~\ref{lem:single_bound} to each factor yields
\[
\prod_{i=1}^r \Psi(D_i;Y_i)
\le
C_{\rm h}^r e^{-r(1-m)T}
\prod_{i=1}^r e^{-\lambda_{\rm P} W(D_i,Y_i)}.
\]
\end{proof}

\paragraph{Purple components.}

For nonempty \(D,E\subseteq\Lambda\), write
\[
        D\sim_R E
        \quad\Longleftrightarrow\quad
        \operatorname{dist}(D,E)\le R.
\]
For each copy \(a\), let \(\mathfrak R^{(a)}\) be the family of Red top
sets and equip it with the graph induced by \(\sim_R\).
A family of top sets is called \(R\)-connected when its graph under
\(\sim_R\) is connected. A Red top set \(D\) is \emph{long} in copy \(a\)
if
\[
        \max_{u\in\operatorname{Span}(\mathcal H_D^{(a)})}
        \operatorname{dist}(u,D)>R/3.
\]
A component \(\mathcal C\) of the Red proximity graph is \emph{bad} if
\(|\mathcal C|\ge2\), or if \(\mathcal C=\{D\}\) and \(D\) is long. Set
\[
        U(\mathcal C):=\bigcup_{D\in\mathcal C}D.
\]
Let
\(\mathsf{Bad}_{\mathrm P}^{(a)}\) be the family of bad components. Then, by the definition of Purple clusters in
Subsection~\ref{subsec:yellowSR},
\[
        V_{\mathrm P}^{(a)}
        =
        \bigcup_{\mathcal C\in\mathsf{Bad}_{\mathrm P}^{(a)}}U(\mathcal C),
\]
and the supports in this union are pairwise disjoint.

For the counting argument, let \(\mathsf{Cand}_{\mathrm P}\) be the collection
of all finite pairwise-disjoint \(R\)-connected families \(\mathcal C\) of
nonempty subsets of \(\Lambda\), including singletons. Thus every realized bad
component belongs to \(\mathsf{Cand}_{\mathrm P}\); a singleton candidate
\(\{D\}\) represents the possible event that \(D\) is an isolated long
component. A family
\(\mathcal S\subseteq\mathsf{Cand}_{\mathrm P}\) is \emph{compatible} if the
supports \(U(\mathcal C)\), \(\mathcal C\in\mathcal S\), are pairwise
disjoint.

For \(\mathcal C\in\mathsf{Cand}_{\mathrm P}\), define
\[
\omega(\mathcal C)
:=
\begin{cases}
\displaystyle
C_{\rm h}e^{-(1-m)T} \sum_{\substack{\xi\in\Lambda,\,
                   \operatorname{dist}(\xi,D)>R/3}}
e^{-\lambda_{\rm P}W(D\cup\{\xi\})},
&\mathcal C=\{D\},\\[4mm]
\displaystyle
\prod_{D\in\mathcal C}
 C_{\rm h}e^{-(1-m)T}e^{-\lambda_{\rm P}W(D)},
&|\mathcal C|\ge2,
\end{cases}
\qquad
\zeta(\mathcal C):=q^{|U(\mathcal C)|}\omega(\mathcal C).
\]

These activities dominate the joint occurrence of compatible bad
components. Namely, for \(a\in\{1,2\}\) and every finite compatible
\(\mathcal S\subseteq\mathsf{Cand}_{\mathrm P}\),
\begin{equation}
\label{eq:purple-component-domination}
\mathbb P\left(
        \mathcal S\subseteq\mathsf{Bad}_{\mathrm P}^{(a)}
        \,\middle|\,\mathcal K
\right)
\le
\prod_{\mathcal C\in\mathcal S}\omega(\mathcal C),
\end{equation}
by applying Lemma~\ref{lem:multi-red-bound} once to all the
pairwise-disjoint top sets, with \(Y_D=\varnothing\) for every top set in a component of size at least two. If one of the candidate top sets is
incompatible with the data in \(\mathcal K\), the probability on the left is
zero.

\begin{lemma}[Summability of Purple activities]
\label{lem:p-local-activity}
Fix \(\lambda_{\rm P}\) sufficiently large in terms of \(d\) and \(q\).
After possibly decreasing \(\beta_0(d,q)\), let \(C_{\rm h}\) denote the
constant supplied by Lemma~\ref{lem:multi-red-bound} corresponding to
this choice of \(\lambda_{\rm P}\). Then there exist constants
\(C_{\rm loc}<\infty\) and \(c_{\rm loc}>0\) such that, in the cutoff
upper-bound regime \(T=t_\star+s\), uniformly as \(s\) ranges over any
fixed bounded interval and for all sufficiently large \(\Lambda\),
\[
        \sup_{v\in\Lambda}
        \sum_{\substack{\mathcal C\in\mathsf{Cand}_{\mathrm P}\\
                        v\in U(\mathcal C)}}
        \zeta(\mathcal C)
        \le
        C_{\rm loc}
        \left(
        e^{-(1-m)T-c_{\rm loc}R}
        +R^d e^{-2(1-m)T}
        \right)
        \le
        C_{\rm loc}R^d e^{-2(1-m)T}.
\]
\end{lemma}

\begin{proof}

We begin with two lattice-animal estimates. For every integer \(k\ge0\),
\begin{equation}
\label{eq:purple-moment}
        \sup_{x\in\Lambda}
        \sum_{D\ni x}
        q^{|D|}|D|^k e^{-\lambda_{\rm P}W(D)}
        \le Ck!.
\end{equation}
Indeed, there is a constant \(a_d<\infty\) such that the number of
connected hulls of size \(\ell\) containing \(x\) is at most
\(a_d^\ell\). For each such hull, the total \(q\)-weight of its subsets
is at most \((1+q)^\ell\). Since
\[
        |D|^k\le \ell^k\le k!e^\ell,
\]
summing over \(\ell\) proves \eqref{eq:purple-moment} once
\(\lambda_{\rm P}\) is sufficiently large.

The same enumeration gives
\begin{equation}
\label{eq:purple-far}
        \sup_{x\in\Lambda}
        \sum_{\substack{D\ni x,\ \xi\in\Lambda\\
                        \operatorname{dist}(\xi,D)>R/3}}
        q^{|D|}e^{-\lambda_{\rm P}W(D\cup\{\xi\})}
        \le Ce^{-cR}.
\end{equation}
Indeed, any connected hull spanning \(D\cup\{\xi\}\) has size at least
\(R/3\), while a hull of size \(\ell\) contains at most \(\ell\) possible
choices for \(\xi\). Thus the sum is an exponentially decaying tail of the
same geometric series.

For singleton candidates, \eqref{eq:purple-far} gives
\begin{equation}
\label{eq:purple-long-sum}
        \sup_{v\in\Lambda}
        \sum_{\substack{\mathcal C=\{D\}\in\mathsf{Cand}_{\mathrm P}\\
                        v\in D}}
        \zeta(\mathcal C)
        \le
        Ce^{-(1-m)T}e^{-cR}.
\end{equation}

It remains to sum the \(R\)-connected families of size at least two. Set
\[
        z_T(D):=C_{\rm h}e^{-(1-m)T}q^{|D|}e^{-\lambda_{\rm P}W(D)}.
\]
By \eqref{eq:purple-moment}, for every \(k\ge0\),
\begin{equation}
\label{eq:purple-weighted-moment}
        \sup_{x\in\Lambda}
        \sum_{D\ni x}z_T(D)|D|^k
        \le CC_{\rm h}e^{-(1-m)T}k!.
\end{equation}
Since the members of \(\mathcal C\) are disjoint,
\(\zeta(\mathcal C)=\prod_{D\in\mathcal C}z_T(D)\) whenever
\(|\mathcal C|\ge2\). For \(r\ge2\), let
\[
        S_r(v):=
        \sum_{\substack{\mathcal C\in\mathsf{Cand}_{\mathrm P}\\
                        |\mathcal C|=r,\ v\in U(\mathcal C)}}
        \zeta(\mathcal C).
\]

The unique member of \(\mathcal C\) containing \(v\) serves as the root.
Every \(R\)-connected family has a spanning tree rooted at that member. For a
fixed rooted tree on the distinct sets of \(\mathcal C\), ordering the
\(k_u\) children of every vertex \(u\) gives exactly \(\prod_u k_u!\) plane
representations. We sum these representations with the symmetry factor
\(\prod_u(1/k_u!)\); summing over all spanning trees only enlarges the
result.

Consider the tree edge from $D$. 
Each edge has at most \(CR^d\) choices for its child-side endpoint,
while a set at a vertex with \(k_u\) children contributes
\(|D|^{k_u}\). After dropping the disjointness and distinctness restrictions,
\eqref{eq:purple-weighted-moment} yields
\[
\begin{aligned}
        S_r(v)
        \le
        \sum_{\tau\in\mathcal T_r^{\rm pl}}
        (CR^d)^{r-1}
        \prod_{u\in\tau}
        \frac{1}{k_u!}
        \sup_{x\in\Lambda}
        \sum_{D\ni x}z_T(D)|D|^{k_u}
        \le
        Ce^{-(1-m)T}\bigl(CR^de^{-(1-m)T}\bigr)^{r-1},
\end{aligned}
\]
where \(\mathcal T_r^{\rm pl}\) is the collection of rooted plane tree
shapes with \(r\) vertices and
\(|\mathcal T_r^{\rm pl}|\le4^{r-1}\).

For \(s\) in a fixed bounded interval,
\[
        T=t_\star+s
        =
        \frac{1}{2(1+\kappa)}\log|\Lambda|+O(1),
\]
so \(R^d e^{-(1-m)T}=o(1)\). Summing the last display over \(r\ge2\) gives
\begin{equation}
\label{eq:purple-connected-sum}
        \sup_{v\in\Lambda}
        \sum_{\substack{\mathcal C\in\mathsf{Cand}_{\mathrm P}\\
                        |\mathcal C|\ge2,\ v\in U(\mathcal C)}}
        \zeta(\mathcal C)
        \le
        CR^de^{-2(1-m)T}.
\end{equation}
Combining \eqref{eq:purple-long-sum} and
\eqref{eq:purple-connected-sum} proves the first inequality. The second
follows for all sufficiently large \(\Lambda\), since \(R/T\to\infty\) in
the stated regime.
\end{proof}

\begin{prop}[Purple overlap]
\label{prop:purple}
After decreasing \(\beta_0(d,q)\) if necessary, there exists
\(C<\infty\) such that, at \(T=t_\star+s\), locally uniformly for
\(s\ge0\),
\[
        \mathbb E\left[
        q^{|V_{\mathrm{P}}^{(1)}\cap V_{\mathrm{P}}^{(2)}|}-1
        \right]
        \le
        C|\Lambda|R^{2d}e^{-4(1-m) T}.
\]
Consequently, after decreasing \(\beta_0(d,q)\) so that
\(2(1-m)>1+\kappa\),
\[
        \mathbb E\left[
        q^{|V_{\mathrm{P}}^{(1)}\cap V_{\mathrm{P}}^{(2)}|}-1
        \right]
        =
        o\left(e^{-2(1+\kappa) s}\right)
\]
as \(|\Lambda|\to\infty\), locally uniformly for \(s\ge0\).
\end{prop}

\begin{proof}
The idea is to thin the intersecting component pairs to a collection that is
disjoint in each copy. Joint occurrence of the retained components can then
be bounded by \eqref{eq:purple-component-domination}, while the total cost of
an intersection is controlled by the square of the local activity.

Let
\[
\mathcal Q:=
\left\{
(\mathcal C,\mathcal C')\in\mathsf{Cand}_{\mathrm P}^2:
U(\mathcal C)\cap U(\mathcal C')\ne\varnothing
\right\},
\]
and fix an ordering of this finite set. For each realization, scan the pairs
in this order and retain \((\mathcal C,\mathcal C')\) if both components are
realized and neither has appeared in an earlier retained pair. Denote the
resulting collection by \(\mathcal M\). Its supports are pairwise disjoint in
each coordinate.

The retained pairs still cover the Purple overlap. Indeed, every
\(x\in V_{\mathrm P}^{(1)}\cap V_{\mathrm P}^{(2)}\) belongs to unique realized
components \(\mathcal C\) and \(\mathcal C'\) in the two copies. If
\((\mathcal C,\mathcal C')\) was not retained, then one of these components
already appeared in a retained pair, whose support therefore contains \(x\).
Thus
\[
|V_{\mathrm P}^{(1)}\cap V_{\mathrm P}^{(2)}|
\le
\sum_{(\mathcal C,\mathcal C')\in\mathcal M}
\bigl(|U(\mathcal C)|+|U(\mathcal C')|\bigr).
\]

Set
\[
b(\mathcal C,\mathcal C')
:=
q^{|U(\mathcal C)|+|U(\mathcal C')|}-1.
\]
Expanding the resulting product and taking expectations gives
\[
\begin{aligned}
\mathbb E\left[
q^{|V_{\mathrm P}^{(1)}\cap V_{\mathrm P}^{(2)}|}-1
\right]
\le
\sum_{\substack{\varnothing\ne\mathcal S\subseteq\mathcal Q\\
                 \mathcal S\ \mathrm{compatible}}}
\left(
\prod_{(\mathcal C,\mathcal C')\in\mathcal S}
b(\mathcal C,\mathcal C')
\right)
\mathbb P(\mathcal S\subseteq\mathcal M),
\end{aligned}
\]
where compatibility means disjoint supports in each coordinate.

For a fixed compatible \(\mathcal S\), the event
\(\mathcal S\subseteq\mathcal M\) requires all first-coordinate components
to occur in copy \(1\) and all second-coordinate components to occur in copy
\(2\). Conditional independence given \(\mathcal K\), together with
\eqref{eq:purple-component-domination}, therefore gives
\[
\mathbb P(\mathcal S\subseteq\mathcal M)
\le
\prod_{(\mathcal C,\mathcal C')\in\mathcal S}
\omega(\mathcal C)\omega(\mathcal C').
\]
Since
$b(\mathcal C,\mathcal C')
\le q^{|U(\mathcal C)|+|U(\mathcal C')|}$,
each pair contributes at most
\(\zeta(\mathcal C)\zeta(\mathcal C')\). Dropping compatibility and summing
over all subsets of \(\mathcal Q\), we obtain
\[
\mathbb E\left[
q^{|V_{\mathrm P}^{(1)}\cap V_{\mathrm P}^{(2)}|}-1
\right]
\le e^\rho-1,
\qquad
\rho:=
\sum_{(\mathcal C,\mathcal C')\in\mathcal Q}
\zeta(\mathcal C)\zeta(\mathcal C').
\]

An intersecting pair can be rooted at any common vertex, so
Lemma~\ref{lem:p-local-activity} yields
\[
\begin{aligned}
\rho
&\le
\sum_{v\in\Lambda}
\left(
\sum_{\substack{\mathcal C\in\mathsf{Cand}_{\mathrm P},\,
                 v\in U(\mathcal C)}}
\zeta(\mathcal C)
\right)^2\le
C|\Lambda|R^{2d}e^{-4(1-m)T}.
\end{aligned}
\]

Using the allowed decrease of \(\beta_0(d,q)\), we may assume $\delta:=4(1-m)-2(1+\kappa)>0$. 
Since \(|\Lambda|=e^{2(1+\kappa)t_\star}\), at \(T=t_\star+s\), $\rho\leq CR^{2d}e^{-\delta t_\star}e^{-4(1-m)s}
=o(1)$
locally uniformly for \(s\ge0\). Hence,
\[
\mathbb E\left[
q^{|V_{\mathrm P}^{(1)}\cap V_{\mathrm P}^{(2)}|}-1
\right]
\le
C|\Lambda|R^{2d}e^{-4(1-m)T}.
\]
with sufficiently large \(\Lambda\) to apply
\(e^\rho-1\le2\rho\).
Finally,
\[
|\Lambda|R^{2d}e^{-4(1-m)T}
=
e^{-2(1+\kappa)s}
R^{2d}e^{-\delta(t_\star+s)}
=
o\left(e^{-2(1+\kappa)s}\right).
\]
\end{proof}

\smallskip

\subsubsection{Strong--Red overlap}

    We bound
\[
\mathbb E\left[
q^{|V_{\mathrm{SR}}^{(1)}\cap V_{\mathrm{SR}}^{(2)}|}-1
\right],
\]
where \(V_{\mathrm{SR}}^{(1)}:=V_{\mathrm{SR}}\) and $V_{\mathrm{SR}}^{(2)}$ is a conditionally independent copy of \(V_{\mathrm{SR}}^{(1)}\) given \(\mathcal H_{G,Y}\vee\mathcal H_{\mathrm{P}}\).
Let
\[
        \mathcal K:=\mathcal H_{G,Y}\vee\mathcal H_{\mathrm{P}} .
\]
We write
\(A\in\mathrm{SR}^{(a)}\) for the event that \(A\) is the top set of a
Strong-Red cluster in copy \(a\).
For copy \(a\in\{1,2\}\), let
\[
    \mathsf T_A^{(a)}
    :=
    \{A\in\mathrm{SR}^{(a)}\}
    \cup
    \{A\subseteq V_{\mathrm{Blue}}^{(a)}\}.
\]

\paragraph{Separation depth}

For a nonempty spatial set \(A\subset\Lambda\) and \(h\ge0\), define
\[
        \mathsf D_h(A)
        :=
        \left\{
        (x,s)\in\Lambda\times[0,T-h]:
        \dist(x,A)\le c_{\rm cone}(T-s)
        \right\},
\]
with the convention that \({\mathsf D}_h(A)=\varnothing\) if \(h>T\).
If \(\mathcal H\subseteq\Lambda\times[0,T]\) is a space--time set, define
\[
        \ell_A(\mathcal H)
        :=
        \min\left\{
        T,\,
        1+
        \inf\left\{
        h\in[0,T):
        \mathcal H\cap\mathsf D_h(A)=\varnothing
        \right\}
        \right\},
\]
where the infimum is defined to be \(T\) when the displayed set is empty.
Plus one is added to ensure that $\mathcal H$ and $\mathsf D_h(A)$ are disjoint at depth $\ell_A(\mathcal H)$.

\begin{prop}
\label{prop:one-block-SR}
For every \(\lambda_{\rm SR}>0\), after decreasing \(\beta_0(d,q)\) if necessary, there exist constants \(C_{\mathrm{sr}}<\infty\) and \(c_{\mathrm{sr}}<\infty\) such that the following holds.  For every non-empty \(A\subseteq\Lambda\), and \(0\le h\le T \),
\[
    \sup_{\substack{X:\ \ell_A(X)=h}}
    q^{|A|}
    \mathbb P\left(
        A\in\mathrm{SR}
        \,\middle|\,
        \mathcal H_A^{-}=X,
        \mathsf T_A
    \right)
    \le
    C_{\mathrm{sr}}
    e^{-(1+\kappa)T}
    e^{-\lambda_{\rm SR} W(A)}
    e^{c_{\mathrm{sr}}\beta h}.
\]
\end{prop}
Its proof is deferred to the end of this subsubsection.

For copy \(a\in\{1,2\}\), let $\mathcal H_A^{-,(a)}$
denote the outside transcript obtained by revealing the histories started from
\((\Lambda\setminus A)\times\{T\}\) in copy \(a\). We write
\[
    X\sim_A \mathcal K
\]
when \(X\) is a compatible outside transcript of Strong Red set $A$, $\mathcal H_A^-$, with the conditioned information \(\mathcal K\). 
If the revealed geometry
in \(\mathcal K\) (or $X$) already forces \(A\) to be Purple, or not a single information percolation cluster, then \(A\) is declared
incompatible with \(\mathcal K\) (or $X$). Define
\[
    \Psi(A,\mathcal K)
    :=
    \sup_{X:\,X\sim_A\mathcal K}
    \mathbb P\left(
        A\in\mathrm{SR}
        \,\middle|\,
        \mathcal H_A^{-}=X,
        \mathsf T_A
    \right).
\]
If \(A\) is incompatible with \(\mathcal K\), we set $\Psi(A,\mathcal K)=0$.

If $A$ is compatible with $\mathcal K$, fix a compatible transcript \(X\sim_A\mathcal K\) which is also compatible with $A$.
Remark that if \(A\) is a geometrically regular Red cluster, every other geometrically Red cluster is disjoint from \({\mathsf D}_0(A)\). Therefore 
$\ell_A(\mathcal H_A^-) =
\ell_A(\mathcal H_{\rm Blue}\cup\mathcal H_G\cup\mathcal H_P)
\leq \ell_A(\mathcal H_G\cup\mathcal H_P)+1$, as blue can survive at most depth one. 
Applying Proposition~\ref{prop:one-block-SR} gives
\begin{equation}\label{eq:activity-pointwise}
q^{|A|}\Psi (A,\mathcal K)
\le
C\,
        e^{-(1+\kappa)T}
        e^{-\lambda_{\rm SR} W(A)}
        e^{c_{\mathrm{sr}}\beta \ell_A(\mathcal H_{G,P})}.
\end{equation}

We shall use the following moment bound.
\begin{lem}
\label{lem:sum-h}
There exist constants \(\delta_0>0\) and \(C_{\mathrm{bd}}<\infty\) such that
the following holds. Let \(K_1,\ldots,K_k\subset\Lambda\) be nonempty sets
satisfying
\[
        \operatorname{dist}(K_i,K_j)>R,
        \qquad i\ne j.
\]
Let \(\mathcal H(S)\) be the backward history generated from \(S\times\{T\}\), for a deterministic
\(S\subseteq\Lambda\).
Then, for every \(0\le\delta\le\delta_0\),
\[
        \mathbb E
        \exp\left\{
        \delta\sum_{i=1}^k \ell_{K_i}(\mathcal H(S))
        \right\}
        \le
        \exp\left\{
        C_{\mathrm{bd}}\sum_{i=1}^k |K_i|
        \right\}.
\]
\end{lem}

\begin{proof}
If \(S=\varnothing\), then \(\mathcal H(S)=\varnothing\) and the claim is
immediate.  Assume henceforth that \(S\ne\varnothing\).

For each \(x\in S\), construct an independent non-coalescing branching process
started from one particle at \(x\) at backward depth \(0\). Each particle rings
at rate one; at a ring it dies with probability \(1-\alpha\), and with
probability \(\alpha\) it is replaced by one child at each of the \(2d\)
neighbouring vertices. Let \(\widehat{\mathcal H}^x\) be the corresponding
space--time trace, and let \(\widehat{\mathcal H}^x_r\) be its spatial support
at depth \(r\), hence $\widehat{\mathcal H}^x_0=\{x\}$. We couple these branching processes with the backward history so
that
\[
        \mathcal H(S)\subseteq
        \widehat{\mathcal H}(S):=
        \bigcup_{x\in S}\widehat{\mathcal H}^x
\]
as space--time sets.
Let \(N_r^x\) be the number of particles descended from \(x\) at
depth \(r\), let \(Y_r^x\) be the number of spatial branch edges created by
depth \(r\), and set
\[
        Z_r^x:=\int_0^r N_u^x\,du .
\]
Define the one-particle branching-process occupation time and total spatial length by
\[
        \widehat L_x:=Z_T^x,
        \qquad
        \widehat\chi_x:=Y_T^x .
\]
For every fixed \(0<\eta<1\) and \(\Gamma<\infty\), after decreasing
\(\beta_0(d,q)\) by an amount depending on \(\eta,\Gamma\), the subcritical
branching process satisfies
\begin{equation}\label{eq:subBP-exp}
    M(\eta,\Gamma)
        :=
        \sup_{n\ge1,\,T\ge0,\,x\in\Lambda_n}
        \widehat{\mathbb E}_x
        \exp\{\eta\widehat L_x+\Gamma\widehat\chi_x\}
        <\infty .
\end{equation}
        
This is a standard result for a subcritical branching process; eg, see Lemma 3.1 of \cite{Exposition}. By decreasing $\beta_0$ if needed, we may assume $M(1/2,10)<\infty$.

For a space-time set $\mathcal A\subset\Lambda\times[0,T]$, write
\[
        \mathcal A_r:=\{x\in\Lambda:(x,T-r)\in \mathcal A   \},
        \qquad 0\le r\le T,
\]
where \(r\) is backward time. For each \(i\), set
\[
        C^{K_i}_r:=\{x\in\Lambda:\dist(x,K_i)\le c_{\rm cone}r\}.
\]
We also define the continuous barrier depth
\[
        g_i(\mathcal A):=
        \sup\{r\in[0,T]:\mathcal A_r\cap C^{K_i}_r\ne\varnothing\},
\]
with value \(0\) if the set is empty.
If \(g_i(\mathcal A)<T\), then \(\mathcal A\cap\mathsf D_h(K_i)=\varnothing\) for every
\(h>g_i(\mathcal A)\). Hence the infimum in the definition of \(\ell_{K_i}(\mathcal A)\) is at
most \(g_i(\mathcal A)\). If no \(h\le T\) makes \(\mathcal A\cap\mathsf D_h(K_i)\) empty, then
\(\ell_{K_i}(\mathcal A)=T\) and \(g_i(\mathcal A)=T\). Therefore, in all cases,
\[
        \ell_{K_i}(\mathcal A)\le 1+g_i(\mathcal A).
\]
Since \(k\le\sum_i |K_i|\), the factor coming from the additive \(1\) in
\(\ell_{K_i}(\mathcal A)\le1+g_i(\mathcal A)\) can be absorbed into the right-hand side. It is
therefore enough to estimate the exponential moment of \(\sum_i g_i\).

Since \(g_i\) is
increasing under inclusion of space-time sets,
\[
        \sum_{i=1}^k g_i(\mathcal H(S))
        \le
        \sum_{i=1}^k g_i(\widehat{\mathcal H}(S)).
\]
For a single process \(\widehat{\mathcal H}^x\), set
\[
        \widehat g_i^x
        :=
        g_i(\widehat{\mathcal H}^x),
        \qquad
        \widehat G_x:=\sum_{i=1}^k \widehat g_i^x .
\]
Then
\[
        \sum_{i=1}^k g_i(\mathcal H(S))
        =\sum_{i=1}^k \left(\sup_{x\in S}\widehat g_i^x\right)
        \le
        \sum_{x\in S}\widehat G_x .
\]
The random variables \((\widehat G_x)_{x\in S}\) are independent.
Let
\[
        I_x:=\#\{i:\widehat g_i^x>0\}.
\]
If \(I_x=0\), then \(\widehat G_x=0\). Suppose
\(I_x=m\ge1\). For every cone hit by \(\widehat{\mathcal H}^x\), choose one hit
point \((y_i,r_i)\), so that
\[
        y_i\in\widehat{\mathcal H}^x_{r_i},
        \qquad
        \dist(y_i,K_i)\le c_{\mathrm{cone}}r_i\le c_{\mathrm{cone}}T .
\]
For distinct hit cones \(i\ne j\),
\[
        \dist(y_i,y_j)
        \ge
        \operatorname{dist}(K_i,K_j)-2c_{\mathrm{cone}}T
        \ge R/2,
\]
because \(R\ge4c_{\mathrm{cone}}T\). The spatial projection of \(\widehat{\mathcal H}^x\) contains a connected graph joining
the \(m\) chosen sites. A connected graph joining \(m\) terminals whose mutual
distances are at least \(R/2\) has total length at least \(R(m-1)/4\). Indeed,
take a connected subgraph of minimal total length joining these terminals; it
is a tree. A depth-first traversal of this tree has total length twice the tree
length. If the terminals are recorded in the order in which they are first
visited by this traversal, then each passage from one newly visited terminal to
the next has length at least \(R/2\). Hence the traversal has length at least
\((m-1)R/2\), and the tree has length at least \(R(m-1)/4\). Since this tree is
contained in the spatial projection of \(\widehat{\mathcal H}^x\), its total length is at
most \(\widehat\chi_x\). Therefore $\widehat\chi_x
        \ge
        \frac R4(I_x-1)$.

Let
\[
        \widehat\tau_x:=
        \sup\{r\in[0,T]:\widehat{\mathcal H}^x_r\ne\varnothing\}.
\]
Then \(\widehat\tau_x\le\widehat L_x\). The largest positive
\(\widehat g_i^x\) is at most \(\widehat\tau_x\), while every other positive
\(\widehat g_i^x\) is at most \(T\). Therefore
\[
        \widehat G_x
        \le
        \widehat L_x+T(I_x-1)
        \le
        \widehat L_x+\frac{4T}{R}\widehat\chi_x .
\]

Next, if \(\widehat g_i^x>0\), then for some \(r\le\widehat\tau_x\) and some
\(y\in\widehat{\mathcal H}^x_r\),
\[
        \dist(y,K_i)\le c_{\mathrm{cone}}r.
\]
Thus
\[
        \dist(x,K_i)
        \le
        \dist(x,y)+\dist(y,K_i)
        \le
        \widehat\chi_x+c_{\mathrm{cone}}\widehat L_x .
\]
Consequently, for every \(\zeta>0\),
\[
        \mathbf 1_{\{\widehat g_i^x>0\}}
        \le
        e^{-\zeta\dist(x,K_i)}
        \exp\{\zeta\widehat\chi_x+\zeta c_{\mathrm{cone}}\widehat L_x\}.
\]

Choose \(\zeta>0\) and then \(\delta_0>0\) so small that, using
\(4T/R\le c_{\mathrm{cone}}^{-1}\), for every
\(0\le\delta\le\delta_0\),
\[
        \delta+c_{\mathrm{cone}}\zeta<1/2,
        \qquad
        \delta(4T/R)+\zeta<10 .
\]

Using
\[
        e^{\delta\widehat G_x}-1
        \le
        \mathbf 1_{\{I_x\ge1\}}
        \exp\{\delta\widehat L_x+\delta(4T/R)\widehat\chi_x\}
        \le
        \left(\sum_{i=1}^{k}\mathbf  1_{\{\widehat g_i^x>0\}}\right)
        \exp\{\delta\widehat L_x+\delta(4T/R)\widehat\chi_x\},
\]
we obtain
\[
\begin{aligned}
        e^{\delta\widehat G_x}-1
        &\le
        \sum_{i=1}^k
        e^{-\zeta\dist(x,K_i)}
        \exp\{(\delta+c_{\mathrm{cone}}\zeta)\widehat L_x
        +(\delta(4T/R)+\zeta)\widehat\chi_x\}.
\end{aligned}
\]
Taking expectation and using \eqref{eq:subBP-exp} with
\(\eta=1/2,\Gamma=10\) gives
\[
        \widehat{\mathbb E}_x(e^{\delta\widehat G_x}-1)
        \le
        M_\delta\sum_{i=1}^k e^{-\zeta\dist(x,K_i)},
\]
where $M_\delta
        :=
        M(\delta+c_{\mathrm{cone}}\zeta,\delta(4T/R)+\zeta)
        \leq M(1/2,10)<\infty$.
Hence
\[
        \widehat{\mathbb E}_x e^{\delta\widehat G_x}
        \le
        1+
        M_\delta\sum_{i=1}^k e^{-\zeta\dist(x,K_i)}.
\]
Using independence over \(x\in S\) and \(1+a\le e^a\),
\[
\begin{aligned}
        \mathbb E
        \exp\left\{
        \delta\sum_{i=1}^k g_i(\mathcal H(S))
        \right\}
        \le
        \prod_{x\in S}
        \left(
        1+M_\delta\sum_{i=1}^k e^{-\zeta\dist(x,K_i)}
        \right)                           \le
        \exp\left\{
        M_\delta\sum_{x\in\Lambda}\sum_{i=1}^k
        e^{-\zeta\dist(x,K_i)}
        \right\}.
\end{aligned}
\]
Finally, uniformly in the torus size,
\[
        \sum_{x\in\Lambda}e^{-\zeta\dist(x,K_i)}
        \le
        \sum_{z\in K_i}\sum_{x\in\Lambda}e^{-\zeta\dist(x,z)}
        \le
        C_{d,\zeta}|K_i|.
\]
Combining the last displays with
\(\ell_{K_i}(\mathcal H(S))\le1+g_i(\mathcal H(S))\) and
\(k\le\sum_i |K_i|\) proves the lemma.
\end{proof}

\paragraph{$R$-separated families of intersecting pairs.}
Consider a family \(\mathcal S\) of pairs \((A,B)\) with
\(\varnothing\ne A,B\subseteq\Lambda\), and define
\[
        \mathcal S^{(1)}:=\{A:(A,B)\in\mathcal S\},
        \qquad
        \mathcal S^{(2)}:=\{B:(A,B)\in\mathcal S\}.
\]
\(\mathcal S\) is called $R$-separated if, for any two distinct pairs
\((A,B),(A',B')\in\mathcal S\),
\[
        \dist(A,A')>R,
        \qquad
        \dist(B,B')>R .
\]

\begin{lemma}
\label{lem:separated-family-testing}
For every finite \(R\)-separated family \(\mathcal S\) of pairs of nonempty subsets of
\(\Lambda\),
\[
    \mathbb P\left(
        A\in\mathrm{SR}^{(1)},\,
        B\in\mathrm{SR}^{(2)}
        \text{ for every }(A,B)\in\mathcal S
        \,\middle|\,
        \mathcal K
    \right)
    \le
    \prod_{(A,B)\in\mathcal S}
        \Psi (A,\mathcal K)\Psi (B,\mathcal K).
\]
\end{lemma}

\begin{proof}
We first prove the one-copy bound. Fix \(a\in\{1,2\}\), and let
\(\mathcal A=\{A_1,\ldots,A_m\}\) be an \(R\)-separated family, ordered
deterministically.  We claim that
\[
\mathbb P
\left(
A_j\in\SR^{(a)}\text{ for all }1\le j\le m
\mid
\mathcal K
\right)
\le
\prod_{i=1}^m \Psi (A_i,\mathcal K).
\]

For a nonempty \(A\subseteq\Lambda\), write $\mathsf T_A^{(a)}:=\{A\in\SR^{(a)}\}\cup \{A\subseteq V_{\mathrm{Blue}}^{(a)}\}$. 
Expose the tests $\{A\in\SR^{(a)}\}$ in the chosen order.  Suppose the first \(i-1\) tests have
succeeded, and let \(\mathscr F_{i-1}\) be the sigma-field generated by
\(\mathcal K\) and all information revealed before testing \(A_i\).  If the revealed information is
not compatible with \(A_i\) being a Strong-Red cluster, then the next success probability is zero. 
Thus, the successful outcomes of the previous tests
\(A_j\in\SR^{(a)}\), \(j<i\), are measurable with respect to
$\sigma(\mathcal K,\mathsf T_{A_i}^{(a)},\mathcal H_{A_i}^{-,(a)})$. 

Write $G_i:=\{A_j\in\SR^{(a)}\text{ for all }1\leq j\leq i\}$.
Observe that since \(\{A_i\in\SR^{(a)}\}\subseteq \mathsf T_{A_i}^{(a)}\), we have
\[
\begin{aligned}
\mathbb P(G_i\mid \mathcal K)
&=
\mathbb E\!\left[
\mathbf 1_{G_{i-1}}\mathbf 1_{\{A_i\in\SR^{(a)}\}}
\,\middle|\,
\mathcal K
\right] \\
&=
\mathbb E\!\left[
\mathbf 1_{G_{i-1}}\mathbf 1_{\mathsf T_{A_i}^{(a)}}
\mathbb P(A_i\in\SR^{(a)}\mid\mathcal K,\mathcal H_{A_i}^{-,(a)},\mathsf T_{A_i}^{(a)})
\,\middle|\,
\mathcal K
\right] \\
&\le
\Psi (A_i,\mathcal K)\,
\mathbb E\!\left[
\mathbf 1_{G_{i-1}}\mathbf 1_{\mathsf T_{A_i}^{(a)}}
\,\middle|\,
\mathcal K
\right]
\le
\Psi (A_i,\mathcal K)\,
\mathbb P(G_{i-1}\mid\mathcal K).
\end{aligned}
\]
Iterating over $i=1,2,\cdots,m$ proves the one-copy bound.

Now return to the pair family \(\mathcal S\).  Conditional on \(\mathcal K\),
the two copies are independent.  Therefore
\[
\begin{aligned}
&\mathbb P\left(
A\in\mathrm{SR}^{(1)},\,
B\in\mathrm{SR}^{(2)}
\text{ for all }(A,B)\in\mathcal S
\,\middle|\,
\mathcal K
\right)\\
&\qquad =
\mathbb P
\left(A\in\mathrm{SR}^{(1)}\text{ for all }A\in\mathcal S^{(1)}\mid
\mathcal K \right)
\mathbb P
\left(B\in\mathrm{SR}^{(2)}\text{ for all }B\in\mathcal S^{(2)}\mid\mathcal K\right)\\
&\qquad\le
\prod_{A\in\mathcal S^{(1)}}\Psi (A,\mathcal K)
\prod_{B\in\mathcal S^{(2)}}\Psi (B,\mathcal K).
\end{aligned}
\]
Since compatibility makes both coordinate projections injective, this is exactly
\[
\prod_{(A,B)\in\mathcal S}\Psi (A,\mathcal K)\Psi (B,\mathcal K).
\]
\end{proof}

\begin{lemma}
\label{lem:weighted-activity}
For every sufficiently large \(\lambda_{\rm SR}>0\), after decreasing \(\beta\) if necessary, there
exist constants \(C<\infty\) such that the following holds.
Define
\[
    r(A)
    :=
    C e^{-(1+\kappa)T}e^{-\lambda_{\rm SR} W(A)/2}.
\]
Then, for every \(R\)-separated family \(\mathcal S\) of pairs of nonempty subsets
of \(\Lambda\),
\[
    \mathbb E\left[
        \prod_{A\in \mathcal{S}^{(1)}}q^{|A|}\Psi (A,\mathcal K)
        \prod_{B\in \mathcal{S}^{(2)}}q^{|B|}\Psi (B,\mathcal K)
    \right]
    \le
    \prod_{A\in\mathcal{S}^{(1)}} r(A)
    \prod_{B\in \mathcal{S}^{(2)}} r(B) .
\]
\end{lemma}
\begin{proof}
By \eqref{eq:activity-pointwise},
\[
\begin{aligned}
&\prod_{A\in \mathcal{S}^{(1)}}q^{|A|}\Psi (A,\mathcal K)
        \prod_{B\in \mathcal{S}^{(2)}}q^{|B|}\Psi (B,\mathcal K)       \\
&\qquad\le
        C_{\mathrm{sr}}^{|\mathcal S^{(1)}|+|\mathcal S^{(2)}|}
        e^{-(1+\kappa)T(|\mathcal S^{(1)}|+|\mathcal S^{(2)}|)}
        e^{-\lambda_{\rm SR}\sum_{A\in\mathcal S^{(1)}}W(A)}
        e^{-\lambda_{\rm SR}\sum_{B\in\mathcal S^{(2)}}W(B)}       \\
&\qquad\quad\times
        \exp\left\{
        c_{\mathrm{sr}}\beta
        \sum_{A\in\mathcal S^{(1)}}\ell_A(\mathcal H_{G,P})
        +
        c_{\mathrm{sr}}\beta
        \sum_{B\in\mathcal S^{(2)}}\ell_B(\mathcal H_{G,P})
        \right\}.
\end{aligned}
\]
By Cauchy--Schwarz:
\[
\begin{aligned}
&\mathbb E
\exp\left\{
c_{\mathrm{sr}}\beta
\sum_{A\in\mathcal S^{(1)}}\ell_A(\mathcal H_{G,P})
+
c_{\mathrm{sr}}\beta
\sum_{B\in\mathcal S^{(2)}}\ell_B(\mathcal H_{G,P})
\right\}                                      \\
&\qquad\le
\left(
\mathbb E
\exp\left\{
2c_{\mathrm{sr}}\beta
\sum_{A\in\mathcal S^{(1)}}\ell_A(\mathcal H_{G,P})
\right\}
\right)^{1/2}
\left(
\mathbb E
\exp\left\{
2c_{\mathrm{sr}}\beta
\sum_{B\in\mathcal S^{(2)}}\ell_B(\mathcal H_{G,P})
\right\}
\right)^{1/2}.
\end{aligned}
\]
Since \(\mathcal H_{G,P}\subseteq \mathcal H(\Lambda)\), monotonicity gives
\[
        \ell_A(\mathcal H_{G,P})\le \ell_A(\mathcal H(\Lambda)).
\]
Taking \(\beta_0\) small enough so that \(2c_{\mathrm{sr}}\beta\le\delta_0\),
Lemma~\ref{lem:sum-h} gives
\[
\begin{aligned}
&\mathbb E
\exp\left\{
c_{\mathrm{sr}}\beta
\sum_{A\in\mathcal S^{(1)}}\ell_A(\mathcal H_{G,P})
+
c_{\mathrm{sr}}\beta
\sum_{B\in\mathcal S^{(2)}}\ell_B(\mathcal H_{G,P})
\right\}                               \le
        \exp\left\{
        \frac{C_{\mathrm{bd}}}{2}
        \sum_{A\in\mathcal S^{(1)}}|A|
        +
        \frac{C_{\mathrm{bd}}}{2}
        \sum_{B\in\mathcal S^{(2)}}|B|
        \right\}.
\end{aligned}
\]
Choosing \(\lambda_{\rm SR}\) sufficiently large, and using \(W(A)\ge |A|\), the last
exponential is absorbed into
\(e^{-\lambda_{\rm SR} W(\cdot)/2}\).  This gives the claimed bound, after increasing
the $C$.
\end{proof}

\begin{prop}[Strong--Red overlap]
\label{prop:strong-red-overlap}
After choosing \(\lambda_{\rm SR}>0\) sufficiently large and decreasing
\(\beta_0(d,q)\) if necessary, there exists \(C<\infty\) such that, at
\(T=t_\star+s\) with \(s\ge0\),
\[
        \mathbb E\left[
        q^{|V_{\mathrm{SR}}^{(1)}\cap V_{\mathrm{SR}}^{(2)}|}-1
        \right]
        \le
        C e^{-2(1+\kappa)s}.
\]
\end{prop}

\begin{proof}
We use the same thinning idea as in Proposition~\ref{prop:purple}. Here the
testing costs depend on the revealed sigma-field \(\mathcal K\), so
Lemmas~\ref{lem:separated-family-testing} and~\ref{lem:weighted-activity}
are used together.

Let
\[
\mathcal Q:=
\left\{
(A,A'):\varnothing\ne A,A'\subseteq\Lambda,\ A\cap A'\ne\varnothing
\right\},
\]
and fix an ordering of this finite set. For each realization, scan the pairs
in this order and retain \((A,A')\) if
\(A\in\mathrm{SR}^{(1)}\), \(A'\in\mathrm{SR}^{(2)}\), and neither set has
appeared in an earlier retained pair. Denote the resulting collection by
\(\mathcal M\). Distinct Strong--Red top sets in the same copy are at
distance greater than \(R\), so \(\mathcal M\) is \(R\)-separated.

The retained pairs still cover the Strong--Red overlap. Indeed, every
\(x\in V_{\mathrm{SR}}^{(1)}\cap V_{\mathrm{SR}}^{(2)}\) belongs to unique
top sets \(A\in\mathrm{SR}^{(1)}\) and \(A'\in\mathrm{SR}^{(2)}\). If
\((A,A')\) was not retained, then one of these sets already appeared in a
retained pair and therefore contains \(x\). Thus
\[
\left|V_{\mathrm{SR}}^{(1)}\cap V_{\mathrm{SR}}^{(2)}\right|
\le
\sum_{(A,A')\in\mathcal M}\bigl(|A|+|A'|\bigr).
\]

Expanding the resulting product and taking expectations gives
\[
\begin{aligned}
\mathbb E\left[
q^{|V_{\mathrm{SR}}^{(1)}\cap V_{\mathrm{SR}}^{(2)}|}-1
\right]
\le
\sum_{\substack{\varnothing\ne\mathcal S\subseteq\mathcal Q\\
                 \mathcal S\ R\text{-separated}}}
\left(
\prod_{(A,A')\in\mathcal S}\left(q^{|A|+|A'|}-1\right)
\right)
\mathbb P(\mathcal S\subseteq\mathcal M).
\end{aligned}
\]
For a fixed \(R\)-separated family \(\mathcal S\), its two coordinate
projections are injective, and the event \(\mathcal S\subseteq\mathcal M\)
requires all of its top sets to occur. The tower property and
Lemma~\ref{lem:separated-family-testing} therefore give
\[
\mathbb P(\mathcal S\subseteq\mathcal M)
\le
\mathbb E\left[
\prod_{(A,A')\in\mathcal S}
\Psi (A,\mathcal K)\Psi (A',\mathcal K)
\right].
\]
Since \(q^{|A|+|A'|}-1\le q^{|A|+|A'|}\),
Lemma~\ref{lem:weighted-activity} yields
\[
\left(
\prod_{(A,A')\in\mathcal S}\left(q^{|A|+|A'|}-1\right)
\right)
\mathbb P(\mathcal S\subseteq\mathcal M)
\le
\prod_{(A,A')\in\mathcal S}r(A)r(A').
\]
Dropping the separation restriction and summing over all subsets of
\(\mathcal Q\), we obtain
\[
\mathbb E\left[
q^{|V_{\mathrm{SR}}^{(1)}\cap V_{\mathrm{SR}}^{(2)}|}-1
\right]
\le e^\rho-1,
\qquad
\rho:=
\sum_{(A,A')\in\mathcal Q}r(A)r(A').
\]

An intersecting pair can be rooted at any common vertex, so $\rho
\le
\sum_{v\in\Lambda}
(
\sum_{{\varnothing\ne A\subseteq\Lambda,\,v\in A}}r(A)
)^2.$
By the lattice-animal entropy bound, there is
\(C_{\mathrm{ent}}<\infty\) such that
\[
\#\{\varnothing\ne A\subseteq\Lambda:v\in A,\ W(A)=k\}
\le C_{\mathrm{ent}}^k.
\]
Taking \(\lambda_{\rm SR}\) sufficiently large, the definition of \(r(A)\) therefore
gives, uniformly in \(v\),
\[
\sum_{\substack{\varnothing\ne A\subseteq\Lambda\\v\in A}}r(A)
\le
C e^{-(1+\kappa)T}
\sum_{k\ge1}C_{\mathrm{ent}}^ke^{-\lambda_{\rm SR} k/2}
\le
C e^{-(1+\kappa)T}.
\]
Consequently, since
\(|\Lambda|e^{-2(1+\kappa)t_\star}=1\),
$\rho
\le
C|\Lambda|e^{-2(1+\kappa)T}
=
C e^{-2(1+\kappa)s}.$
As \(s\ge0\), the exponent is uniformly bounded. Since
\(e^\rho-1\le\rho e^\rho\), it follows that
\[
\mathbb E\left[
q^{|V_{\mathrm{SR}}^{(1)}\cap V_{\mathrm{SR}}^{(2)}|}-1
\right]
\le
C e^{-2(1+\kappa)s}.
\]
\end{proof}

\paragraph{Proof of Proposition~\ref{prop:one-block-SR}.}
The proof combines a geometric estimate with an exponential-moment computation along the cone trials of Subsection~\ref{subsec:yellowSR}. Lemma~\ref{lem:stopped-weighted-red} bounds the exponential moment of a stopped functional of the backward history via the dominating non-coalescing branching envelope. Lemma~\ref{lem:no-admissible-delay} shows the exploration cannot run long without an admissible cone trial, which together with the previous bound yields
Lemma~\ref{lem:first-h-admissible} on the depth of the first such trial. Since a trial can fail,
Lemma~\ref{lem:one-site-continuation} shows the sequence of trials contracts geometrically, bounding the total weight accumulated over all of them. The proposition then follows by expressing \(\{A\in\mathrm{SR}\}\) via some trial succeeding and being declared Strong Red, and assembling these four estimates.

Fix an admissible outside history \(X\) for \(A\).  For each \(u\in A\), define
\[
        s_u=s_u(X):=
        \sup\{t<T:(u,t)\in X\},
\]
with the convention \(s_u=-\infty\) if the set is empty. Let
\[
        A_{\mathrm{top}}
        :=
        \{u\in A:s_u>T-1\}.
\]
Define the top-unit event
\begin{equation}\label{eq:top-unit-U}
\mathcal U_A=\mathcal U_A(X)
        :=
        \bigcap_{u\in A_{\mathrm{top}}}
        \left\{
        \text{there is at least one update at }u
        \text{ in }(s_u,T]
        \right\}.
\end{equation}
The event \(\mathcal U_A(X)\) depends only on graphical marks in the top unit
slab at sites in \(A\).  In statements where \(X\) is fixed we suppress the
dependence on \(X\).
For backward depth \(u\), define the time-reversed actual-history slice by
\[
        \widehat{\mathcal H}_A(u)
        :=
        \mathcal H_A(T-u)
        =\{x\in\Lambda:(x,T-u)\in\mathcal H_A\},
        \qquad 0\le u\le T.
\]
Let \((\widehat{\mathscr F}_u^A)_{0\le u\le T}\) be the corresponding
right-continuous depth filtration
generated by the graphical marks of the \(A\)-exploration revealed in the slab
\(\Lambda\times[T-u,T]\).
Here and below, the hat on \(\widehat{\mathcal H}_A\) and
\(\widehat{\mathscr F}^A\) denotes time reversal; the hatted processes with a
superscript \(x\) introduced earlier denote the non-coalescing branching
envelope.

We use the standard non-coalescing branching envelope. Start one particle from each site of \(A\). Each particle carries an independent rate-one clock. At a ring it is killed with probability \(1-\alpha\), and with probability \(\alpha\) it is replaced by the \(D=2d\) neighbouring sites. Two particles may occupy the same site. This process can be coupled so that
the spatial support of the actual backward history is contained in the support of the non-coalescing branching process at every depth.

\begin{lemma}
\label{lem:stopped-weighted-red}
Fix \(\lambda_0>0\). There exist
constants \(C<\infty\), \(K<\infty\), and \(\theta_0\in(0,1)\), depending only on
\(d,q,\lambda_0\), such that the following holds.

Let \(A\subset\Lambda\) be finite and nonempty, let \(X\) be a compatible outside history with \(h=\ell_A(X)\). Let \(\sigma\) be a \([0,T]\cup\{\infty\}\)-valued stopping time for \((\widehat{\mathscr F}_u^A)\). Construct the non-coalescing branching envelope on the same probability space, and assume that \(\sigma\) is also a stopping time for its filtration.
Assume that, on \(\{\sigma<\infty\}\), one has \(\sigma\ge h\).

Let
\[
        L_s:=\int_0^s |\widehat{\mathcal H}_A(u)|\,du,
\]
and let \(\chi_s\) be the total spatial length of all spatial edges in the actual history up to backward time $s$. Let \(\mathcal C_s\) be the event that \(\widehat{\mathcal H}_A(s)\ne\varnothing\) and the truncated space--time history $\mathcal H_A\cap(\Lambda\times[T-s,T])$ is connected.

Then for every \(0<\theta\le\theta_0\), with \(\gamma=c_{\mathrm{cone}}^{-1}\theta\), there exists $\beta_0(\theta)$ such that, for all $\beta<\beta_0$, for every \(a\in(0,1-K\beta)\),
\begin{equation*}
\mathbb E
\left[
\mathbf 1_{\{\sigma<\infty\}}\mathbf 1_{\mathcal C_\sigma}
e^{a\sigma}
e^{\theta(L_\sigma-\sigma)+\gamma\chi_\sigma}
\,\middle|\,
\mathcal U_A(X)
\right]
\le
C e^{-\lambda_0 W(A)}
e^{-(1-K\beta-a)h}.
\end{equation*}

\end{lemma}

\begin{proof}
Let \(D:=2d\). We use the standard non-coalescing branching
domination. Let \(N_t\) be the number of labelled particles at backward depth
\(t\), let \(Y_t\) be the number of spatial branch edges created by depth
\(t\), and set
\[
        Z_t:=\int_0^t N_s\,ds .
\]
Then
\[
        |\widehat{\mathcal H}_A(t)|\le N_t,\qquad
        L_t\le Z_t,\qquad
        \chi_t\le Y_t .
\]

Set \(\Phi(n):=e^n-1\). There exist constants \(c_\Phi>0\) and
\(C_\Phi<\infty\), depending only on \(d\), such that, for all \(n\ge1\),
\begin{equation}\label{eq:Psi}
        n(\Phi(n)-\Phi(n-1))-\Phi(n)\ge c_\Phi(n-1)\Phi(n),
        \qquad
        \Phi(n+D-1)\le C_\Phi\Phi(n).
\end{equation}

\smallskip
\noindent\textbf{Step 1: the top-unit bound.}
We first record a one-unit bound under the conditioning \(\mathcal U_A(X)\).
For every fixed finite \(\bar R\), after decreasing
\(\beta_0(d,q,\bar R)\), there is \(C_*<\infty\), independent of
\(A,X,T,\bar R\), such that
\begin{equation}\label{eq:Q-ini}
\mathbb E\left[
        e^{\theta Z_1+\bar RY_1}\Phi(N_1)
        \,\middle|\,
        \mathcal U_A(X)
\right]
\le e^{C_*|A|}
\end{equation}
uniformly in \(0\le\theta\le1\).

Indeed, first ignore the conditioning. For the labelled process started from
one particle, define
\[
        F_t:=
        \mathbb E_1
        \exp\left\{\theta Z_t+\bar RY_t+N_t\right\},
        \qquad 0\le t\le1 .
\]
Let $Q_t:=\exp\left\{\theta Z_t+\bar RY_t+N_t\right\}$.
The derivative of $F_t$ collects the first-order changes coming from the three parts: $\theta Z_t,\, \bar RY_t,\, N_t$. 
For the continuous growth of \(Z_t\), since \(dZ_t=N_t\,dt\), this gives $\theta N_tQ_t$. 
For the changes of \(Y_t\) and \(N_t\), when an oblivious update happens, \(N_t\) jumps down by \(1\), so \(Q_t\) is multiplied by \(e^{-1}\). Thus the contribution is $(1-\alpha)N_t(e^{-1}-1)Q_t.$ 
When a non-oblivious update happens, \(Y_t\) increases by \(D\), and \(N_t\) increases by \(D-1\). So \(Q_t\) is multiplied by $e^{\bar RD+D-1}$, giving contribution $\alpha N_t(e^{\bar RD+D-1}-1)Q_t.$
Therefore, since $\theta\in(0,1)$,

\[
        \partial_tF_t
        =
        (\theta-1)F_t
        +(1-\alpha)
        +\alpha e^{\bar R D}F_t^D
        \le
        1+\alpha e^{\bar R D}F_t^D ,
        \qquad
        F_0=e .
\]
Set \(B:=e+2\). After decreasing \(\beta_0(d,q,\bar R)\), we may
assume
\begin{equation}\label{eq:beta-dec}
    \alpha e^{\bar R D}B^D\le1
\end{equation}
Then, as long as \(F_t\le B\), one has \(\partial_tF_t\le2\), and hence
\(F_t\le e+2\leq B\) for \(0\le t\le1\). By continuity,
\[
        \sup_{0\le t\le1}F_t\le B .
\]
Starting from one particle at each site of \(A\), we get
\begin{equation}\label{eq:indep}
        \mathbb E_A e^{\theta Z_1+\bar RY_1}\Phi(N_1)
        \le
        \mathbb E_A e^{\theta Z_1+\bar RY_1+N_1}
        \le B^{|A|}.
\end{equation}
We now pass from the unconditioned top slab to the top slab conditioned on
\(\mathcal U_A(X)\). For each \(x\in A_{\rm top}\), the event
\(\mathcal U_A(X)\) requires at least one clock ring in \(I_x=(s_x,T]\). A
Poisson clock conditioned to have at least one ring in \(I_x\) is stochastically
dominated by an unconditioned clock on \(I_x\) together with one extra forced
ring in \(I_x\). 
An oblivious forced ring does not enlarge the process. A non-oblivious
forced ring has probability \(\alpha\), creates at most \(D\) descendants and adds at most \(D\) spatial edges. Its remaining duration is at most one, so \eqref{eq:indep} bounds its future multiplicative cost by \(e^{\bar R D}B^D\). Hence one
forced ring costs at most $1+\alpha e^{\bar R D}B^D$, which is at most 2 from \eqref{eq:beta-dec}. Since
there are at most \(|A_{\rm top}|\le |A|\) forced rings, \eqref{eq:Q-ini} holds with $C_*:=\log (2B)=\log(2(e+2))$.

\smallskip
\noindent\textbf{Step 2: the supermartingale construction.}
We now prove the desired bound in the Lemma. Put
\[
        \rho:=\lambda_0+C_*+2 .
\]
Choose \(\theta_0>0\) small, then set
\[
        \bar R:=\rho+c_{\mathrm{cone}}^{-1}\theta_0 .
\]
Choose \(K<\infty\) large enough, and then decrease \(\beta_0\), so that for
every \(0<\theta\le\theta_0\) and every \(0\le\zeta\le\bar R\),
\[
        Q_t(\zeta):=
        e^{(1-K\beta)t}
        e^{\theta(Z_t-t)+\zeta Y_t}\Phi(N_t)
\]
is a nonnegative supermartingale until extinction.

To check this, suppose \(N_t=n\ge1\). The drift of \(Q_t(\zeta)\), divided by
the common exponential factor, equals
\[
(1-K\beta-\theta+\theta n)\Phi(n)
+n(1-\alpha)(\Phi(n-1)-\Phi(n))
+n\alpha\bigl(e^{\zeta D}\Phi(n+D-1)-\Phi(n)\bigr).
\]
Indeed, between clock rings, \(N_t\) and \(Y_t\) are constant, while $\frac{d}{dt}(Z_t-t)=N_t-1=n-1$. Thus the factor \(e^{(1-K\beta)t}e^{\theta(Z_t-t)}\) has logarithmic derivative $(1-K\beta)+\theta(n-1)=1-K\beta-\theta+\theta n$, which contributes $(1-K\beta-\theta+\theta n)\Phi(n).$
The remaining terms come from jumps. Since there are \(n\) active particles, clock rings occur at total rate \(n\).
At such a ring, with probability \(1-\alpha\), the ringing particle dies; then \(N_t:n\mapsto n-1\) and \(Y_t\) does not change, giving $n(1-\alpha)(\Phi(n-1)-\Phi(n))$.
With probability \(\alpha\), the ringing particle branches to \(D\) children; then \(N_t:n\mapsto n+D-1\), and \(Y_t\) increases by \(D\), so the weight \(e^{\zeta Y_t}\) gains a factor \(e^{\zeta D}\). This gives $n\alpha(e^{\zeta D}\Phi(n+D-1)-\Phi(n))$.
Using \eqref{eq:Psi} and \(\zeta\le\bar R\), and writing
\(B_{\bar R}:=C_\Phi e^{\bar R D}\), this is at most
\[
        \bigl[-c_\Phi(n-1)+\theta(n-1)+\alpha n B_{\bar R}-K\beta\bigr]\Phi(n)
        =
        \bigl[-(c_\Phi-\theta-\alpha B_{\bar R})(n-1)
        +\alpha B_{\bar R}-K\beta\bigr]\Phi(n).
\]
Since \(\alpha(\beta)=O(\beta)\), we first choose
\(\theta_0<c_\Phi/2\), then choose \(K\) large enough, and finally decrease
\(\beta_0\), so that the last display is nonpositive for every \(n\ge1\).

\smallskip
\noindent\textbf{Step 3: optional stopping and conclusion.}
Since \(\sigma\ge h\geq 1\) on \(\{\sigma<\infty\}\), we may condition on the dominating process up to depth \(1\).  Given this top-slab information, the graphical marks below real time \(T-1\) are independent of \(\mathcal U_A(X)\), and \(\sigma-1\) is a stopping time for the shifted labelled process. Applying
optional stopping to the shifted supermartingale for $Q$ followed by \(\Phi(N_\sigma)\ge\Phi(1)\mathbf 1_{\{N_\sigma\ge1\}}\) gives

\begin{equation}\label{eq:OST}
\mathbb E\left[
        \mathbf 1_{\{\sigma<\infty\}}
        \mathbf 1_{\{N_\sigma\ge1\}}
        e^{(1-K\beta)\sigma+\theta(Z_\sigma-\sigma)+\zeta Y_\sigma}
        \,\middle|\,
        \mathcal U_A(X)
\right] \le
C\,
\mathbb E\left[
        e^{\theta Z_1+\zeta Y_1}\Phi(N_1)
        \,\middle|\,
        \mathcal U_A(X)
\right]
\le
C e^{C_*|A|},
\end{equation}
one may take $C=e/(e-1)$, and the last step uses \eqref{eq:Q-ini}.

On \(\mathcal C_\sigma\), the spatial projection of the truncated history contains a connected graph spanning \(A\).
Thus $\chi_\sigma\ge W(A)-1$. 
Also, on \(\{\sigma<\infty\}\), \(\sigma\ge h\). Hence, for \(a\in(0,1-K\beta)\),
\[
        e^{a\sigma}e^{\gamma\chi_\sigma}
        \le
        e^{-(1-K\beta-a)h}
        e^{(1-K\beta)\sigma}
        e^{-\rho(W(A)-1)}
        e^{(\rho+\gamma)\chi_\sigma}.
\]
Applying \eqref{eq:OST} with \(\zeta=\rho+\gamma\), we obtain
\[
\begin{aligned}
&\mathbb E
\left[
\mathbf 1_{\{\sigma<\infty\}}\mathbf 1_{\mathcal C_\sigma}
e^{a\sigma}
e^{\theta(L_\sigma-\sigma)+\gamma\chi_\sigma}
\,\middle|\,
\mathcal U_A(X)
\right] \le
C e^{-(1-K\beta-a)h}
e^{C_*|A|-\rho(W(A)-1)} .
\end{aligned}
\]
Finally, \(|A|\le W(A)\) and \(\rho=\lambda_0+C_*+2\) imply
\[
        e^{C_*|A|-\rho(W(A)-1)}
        \le C e^{-\lambda_0 W(A)} .
\]
This proves the desired bound.
\end{proof}

Continue to work with the fixed outside transcript
\(\mathcal H_A^-=X\), and put \(h:=\ell_A(X)\).
A backward depth \(r\in[h,T]\) is called
\emph{\((A,h)\)-admissible} if there exists \(v\in\Lambda\) such that
\[
    \widehat{\mathcal H}_A(r)=\{v\},
    \qquad
    \mathsf C_r(v)\subset\mathsf D_h(A).
\]
In this case, denote the unique vertex by \(v(r)\).

For \(0\le s\le T\), define
\[
    D_s
    :=
    L_s-s
    =
    \int_0^s\bigl(|\widehat{\mathcal H}_A(u)|-1\bigr)\,du.
\]

\begin{lemma}[Deterministic delay without an admissible singleton]
\label{lem:no-admissible-delay}
Let \(A\subset\Lambda\) be finite and nonempty, and fix
\(0\le h\le r\le T\). Assume that
\(\widehat{\mathcal H}_A(u)\ne\varnothing\) for every \(u\in[0,r]\)
and that no \((A,h)\)-admissible depth occurs in \([h,r)\).
Then
\[
    r-h
    \le
    D_r+c_{\rm cone}^{-1}\chi_r.
\]
\end{lemma}

\begin{proof}
For \(u\in[h,r)\), absence of an \((A,h)\)-admissible depth implies that either
\(|\widehat{\mathcal H}_A(u)|\ge2\), or \(\widehat{\mathcal H}_A(u)=\{v_u\}\) and the depth cone
\(\mathsf C_u(v_u)\) is not contained in \({\mathsf D}_h(A)\).  In the
second case, using the definition of \({\mathsf D}_h(A)\), one has
\[
        \operatorname{dist}(v_u,A)>c_{\rm cone}u .
\]
Whenever \(\widehat{\mathcal H}_A(u)=\{v_u\}\), the truncated history connects \(v_u\) to at least
one vertex of \(A\).  Therefore the spatial length accumulated by depth \(u\)
is at least \(\operatorname{dist}(v_u,A)\).  Hence every singleton
cone-failing depth \(u\) satisfies \(u<c_{\rm cone}^{-1}\chi_r\).  The set of
such depths has Lebesgue measure at most \(c_{\rm cone}^{-1}\chi_r\).
The remaining depths in \([h,r)\) have \(|\widehat{\mathcal H}_A(u)|\ge2\), and their Lebesgue
measure is at most \(D_r\). Adding the two contributions gives the displayed
bound.
\end{proof}

Fix a cemetery vertex \(v_\dagger\in A\).
Set \(e_{A,0}:=h\). If \(e_{A,0}\ge T\), the exploration stops without a
cone trial. At stage \(k\), let \(\mathcal R_k\) be the
event that the stage is reached and an \((A,h)\)-admissible depth
exists in \([e_{A,k-1},T]\). On \(\mathcal R_k\), set
\[
    r_{A,k}
    :=
    \inf\left\{
        r\in[e_{A,k-1},T]:
        r\text{ is \((A,h)\)-admissible}
    \right\},
    \quad
    \tau_{A,k}:=T-r_{A,k},
    \quad
    v_{A,k}:=v(r_{A,k})=v_{A,\tau_{A,k}}.
\]
On \(\mathcal R_k^c\), the exploration stops. On this event, and for all
stages after the exploration stops, use the
finite cemetery values
\[
        (r_{A,k},\tau_{A,k},v_{A,k}):=(T,0,v_\dagger).
\]
Thus \(r_{A,k}\) is the trial's backward depth and \(\tau_{A,k}\) is its
absolute tip time. On \(\mathcal R_k\), set
\[
         \mathsf C_{A,k}
         :=\mathsf C_{r_{A,k}}(v_{A,k})
         =\mathsf C_{T-\tau_{A,k}}(v_{A,k}).
\]
With the natural convention \(\mathsf{Bot}_0(v):=\Omega\), the cone test at
stage \(k\), on \(\mathcal R_k\), is the event
\[
        \mathsf{Bot}_{\tau_{A,k}}(v_{A,k}).
\]
If this event
occurs, the cone test succeeds and the sequential exploration stops. If it
does not occur, either the lower history becomes empty while still inside the
cone, in which case the exploration stops and \(A\) is not Red, or the lower
history exits \(\mathsf C_{A,k}\) before reaching time \(0\). In the latter
case, let \(e_{A,k}\) be the backward depth of the first exit. It is an
\((\widehat{\mathscr F}_u^A)\)-stopping time. We reveal the lower history only up to
depth \(e_{A,k}\), and then continue to stage \(k+1\).
Each successful trial is a valid separation candidate in the notation of
Subsection~\ref{subsec:yellowSR}.

\begin{lemma}[First admissible cone estimate]
\label{lem:first-h-admissible}
Fix \(\lambda_0>0\). After decreasing
\(\beta_0(d,q)\), there are constants \(C,c<\infty\) such that the following
holds.

Fix \(h:=\ell_A(X)\). On \(\mathcal R_1\), let \(r_{A,1}\in[h,T]\) be the
first \((A,h)\)-admissible backward depth in the sequential exploration above. Let
\(\mathcal N_0\) be the event that \(A\) is Red and \(\mathcal R_1^c\) occurs.
\[
        \mathbb E
        \left[
            e^{(1+\kappa)r_{A,1}}\mathbf 1_{\mathcal R_1}
            +
            e^{(1+\kappa)T}\mathbf 1_{\mathcal N_0}
            \,\middle|\,
            \mathcal U_A
        \right]
        \le
        C e^{-\lambda_0 W(A)}e^{c\beta h}.
\]
\end{lemma}

\begin{proof}
First consider \(\mathcal N_0\). On \(\mathcal N_0\), no \((A,h)\)-admissible depth
occurs in \([h,T)\), and the history survives to depth \(T\).
Lemma~\ref{lem:no-admissible-delay}, applied with \(r=T\), gives
\[
        T-h
        \le
        D_T+c_{\rm cone}^{-1}\chi_T.
\]
Thus, for any $\theta\geq 0$,
\[
        \mathbf 1_{\mathcal N_0}
        \le
        \mathbf 1_{\{\widehat{\mathcal H}_A(T)\ne\varnothing\}}
        e^{-\theta(T-h)}
        e^{\theta(L_T-T)+c_{\rm cone}^{-1}\theta\chi_T}.
\]
Applying Lemma~\ref{lem:stopped-weighted-red} with
\[
        \sigma=T,
        \qquad
        a=1+\kappa-\theta,
        \qquad
        \gamma=c_{\rm cone}^{-1}\theta,
\]
and choosing \(\theta>0\) fixed and then \(\beta_0\) small so that
\[
        1+\kappa-\theta<1-K\beta,
\]
gives
\begin{equation}\label{eq:N0}
    e^{(1+\kappa) T}
        \mathbb P(\mathcal N_0\mid\mathcal U_A)
        \le
        C e^{-\lambda_0W(A)}e^{(K\beta+\kappa)h}.
\end{equation}
We can take $c:=K$ because $\kappa<0$.

On \(\mathcal R_1\), no \((A,h)\)-admissible depth occurs in \([h,r_{A,1})\), and the
history is nonempty up to depth \(r_{A,1}\). By
Lemma~\ref{lem:no-admissible-delay},
\[
        r_{A,1}-h
        \le
        D_{r_{A,1}}+c_{\rm cone}^{-1}\chi_{r_{A,1}}.
\]
Therefore
\[
\begin{aligned}
e^{(1+\kappa)r_{A,1}}\mathbf 1_{\mathcal R_1}
&\le
e^{\theta h}
\mathbf 1_{\mathcal R_1}
e^{(1+\kappa-\theta)r_{A,1}}
e^{\theta D_{r_{A,1}}+c_{\rm cone}^{-1}\theta\chi_{r_{A,1}}}                      \\
&=
e^{\theta h }
\mathbf 1_{\mathcal R_1}
e^{(1+\kappa-\theta)r_{A,1}}
e^{\theta(L_{r_{A,1}}-r_{A,1})+c_{\rm cone}^{-1}\theta\chi_{r_{A,1}}} .
\end{aligned}
\]
On \(\mathcal R_1\), at depth \(r_{A,1}\), the history is a singleton. Therefore
\(\mathcal C_{r_{A,1}}\) occurs because all branches from \(A\times\{T\}\) have
merged into the single space--time point at depth \(r_{A,1}\). Define the
first-trial depth stopped at infinity by
\[
        \sigma_1
        :=
        \begin{cases}
        r_{A,1},&\mathcal R_1\text{ occurs},\\
        \infty,&\mathcal R_1^c\text{ occurs},
        \end{cases}
\]
is an \((\widehat{\mathscr F}_u^A)\)-stopping time: the event \(\{\sigma_1\le r\}\) is
determined by the history revealed up to depth \(r\), using the right-continuous
convention for the singleton slice. Applying
Lemma~\ref{lem:stopped-weighted-red} with
\[
        \sigma=\sigma_1,
        \qquad
        a=1+\kappa-\theta,
        \qquad
        \gamma=c_{\rm cone}^{-1}\theta,
\]
gives
\[
\mathbb E
\left[
e^{(1+\kappa)r_{A,1}}\mathbf 1_{\mathcal R_1}
\,\middle|\,
\mathcal U_A
\right]
\le
C e^{\theta h}
e^{-\lambda_0W(A)}
e^{-(-K\beta-\kappa+\theta)h} 
=C 
e^{-\lambda_0W(A)}
e^{(K\beta+\kappa)h}.
\]
Together with \eqref{eq:N0}, this proves the lemma by taking constant $2C$ and taking $c:=K$.
\end{proof}

\begin{lemma}[Weighted continuation bound]
\label{lem:one-site-continuation}
After decreasing \(\beta_0(d,q)\), there exists \(C_{\rm cont}<\infty\), depending only on \(d\) and \(q\), such that the following holds.

Fix a finite nonempty \(A\subset\Lambda\), \(h\in[0,T]\), and a reached
\((A,h)\)-admissible trial at depth \(r\).  Let \(v=v(r)\), and let
\(\mathscr G\) be the sigma-field generated by the exploration revealed before
testing the cone at \((v,T-r)\), including the data
\((r,v,\mathsf C_r(v))\), but excluding the graphical marks strictly below its tip. 
Continue the sequential exploration from this trial, and write
\[
        r=r_0<r_1<\cdots<r_J
\]
for the depths of all trials subsequently reached, including the current one.
Let
\[
        \mathcal N_r
        :=
        \{\text{the continuation reaches depth \(T\) without a successful trial}\}.
\]
Then
\begin{equation}\label{eq:one-site-continuation}
\mathbb E\left[
    \sum_{j=0}^{J}e^{(1+\kappa)(r_j-r)}
    +e^{(1+\kappa)(T-r)}\mathbf 1_{\mathcal N_r}
    \,\middle|\,
    \mathscr G
\right]
\le C_{\rm cont}.
\end{equation}
\end{lemma}

\begin{proof}
The proof has two steps: Step 1 shows that a single trial contracts the exponential weight by a
fixed factor \(\rho_{\rm cont}<1\); Step 2 sums this one-trial contraction over the sequence of
trials \(r_0<\cdots<r_J\), via a supermartingale-type recursion, to obtain the uniform bound
\(C_{\rm cont}\).

\smallskip
\noindent\textbf{Step 1: contraction for one trial.}
We first prove a uniform contraction for
one trial. Fix a reached trial at depth \(u<T\), write \(v_u:=v(u)\), and
denote its pre-test sigma-field by \(\mathscr G_u\). Conditional on
\(\mathscr G_u\), the history strictly below the cone tip is a fresh one-site
history.
Let \(\mathcal E_u^+\) be the event that the current cone exits and a later
trial is reached, and let \(u^+\) be the depth of the first such trial, with
\(u^+:=T\) off \(\mathcal E_u^+\). Let \(\mathcal E_u^0\) be the event that
the current cone exits, the history survives to depth \(T\), and no later
trial is reached. We claim that, after decreasing \(\beta_0\),
\begin{equation}\label{eq:cycle-contraction}
\mathbb E\left[
    e^{(1+\kappa)(u^+-u)}\mathbf 1_{\mathcal E_u^+}
    +e^{(1+\kappa)(T-u)}\mathbf 1_{\mathcal E_u^0}
    \,\middle|\,
    \mathscr G_u
\right]
\le \rho_{\rm cont},
\end{equation}
for a constant \(\rho_{\rm cont}<1\), uniformly in the reached pre-test state.

By Lemma~\ref{lem:kappa bound}, after decreasing \(\beta_0\), $1/2\leq 1+\kappa<1$. 
Set
\(\mathcal E:=\mathcal E_u^+\sqcup\mathcal E_u^0\). On \(\mathcal E\), let
\[
\Delta:=
\begin{cases}
u^+-u,&\mathcal E_u^+\text{ occurs},\\
T-u,&\mathcal E_u^0\text{ occurs}.
\end{cases}
\]
Let \(\delta\le\Delta\) be the elapsed depth of the first exit from the
current cone, and let \(\widetilde L_t,\widetilde\chi_t\) denote the lineage-time and spatial length accumulated between backward depths \(u\) and \(u+t\).

On \(\mathcal E\),
the history is nonempty throughout \([0,\Delta]\).

The exit occurs at elapsed depth \(\delta\), so
\[
t\le\delta
\le c_{\rm cone}^{-1}\widetilde\chi_\delta
\le c_{\rm cone}^{-1}\widetilde\chi_\Delta,
\qquad 0\le t\le\delta.
\]
Thus every singleton time before the exit lies in
\([0,c_{\rm cone}^{-1}\widetilde\chi_\Delta]\).
Now suppose that \(\delta<t<\Delta\) and that the history at depth \(u+t\)
is a singleton \(\{w\}\). Since no later trial has yet been reached, this
singleton is not admissible. As \(u+t\ge h\), this means
$c_{\rm cone}(u+t)
<\operatorname{dist}(w,A)$.
Moreover, the current trial is admissible, so
\(\operatorname{dist}(v_u,A)\le c_{\rm cone}u\), and the exposed history
contains a spatial path from \(v_u\) to \(w\). Therefore
\[
c_{\rm cone}(u+t)
<\operatorname{dist}(w,A)
\le\operatorname{dist}(v_u,A)+\widetilde\chi_t
\le c_{\rm cone}u+\widetilde\chi_t.
\]
Hence
\[
t<c_{\rm cone}^{-1}\widetilde\chi_t
\le c_{\rm cone}^{-1}\widetilde\chi_\Delta.
\]
Consequently, all singleton times in \([0,\Delta)\) have total length at
most \(c_{\rm cone}^{-1}\widetilde\chi_\Delta\), while the times with at least
two active vertices have total length at most
\(\frac12\widetilde L_\Delta\). Therefore
\begin{equation}\label{eq:cycle-delay}
        \Delta
        \le \frac12\widetilde L_\Delta
        +c_{\rm cone}^{-1}\widetilde\chi_\Delta
        \qquad\text{on }\mathcal E.
\end{equation}

Couple the fresh history below the tip to the one-particle branching
envelope, and let \(\widehat L,\widehat\chi\) be the envelope's occupation
time and spatial length up to depth \(T-u\). Then
\[
        \widetilde L_\Delta\le\widehat L,
        \qquad
        \widetilde\chi_\Delta\le\widehat\chi.
\]
Set
\[
        \eta:=\frac{1+\kappa}{2},
        \qquad
        \zeta:=\frac{1+\kappa}{c_{\rm cone}}.
\]
Since \(c_{\rm cone}=2\) and \(1+\kappa<1\), we have \(\zeta<1/2<10\). Let
\(M_*:=M(1/2,10)<\infty\) be the uniform envelope bound from
\eqref{eq:subBP-exp}.

On \(\mathcal E\), \eqref{eq:cycle-delay} and the envelope domination give
\((1+\kappa)\Delta\le\eta\widehat L+\zeta\widehat\chi\), while an exit is
possible only if the first ring of the initial envelope particle branches.
Decomposing at that ring and applying
\eqref{eq:subBP-exp} to its \(D=2d\) descendants therefore gives
\[
\begin{aligned}
\mathbb E\left[
 e^{(1+\kappa)\Delta}\mathbf 1_{\mathcal E}
 \,\middle|\,\mathscr G_u
\right]
&\le
\widehat{\mathbb E}_1\left[
 e^{\eta\widehat L+\zeta\widehat\chi};
 \text{the first ring branches before \(T-u\)}
\right]\\
&\le
\int_0^{T-u}
\alpha e^{-t}
e^{\eta t}
e^{\zeta D}
M_*^D\,dt
\le 2\alpha e^{10D}M_*^D .
\end{aligned}
\]
Indeed, \(\alpha e^{-t}\,dt\) is the joint law of a first ring at depth \(t\)
that branches; the initial segment contributes \(e^{\eta t}\), the branch
contributes \(e^{\zeta D}\), and the \(D\) descendants contribute at most
\(M_*^D\). Since \(\eta\le1/2\), \(\zeta\le10\), and
\(\alpha=O(\beta)\), decreasing \(\beta_0\) makes the last display at most
\(\rho_{\rm cont}:=1/3\), proving \eqref{eq:cycle-contraction}.
For a reached trial at depth \(T\), the same estimate is trivial because the
zero-height cone succeeds.

\smallskip
\noindent\textbf{Step 2: summing over the trial sequence.}
For \(j\ge0\), put \(\mathcal T_j:=\{J\ge j\}\), and set
\(r_j:=T\) on \(\mathcal T_j^c\). Let \(\mathscr G_j\) be the sigma-field
just before the \(j\)-th test on \(\mathcal T_j\), continued by the terminal
exploration transcript on \(\mathcal T_j^c\). Then
\[
        \mathscr G=\mathscr G_0\subseteq\mathscr G_1\subseteq\cdots,
        \qquad
        \mathcal T_j\in\mathscr G_j,
        \qquad
        r_j\text{ is }\mathscr G_j\text{-measurable}.
\]
On \(\mathcal T_j\), the two alternatives in
\eqref{eq:cycle-contraction} are respectively \(\mathcal T_{j+1}\) and
\(\mathcal N_r\cap\{J=j\}\); off \(\mathcal T_j\), both vanish.
Thus its global form is
\[
\mathbb E\left[
 e^{(1+\kappa)(r_{j+1}-r_j)}\mathbf 1_{\mathcal T_{j+1}}
 +e^{(1+\kappa)(T-r_j)}
  \mathbf 1_{\mathcal N_r\cap\{J=j\}}
 \,\middle|\,\mathscr G_j
\right]
\le \rho_{\rm cont}\mathbf 1_{\mathcal T_j}.
\]

Define
\[
a_j
:=
\mathbb E\left[
 e^{(1+\kappa)(r_j-r)}\mathbf 1_{\{J\ge j\}}
 \,\middle|\,
 \mathscr G
\right],
\qquad
f_j
:=
\mathbb E\left[
 e^{(1+\kappa)(T-r)}
 \mathbf 1_{\mathcal N_r\cap\{J=j\}}
 \,\middle|\,
 \mathscr G
\right].
\]

Applying \eqref{eq:cycle-contraction} at the \(j\)-th reached trial gives
\[
a_{j+1}+f_j\le\rho_{\rm cont}a_j,
\qquad a_0=1.
\]
Therefore \(a_j\le\rho_{\rm cont}^j\). Since $\mathcal N_r
        =\bigsqcup_{j\ge0}
          (\mathcal N_r\cap\{J=j\})$, we obtain
\[
\mathbb E\left[
    \sum_{j=0}^{J}e^{(1+\kappa)(r_j-r)}
    +e^{(1+\kappa)(T-r)}\mathbf 1_{\mathcal N_r}
    \,\middle|\,\mathscr G
\right]
=\sum_{j\ge0}(a_j+f_j)                 \le1+\rho_{\rm cont}\sum_{j\ge0}a_j
\le\frac1{1-\rho_{\rm cont}}
\le\frac32.
\]
This proves \eqref{eq:one-site-continuation} with \(C_{\rm cont}:=3/2\).
\end{proof}

\begin{proof}[Proof of Proposition~\ref{prop:one-block-SR}]

We suppress the copy index \(a\) throughout the proof.
Fix the outside transcript \(\mathcal{H}_A^-=X\), and set
\[
        h:=\ell_A(X).
\]

%%%%
For a given subset $S\subset \Lambda$, define $V_{\mathrm{Blue}_S^*}$ be the union of the top sets of the Blue clusters that arise when
exposing the joint histories of $S$. Then,
\[
\{A\subset V_{\rm{Blue}}\}=\{A\subset V_{\mathrm{Blue}_A^*}\}\cap \{\mathcal{H}_A\cap X=\eset\}.
\]

%%%%
In an exploration from \(A\times\{T\}\), define
\[
        E_A^*(h):=E_{{\rm res},A}^*(h)\cup E_{{\rm fail},A}^*(h),
\]
where \(E_{{\rm res},A}^*(h)\) is the event that, while exposing only histories
of \(A\), some reached admissible cone trial \(k\) succeeds and takes the
residual branch
\[
        J_{(v_{A,k},\tau_{A,k})}>p_{\tau_{A,k}}(v_{A,k}),
\]
and \(E_{{\rm fail},A}^*(h)\) is the event that, while exposing only histories
of \(A\), the set \(A\) is Red but no admissible cone trial succeeds.
Then,
\[
\{A\in\SR\}\subset E_A^*(h)\cap\{\mathcal H_A\cap X=\eset\}.
\]
Remark that as $X\cap \mathsf D_h(A)=\varnothing$, every cone \(\mathsf C\subset\mathsf D_h(A)\) is disjoint from the outside
transcript \(X\).

%%%
We use the top-unit event from \eqref{eq:top-unit-U} on time slab $[T-1,T]$.
On $\mathsf T_A:=\{A\in{\rm SR}\}\cup \{A\subset V_{\rm Blue}\}$, \(\mathcal U_A\) must occur.
Indeed, if some
\(u\in A\) has no update in \((s_u,T]\), then the vertical branch
from \((u,T)\) hits the outside transcript at \((u,s_u)\). Therefore,
$\mathbb{P}\big(A\in\mathrm{SR} \,\big\vert\, \mathcal{H}_A^-= X,\{A\in\mathrm{SR}\}\cup\{A\subset V_{\mathrm{Blue}}\}\big)$ is bounded above by

\begin{align*}
\frac{
\mathbb{P}\left(
E_A^*(h),\,
\mathcal{H}_{A}\cap  X=\emptyset,\,
\mathcal{U}_A
\mid \mathcal{H}_{A}^{-}= X
\right)
}{
\mathbb{P}\left(
A\subset V_{\mathrm{Blue}_{A}^{*}},\,
\mathcal{H}_{A}\cap X=\emptyset,\,
\mathcal{U}_A
\mid \mathcal{H}_{A}^{-}= X
\right)
}
=
\frac{
\mathbb{P}\left(
E_A^*(h),\,
\mathcal{H}_{A}\cap  X=\emptyset
\mid \mathcal{U}_A
\right)
}{
\mathbb{P}\left(
A\subset V_{\mathrm{Blue}_{A}^{*}},\,
\mathcal{H}_{A}\cap X=\emptyset
\mid \mathcal{U}_A
\right)
}
.
\end{align*}
The equality holds because all indicated events are now $\mathcal{H}_A$ measurable.

%%%
For the denominator, there is a constant
\(c_{\rm bl}>0\) such that
\[
\mathbb{P}\left(
A\subset V_{\mathrm{Blue}_{A}^{*}},\,
\mathcal{H}_{A}\cap X=\emptyset
\mid \mathcal{U}_A
\right)
        \ge
        c_{\rm bl}^{|A|}.
\]
Indeed, for \(u\in A_{\rm top}\), conditional on the existence of an update in
\((s_u,T]\), the latest such update is oblivious with probability
\(p_{\rm obl}\). For \(u\in A\setminus A_{\rm top}\), require that the latest
update in \((T-1,T]\) exists and is oblivious; this has probability
\((1-e^{-1})p_{\rm obl}\). If all these independent events occur, then the
histories from all sites of \(A\) die before time \(T-1\), and \(A\subset
V_{\rm Blue}\). Thus one may take $c_{\rm bl}:=p_{\rm obl}\min\{1,1-e^{-1}\}$. Therefore,
\begin{align*}
\mathbb P
\left(
    A\in{\rm SR}
    \,\middle|\,
    \mathcal H_A^-=X,
    \{A\subset V_{\rm Blue}\}\cup\{A\in{\rm SR}\}
\right)  
\leq
c_{\rm bl}^{-|A|}
        \mathbb P
        \left(
            E_A^*(h),\,
\mathcal{H}_{A}\cap  X=\emptyset
            \,\middle|\,
            \mathcal U_A
        \right)
\leq
c_{\rm bl}^{-|A|}
        \mathbb P
        \left(
            E_A^*(h)
            \,\middle|\,
            \mathcal U_A
        \right).
\end{align*}

It remains to bound $\mathbb P\left(E_A^*(h)\,\middle|\,\mathcal U_A\right)$.
Recall that \(\mathcal R_k\) is the event that trial \(k\) is reached, that
\(r_{A,k}\) is its backward depth on this event, and that
\(\tau_{A,k}=T-r_{A,k}\) is its absolute tip time, hence the remaining
lower-cone time.
Define
\[
        \mathcal M_A
        :=
        \sum_{k\ge1}
        \mathbf 1_{\mathcal R_k}
        e^{(1+\kappa)r_{A,k}}
        +
        e^{(1+\kappa) T}
        \mathbf 1_{E_{{\rm fail},A}^*(h)} .
\]
We first prove that, for every \(\lambda_0>0\), after decreasing
\(\beta_0(d,q)\) if necessary,
\begin{equation}\label{eq:MA-bound}
        \mathbb E\left[
            \mathcal M_A
            \,\middle|\,
            \mathcal U_A
        \right]
        \le
        C e^{-\lambda_0W(A)}e^{c\beta h}.
\end{equation}
On \(\mathcal R_1\), define the continuation functional from the first trial by
\[
\begin{aligned}
        \mathcal C_{A,1}
        &:=
        \sum_{j\ge0}
        \mathbf 1_{\mathcal R_{1+j}}
        e^{(1+\kappa)(r_{A,1+j}-r_{A,1})}
        +
        e^{(1+\kappa)(T-r_{A,1})}
        \mathbf 1_{E_{{\rm fail},A}^*(h)} .
\end{aligned}
\]
Then, pathwise on \(\mathcal R_1\), $\mathcal M_A
        =
        e^{(1+\kappa)r_{A,1}}\mathcal C_{A,1}$.
Let \(\mathcal G_{A,1}\) be the intrinsic pre-test sigma-field of the first
trial, enlarged by \(\sigma(\mathcal U_A)\). Since \(r_{A,1}\ge h\) on
\(\mathcal R_1\), and since
\(h=\ell_A(X)\ge1\) by the definition of \(\ell_A\), the top-unit event
\(\mathcal U_A\) is determined by the history exposed above the first trial.
Thus, on \(\mathcal R_1\), adding \(\mathcal U_A\) does not change the
pre-test state or the fresh lower-cone law.
Lemma~\ref{lem:one-site-continuation}, applied with
\(r=r_{A,1}\), gives
\[
        \mathbb E\left[
            \mathcal C_{A,1}
            \,\middle|\,
            \mathcal G_{A,1}
        \right]
        \le C_{\rm cont}.
\]
Since
\(e^{(1+\kappa)r_{A,1}}\mathbf 1_{\mathcal R_1}\) is
\(\mathcal G_{A,1}\)-measurable,
\[
\begin{aligned}
        \mathbb E\left[
            \mathcal M_A\mathbf 1_{\mathcal R_1}
            \,\middle|\,
            \mathcal U_A
        \right]
        &\le
        C_{\rm cont}
        \mathbb E\left[
            e^{(1+\kappa)r_{A,1}}
            \mathbf 1_{\mathcal R_1}
            \,\middle|\,
            \mathcal U_A
        \right].
\end{aligned}
\]
On \(\mathcal R_1^c\), no trial is reached. Hence
\[
        E_{{\rm fail},A}^*(h)\cap\mathcal R_1^c
        =
        \mathcal N_0,
\]
where \(\mathcal N_0\) is the no-admissible-trial Red event from
Lemma~\ref{lem:first-h-admissible}. Therefore
\[
\begin{aligned}
        \mathbb E\left[
            \mathcal M_A
            \,\middle|\,
            \mathcal U_A
        \right]
        &\le
        C
        \mathbb E\left[
            e^{(1+\kappa)r_{A,1}}
            \mathbf 1_{\mathcal R_1}
            +
            e^{(1+\kappa) T}\mathbf 1_{\mathcal N_0}
            \,\middle|\,
            \mathcal U_A
        \right] \le
        C e^{-\lambda_0W(A)}e^{c\beta h},
\end{aligned}
\]
where the last inequality is Lemma~\ref{lem:first-h-admissible}. This proves
\eqref{eq:MA-bound}.

We now convert the weighted estimate into a probability bound for
\(E_A^*(h)\). By definition,
\[
        E_{{\rm res},A}^*(h)
        =
        \bigcup_{k\ge1}
        \left\{
            \mathcal R_k,\,
            \mathsf{Bot}_{\tau_{A,k}}(v_{A,k}),\,
            J_{(v_{A,k},\tau_{A,k})}>p_{\tau_{A,k}}(v_{A,k})
        \right\}.
\]
Let \(\mathcal G_{A,k}\) be the \(k\)-th pre-test sigma-field enlarged
by \(\sigma(\mathcal U_A)\). On \(\mathcal R_k\), the lower cone below
\((v_{A,k},\tau_{A,k})\) is contained in \(\mathsf D_h(A)\), and
therefore is disjoint from the outside transcript \(X\). Since
\(r_{A,k}\ge h\ge1\), the event \(\mathcal U_A\) is already
determined by the pre-test history, and the lower cone is fresh conditionally
on \(\mathcal G_{A,k}\). Proposition \ref{prop:isolated-cone-residual}
applies when \(\tau_{A,k}>0\), while the case \(\tau_{A,k}=0\) is trivial after
increasing \(C\). Thus, in both cases, the following holds on \(\mathcal R_k\):
\[
\begin{aligned}
        \mathbb P
        \left(
            \mathsf{Bot}_{\tau_{A,k}}(v_{A,k}),
            J_{(v_{A,k},\tau_{A,k})}>p_{\tau_{A,k}}(v_{A,k})
            \,\middle|\,
            \mathcal G_{A,k}
        \right)
        &=
        M_{\tau_{A,k}}\bigl(1-p_{\tau_{A,k}}(v_{A,k})\bigr)
        \le
        C e^{-(1+\kappa)\tau_{A,k}}     
        =
        C e^{-(1+\kappa) T}e^{(1+\kappa)r_{A,k}} .
\end{aligned}
\]
The event \(\mathcal R_k\) is \(\mathcal G_{A,k}\)-measurable, so the
tower property and a union bound imply
\[
\begin{aligned}
        \mathbb P(E_{{\rm res},A}^*(h)\mid\mathcal U_A)
        &\le
        C e^{-(1+\kappa) T}
        \mathbb E
        \left[
            \sum_{k\ge1}
            \mathbf 1_{\mathcal R_k}
            e^{(1+\kappa)r_{A,k}}
            \,\middle|\,
            \mathcal U_A
        \right]
        \le
        C e^{-(1+\kappa) T}
        e^{-\lambda_0W(A)}e^{c\beta h},
\end{aligned}
\]
by \eqref{eq:MA-bound}.
For the terminal branch,
\[
\begin{aligned}
        \mathbb P(E_{{\rm fail},A}^*(h)\mid\mathcal U_A)
        &=
        e^{-(1+\kappa) T}
        \mathbb E
        \left[
            e^{(1+\kappa) T}
            \mathbf 1_{E_{{\rm fail},A}^*(h)}
            \,\middle|\,
            \mathcal U_A
        \right]  \le
        C e^{-(1+\kappa) T}
        e^{-\lambda_0W(A)}e^{c\beta h}.
\end{aligned}
\]
Since \(E_A^*(h)=E_{{\rm res},A}^*(h)\cup E_{{\rm fail},A}^*(h)\),
\[
        \mathbb P(E_A^*(h)\mid\mathcal U_A)
        \le
        C e^{-(1+\kappa)T}e^{-\lambda_0W(A)}e^{c\beta h}.
\]
Choose
\[
        \lambda_0:=\lambda_{\rm SR}+\log q+\log c_{\rm bl}^{-1}+1.
\]
Using \(W(A)\ge |A|\), the preceding display and the bound before this
paragraph give
\[
\begin{aligned}
    q^{|A|}
    \mathbb P
    \left(
        A\in{\rm SR}
        \,\middle|\,
        \mathcal H_A^-=X,\mathsf T_A
    \right)
    &\le
    C
    e^{-(1+\kappa)T}
    e^{-\lambda_{\rm SR} W(A)}
    e^{c\beta h}.
\end{aligned}
\]
Renaming the constants as \(C_{\rm sr}\) and \(c_{\rm sr}\) proves the claim.

\end{proof}

%==========================================

\subsection{Lower bound}
Set $t=t_\star-s$ with $s>0$. Recall the one-site magnetization
\[
    \mathfrak{m}_t^{(n)}(\sigma_0,c)
    = \mathbb P_{\sigma_0}(\sigma_t(o)=c)-\frac1q.
\]
In this section, we abbreviate it by $\mathfrak{m}_t$ and $\Lambda:=\Lambda_n$. Fix the monochromatic initial configuration $\mathfrak{c}\equiv c:=1$, and let $\sigma_t$ denote the Glauber dynamics with initial configuration $\mathfrak{c}$. Define the distinguishing statistic $Y$ and its stationary counterpart $Y'$,
\[
    f(\sigma) := \sum_{v\in\Lambda} \mathfrak{m}_t(\mathfrak{c},c)\bigl(\mathbf 1(\sigma(v)=1)-\mathbf 1(\sigma(v)=2)\bigr), \qquad Y := f(\sigma_t), \quad Y' := f(\pi),
\]
where $\pi$ is stationary distribution.

\paragraph{Mean.}
By the spatial and color symmetries among $2,\dots,q$ under the monochromatic start,
\[
    \mathbb P_{\mathfrak{c}}(\sigma_t(o)=2) = \frac{1-\mathbb P_{\mathfrak{c}}(\sigma_t(o)=1)}{q-1},
\]
and therefore
\[
    \mathbb E[Y] = |\Lambda|\,\mathfrak{m}_t(\mathfrak{c},c)\Bigl(\mathbb P_{\mathfrak{c}}(\sigma_t(o)=1)-\mathbb P_{\mathfrak{c}}(\sigma_t(o)=2)\Bigr) = \frac{q}{q-1}|\Lambda|\,\mathfrak{m}_t(\mathfrak{c},c)^2.
\]
Symmetrically, $\mathbb E[Y']=0$ under $\pi$.

\paragraph{Variance.}
For $v\in \Lambda$, denote the information percolation cluster containing $v$ by $\mathcal{C}_v$. Because it is driven by a subcritical branching process, $\mathbb{E}|\mathcal{C}_v|<K$ as established in Lemma 2.2 of \cite{Universality}. Moreover, mimicking the covariance decomposition from Claim 3.4 in \cite{Universality}, it follows that for any function $\psi:[q]\to\mathbb R$ with $\|\psi\|_\infty\le A$ and setting $\psi_t(u):=\psi(\sigma_t(u))$, we have the covariance bound
\begin{equation}\label{eq:cov-sum-bound}
    \sum_{u\in \Lambda}\mathrm{Cov}\big(\psi_t(u),\psi_t(v)\big) \le 2A^2K.
\end{equation}
Let $\phi(a)=\mathbf 1(a=1)-\mathbf 1(a=2)$, which is bounded by $|\phi|\le 1$, and $Y=\mathfrak{m}_t(\mathfrak{c},c)\sum_{v}\phi(\sigma_t(v))$. Applying the covariance sum estimate \eqref{eq:cov-sum-bound} gives
\[
    \mathrm{Var}(Y) = \mathfrak{m}_t(\mathfrak{c},c)^2\sum_{u,v\in\Lambda}\mathrm{Cov}\big(\phi(\sigma_t(u)),\phi(\sigma_t(v))\big) \le C|\Lambda|\,\mathfrak{m}_t(\mathfrak{c},c)^2
\]
for a constant $C<\infty$, and an identical bound holds for $\mathrm{Var}(Y')$.

By Corollary~\ref{cor:torus-magnetization-bounds}, for any $\sigma_0$, $\abs{\mathfrak{m}_t(\sigma_0,c)}\leq C_{\mathrm{m}}e^{-(1+\kappa)t}$, establishing $\mathfrak{m}_t(\pi,c)^2\leq C e^{-2(1+\kappa)t}$.
Evaluated at $t=t_\star-s$, we deduce $|\Lambda|\mathfrak{m}_t(\pi,c)^2\leq e^{2(1+\kappa)s}$. 
Furthermore, we have $\mathfrak{m}_t(\mathfrak{c},c) \ge C_{\mathrm{all}}\,e^{-(1+\kappa)t}$, ensuring the lower bound $\mathbb E[Y] \ge c e^{2(1+\kappa)s}$.
Therefore, Chebyshev's inequality gives
\[
    \mathbb P\Bigl(Y\le \tfrac12\mathbb E[Y]\Bigr) \le \frac{4\,\mathrm{Var}(Y)}{(\mathbb E[Y])^2} \le \frac{C'}{|\Lambda|\mathfrak{m}_t(\mathfrak{c},c)^2},
    \qquad
    \mathbb P\Bigl(Y'\ge \tfrac12\mathbb E[Y]\Bigr) \le \frac{4\mathrm{Var}(Y')}{(\mathbb E[Y])^2} \le \frac{C'\mathfrak{m}_t(\pi,c)^2}{|\Lambda|\mathfrak{m}_t(\mathfrak{c},c)^4},
\]
and choosing $s=s(\varepsilon)$ sufficiently large suppresses both tail probabilities below $\varepsilon/2$. Consequently,
\[
    \bigl\|\mathbb P_{\mathfrak{c}}(\sigma_{t_\star-s}\in\cdot)-\pi\bigr\|_{\mathrm{TV}} \ge \mathbb P\Bigl(Y\ge \tfrac12\mathbb E[Y]\Bigr) - \mathbb P\Bigl(Y'\ge \tfrac12\mathbb E[Y]\Bigr) \ge 1-\varepsilon.
\]
\qed

\subsection{Conclusion}
Combining the upper and lower bounds, we complete the proof of Theorem~\ref{thm:main}. Specifically,
for the Potts Glauber dynamics on the discrete torus \(\Lambda_n=(\mathbb Z/n\mathbb Z)^d\), there exists
\(\kappa=\kappa(d,q,\beta)\in(-1,0)\) such that, centered at
\[
t_\star=\frac{1}{2(1+\kappa)}\log|\Lambda|,
\]
the family exhibits cutoff with an \(O(1)\) window. Equivalently, for every fixed \(\varepsilon\in(0,1)\),
\[
t_{\mathrm{mix}}(\varepsilon)=t_\star+O(1).
\]

%%%%%%%%%%%%%%%%%%%%%%%%%%%%%%%%%%%%%%%%%%%%%%%%%%%%%%%%%%%%%%%%%%%%%%%%%%%%%%%%%%%%%%%%%%%%%%%%%%%%%%%%%%%%%%%%%%%%%%%%%%%%%%%%%%%%%%%%%%%%%%%%%%%%%%%%%%%%%%%%%%%%%%%%%%%%%%%%%%%%%%%%%%%%%%%%%%%%%%%%%%%%%%%%%%%%%%%%%%%%%%%%%%%%%%%%%%%%%%%%
\medskip

\section{Extremality of monochromatic initial states}\label{sec:extremality}

Throughout this section we use the notation and constants from Section~\ref{sec:sausage}.
In particular, \(\eta>1+\kappa\), and after decreasing \(\beta_0(d,q)\) if necessary we assume
\begin{equation}\label{eq:mono-smallness-final}
36\,c_{\rm br}\,C_{\rm ball}(d,\vartheta)\le c_{\rm str},
\end{equation}
where
\[
C_{\rm ball}(d,\zeta):=\sum_{j\ge1}\zeta^{j/2}(2j+1)^d,
\]
and
\[
c_{\rm str}:=\frac{q}{q-1}\frac{c_{\rm len}e^{-\eta}}{1-e^{-\eta}},
\qquad
c_{\rm br}:=\frac{2C_{\rm jnt}e^{-\eta}}{1-e^{-\eta}}.
\]

Here \(1+\kappa\) is the decay rate of the signed influence kernel, whereas \(\eta\) is the
terminal-sausage tail rate. The strict inequality between them allows the cutoff below to make
deep terminal histories negligible relative to the leading kernel contribution. The constant
\(c_{\rm str}\) measures the positive contribution of non-branching terminal sausages, while
\(c_{\rm br}\) controls the possible negative contribution of branching terminal sausages.
Since \(0<\vartheta<1\), this also implies
\[
c_{\rm br}C_{\rm ball}(d,\vartheta^2)\le \frac{c_{\rm str}}{2}.
\]

The two extremality statements use this comparison in different ways. For total occupation,
summing over observation sites reduces the signed kernel to its positive total mass and gives an
exact extremality statement. For a one-site observable, that spatial summation is unavailable, so
we first establish local positivity and domination estimates for the kernel.

\subsection{Total-occupation extremality}
We first prove Theorem~\ref{thm:mono-extremal}.
\begin{proof}[Proof of Theorem \ref{thm:mono-extremal}]
By color symmetry, it suffices to consider the color \(1\). We use bars for sites of the torus
\(\Lambda_n=(\mathbb Z/n\mathbb Z)^d\). For
\(\sigma_0\in[q]^{\Lambda_n}\), set

\[
M_t(\sigma_0):=
\mathbb E_{\sigma_0}\left[\sum_{\bar w\in\Lambda_n}
\mathbf 1\{\sigma_t(\bar w)=1\}\right].
\]
By translation invariance, we may take the changed site to be \(\bar o\). It is enough to prove the
following single-site replacement inequality. Suppose that
\(\sigma_0,\sigma_0'\in[q]^{\Lambda_n}\) differ only at \(\bar o\), with
\[
\sigma_0(\bar o)=1,\qquad \sigma_0'(\bar o)\neq 1.
\]
Then
\begin{equation}\label{eq:single-site-total-occ}
M_t(\sigma_0)\ge M_t(\sigma_0')
\qquad\text{for every }t\ge0.
\end{equation}
Indeed, starting from any initial configuration, one can change all non-\(1\) spins into \(1\), one
site at a time. Repeated application of \eqref{eq:single-site-total-occ} then gives
\[
M_t(\mathfrak c)\ge M_t(\sigma_0),
\]
where \(\mathfrak c\equiv 1\).

To apply the infinite-volume sausage expansion to a torus initial condition, periodize the
influence kernel by setting
\[
\overline{\IK}(\bar z,s):=\sum_{m\in\mathbb Z^d}\IK(z+nm,s),
\qquad \bar z\in\Lambda_n,
\]
where \(z\) is any lift of \(\bar z\). This is well-defined by Lemma~\ref{lem:IK-l1}.
Projecting the magnetization expansion \eqref{eq:mag-sausage} to the torus gives, for every
\(\bar w\in\Lambda_n\),
\[
\mathbb P_{\sigma_0}(\sigma_t(\bar w)=1)-\frac1q
=
\frac{q-1}{q}
\sum_{\tau=0}^{n_t}\sum_{x_0\in\Xi}
\mu(x_0)\mathbf 1\{\ell(x_0)>\theta_t+\tau\}
\sum_{\bar z\in\Lambda_n}
\overline{\IK}(\bar z-\bar w,n_t-\tau)\,
g_{x_0,\sigma_0,1}(\bar z,\theta_t+\tau).
\]
Here the function \(g\) is evaluated using the periodic lift of the torus initial condition.

Subtracting the analogous identity for \(\sigma_0'\) and summing over \(\bar w\in\Lambda_n\), we get
\begin{equation}\label{eq:M-diff-gamma}
M_t(\sigma_0)-M_t(\sigma_0')
=
\frac{q-1}{q}
\sum_{\tau=0}^{n_t}\mathcal I_{n_t-\tau}\,\Gamma_t(\tau),
\end{equation}
where
\[
\mathcal I_s
:=
\sum_{\bar u\in\Lambda_n}\overline{\IK}(\bar u,s)
=
\sum_{u\in\mathbb Z^d}\IK(u,s),
\]
and
\[
\Gamma_t(\tau)
:=
\sum_{x_0\in\Xi}\mu(x_0)\mathbf 1\{\ell(x_0)>\tau+\theta_t\}
\sum_{\bar z\in\Lambda_n}
\Delta g_{x_0}(\bar z,\tau+\theta_t),
\quad
\Delta g_{x_0}(\bar z,r)
:=
g_{x_0,\sigma_0,1}(\bar z,r)-g_{x_0,\sigma_0',1}(\bar z,r).
\]

We first note that
\[
\mathcal I_s>0
\qquad\text{for every }s\ge0.
\]
Indeed, \(\mathcal I_0=1\), while for \(s\ge1\), summing the tilted representation
\eqref{eq:IK-tilt} over \(u\in\mathbb Z^d\) and using
\(\sum_u K(u;r,s)=1\) gives
\[
\mathcal I_s
=
e^{-(1+\kappa)s}
\sum_{r\ge1}\widetilde{\mathbb P}_\kappa(S_r=s)>0.
\]
Thus, by \eqref{eq:M-diff-gamma}, it remains to show that
\begin{equation}\label{eq:Gamma-nonnegative}
\Gamma_t(\tau)\ge0
\qquad\text{for every }0\le \tau\le n_t.
\end{equation}

Fix \(0\le \tau\le n_t\) and write \(s_\tau:=\theta_t+\tau\). Split
\[
\Gamma_t(\tau)=\Gamma_t^{(0)}(\tau)+\Gamma_t^{(\ge1)}(\tau),
\]
according to whether the truncated terminal sausage has no branch-out,
\(\tilde b_{s_\tau}(x_0)=0\), or at least one branch-out.

If \(\tilde b_{s_\tau}(x_0)=0\), then the truncated terminal sausage is vertical. Hence
\(\Delta g_{x_0}(\bar z,s_\tau)\) is supported only at the site \(\bar o\), and at that site
\[
\Delta g_{x_0}(\bar o,s_\tau)
=
1-\left(-\frac1{q-1}\right)
=
\frac q{q-1}.
\]
Therefore, by Lemma~\ref{lem:terminal-sausage-tails},
\[
\Gamma_t^{(0)}(\tau)
=
\frac q{q-1}
\mu\bigl(\ell(x_0)>s_\tau,\ \tilde b_{s_\tau}(x_0)=0\bigr)
\ge
c_{\rm str}e^{-\eta\tau}.
\]

We now control the branching contribution. If
\(\Delta g_{x_0}(\bar z,s_\tau)\neq0\), then the truncated terminal sausage must reach the unique
site \(\bar o\) where the two initial configurations differ. Since each branch-out changes the
spatial position of a lineage by at most one in \(\ell^\infty\)-distance, where $\rho_\infty(\cdot,\cdot)$ denotes the torus distance,
\[
\rho_\infty(\bar z,\bar o)\le \tilde b_{s_\tau}(x_0).
\]
Hence the number of possible \(\bar z\)'s is at most
\((2\tilde b_{s_\tau}(x_0)+1)^d\). Using \(|\Delta g_{x_0}|\le2\), we obtain
\[
\sum_{\bar z\in\Lambda_n}\Delta g_{x_0}(\bar z,s_\tau)
\ge
-2(2\tilde b_{s_\tau}(x_0)+1)^d.
\]
Thus
\[
\Gamma_t^{(\ge1)}(\tau)
\ge
-2\sum_{j\ge1}(2j+1)^d
\mu\bigl(\ell(x_0)>s_\tau,\ \tilde b_{s_\tau}(x_0)\ge j\bigr).
\]
By Lemma~\ref{lem:terminal-sausage-tails},
\[
\mu\bigl(\ell(x_0)>\tau+\theta_t,\ \tilde b_{s_\tau}(x_0)\ge j\bigr)
\le
C_{\mathrm{jnt}}\vartheta^j\sum_{\ell\ge \tau+1}e^{-\eta\ell}
=
\frac{C_{\mathrm{jnt}}e^{-\eta}}{1-e^{-\eta}}\,\vartheta^j e^{-\eta\tau}.
\]
Consequently,
\[
\Gamma_t^{(\ge1)}(\tau)
\ge
-c_{\rm br} C_{\rm ball}(d,\vartheta^2)e^{-\eta\tau}
\ge
-\frac{c_{\rm str}}2 e^{-\eta\tau}.
\]
Combining the last two estimates gives
\[
\Gamma_t(\tau)
\ge
\frac{c_{\rm str}}2e^{-\eta\tau}\ge0.
\]
Thus the non-branching contribution dominates the entire branching error at every depth \(\tau\).
This proves \eqref{eq:Gamma-nonnegative}, and therefore \eqref{eq:single-site-total-occ}.
The theorem follows by iterating over the sites at which the initial configuration differs from
\(\mathfrak c\equiv1\).

\end{proof}

\subsection{A local kernel input}

The total-occupation argument above worked by summing \(\IK(\cdot,\tau)\) over all sites, which
collapses the signed kernel to its total mass and reduces positivity to the scalar comparison
\eqref{eq:mono-smallness-final}. For the one-site statement of Theorem~\ref{thm:worstini}, that
summation is unavailable: we instead need \(\IK(u,s)\) itself to be positive, and in fact to
dominate its own spatial tail, at the single site \(u\) of interest. We obtain this by replacing
total-mass positivity with a \emph{local}, pointwise estimate coming from the local-limit
expansion for the signed convolution powers underlying \(\IK\): the following three consequences
show that on the diffusive (moderate-deviation) scale, \(\IK(u,s)\) is comparable to a Gaussian
density and dominates the mass in a geometric neighborhood of \(u\).

\begin{prop}[Local estimates for the signed influence kernel]\label{prop:local-kernel-input}
Fix \(A>0\) and \(\zeta\in(0,1)\). The following estimates hold.

\begin{enumerate}
\item[\rm(a)] There exist \(c_{\rm md}(A)>0\), \(B_{\rm md}(A)<\infty\), and
\(t_{\rm md}(A)<\infty\) such that, for all \(s\ge t_{\rm md}(A)\) and all
\(u\in\mathbb Z^d\) satisfying
\[
|u|\le A\sqrt{s\log\log s},
\]
one has
\[
\IK(u,s)
\ge
c_{\rm md}(A)s^{-d/2}(\log s)^{-B_{\rm md}(A)}
e^{-(1+\kappa)s}.
\]

\item[\rm(b)] There exists \(t_{\rm dom}(A,\zeta)<\infty\) such that, for all
\(s\ge t_{\rm dom}(A,\zeta)\) and all
\[
|u|\le A\sqrt{s\log\log s},
\]
one has
\[
\IK(u,s)
\ge
\frac{1}{9C_{\rm ball}(d,\zeta)}
\sum_{b\ge1}\zeta^b
\sum_{z\in B_\infty(u,b)}|\IK(z,s)|.
\]
In particular, \(\IK(u,s)>0\) on this scale.

\item[\rm(c)] There exist \(C_{\rm out}(A,\zeta)<\infty\), \(\Delta(A)>0\), and
\(t_{\rm out}(A,\zeta)<\infty\) such that, for all \(s\ge t_{\rm out}(A,\zeta)\),
\[
\sum_{|u|>A\sqrt{s\log\log s}}
\sum_{b\ge0}\zeta^b
\sum_{z\in B_\infty(u,b)}|\IK(z,s)|
\le
C_{\rm out}(A,\zeta)e^{-(1+\kappa)s}(\log s)^{-\Delta(A)}.
\]
\end{enumerate}
\end{prop}
Part~(b) bounds the weighted branching error by the kernel value at \(u\), which is the comparison
needed at each shallow depth. Part~(a) gives the quantitative lower bound needed to dominate the
deep tail, and part~(c) controls the influence of spins outside the improvement ball in the final
proof. The proof is deferred to Appendix~\ref{sec:appendix}.

We choose the cutoff depth so that the deep terminal-sausage tail decays strictly faster than the
kernel signal. Set
\begin{equation}\label{eq:ccut-def}
c_{\rm cut}
:=
\frac12\left(1+\frac{1+\kappa}{\eta}\right).
\end{equation}
Since \(\eta>1+\kappa\), we have \(c_{\rm cut}\in(0,1)\) and
\[
\eta c_{\rm cut}>1+\kappa.
\]

\begin{lemma}[Local improvement]\label{lem:local-improvement}
There exists \(t_{\rm loc}<\infty\) such that the following holds for every \(t\ge t_{\rm loc}\).
Let
\[
n_c:=\lfloor c_{\rm cut}n_t\rfloor,
\qquad
R_t:=\sqrt{(n_t-n_c)\log\log(n_t-n_c)},
\qquad
\mathfrak B_t:=\{u\in\mathbb Z^d:\|u\|_\infty\le R_t\}.
\]
If two initial configurations \(\sigma_0,\sigma_0'\in[q]^{\mathbb Z^d}\) differ only at a site
\(u\in\mathfrak B_t\), with
\[
\sigma_0(u)=c,
\qquad
\sigma_0'(u)\neq c,
\]
then
\[
\mathfrak m_t(\sigma_0,c)\ge \mathfrak m_t(\sigma_0',c).
\]
\end{lemma}

\begin{proof}
Write
\[
D_t(u):=\mathfrak m_t(\sigma_0,c)-\mathfrak m_t(\sigma_0',c).
\]
Using the magnetization expansion \eqref{eq:mag-sausage}, split
\[
D_t(u)=D_t^{\le n_c}(u)+D_t^{>n_c}(u)
\]
according to whether \(\tau\le n_c\) or \(\tau>n_c\).
We call the latter the deep part. In the shallow part, the remaining convolution time is still of order \(n_t\), so the local
kernel estimates apply.

First consider the deep part. Since \(|g-g'|\le2\), the terminal-sausage tail estimate and the
\(\ell^1\)-bound on \(\IK\) give
\begin{equation}\label{eq:deep-bound}
|D_t^{>n_c}(u)|
\le
C\sum_{\tau=n_c+1}^{n_t}e^{-\eta\tau}
\le
C e^{-\eta n_c}.
\end{equation}

Now fix \(0\le\tau\le n_c\) and set
\[
T:=n_t-\tau.
\]
Since \(T\ge n_t-n_c\) and \(u\in\mathfrak B_t\),
\[
|u|\le R_t\le \sqrt{T\log\log T}
\]
for all sufficiently large \(t\). Applying Proposition~\ref{prop:local-kernel-input}(b) with
\(A=1\) and \(\zeta=\vartheta\), we get
\begin{equation}\label{eq:local-domination}
\IK(u,T)
\ge
\frac{1}{9C_{\rm ball}(d,\vartheta)}
\sum_{b\ge1}\vartheta^b
\sum_{z\in B_\infty(u,b)}|\IK(z,T)|.
\end{equation}

We estimate the contribution at the fixed depth \(\tau\). If
\(\tilde b_{\theta_t+\tau}(x_0)=0\), then the truncated terminal sausage is vertical. Hence the
difference of terminal biases is supported at \(u\), and there it equals \(q/(q-1)\). Therefore the
non-branching contribution at depth \(\tau\) is bounded below by
\[
\frac{q-1}{q}\,
c_{\rm str}e^{-\eta\tau}\IK(u,T).
\]

If \(\tilde b_{\theta_t+\tau}(x_0)\ge1\), then a nonzero contribution at a site \(z\) requires the
truncated terminal sausage to reach \(u\). Hence
\[
z\in B_\infty\bigl(u,\tilde b_{\theta_t+\tau}(x_0)\bigr).
\]
Using \(|g-g'|\le2\) and the terminal-sausage tail estimate, the absolute value of the branching
contribution at depth \(\tau\) is at most
\[
\frac{q-1}{q}\,
c_{\rm br}e^{-\eta\tau}
\sum_{b\ge1}\vartheta^b
\sum_{z\in B_\infty(u,b)}|\IK(z,T)|.
\]
By \eqref{eq:local-domination}, this is at most
\[
\frac{q-1}{q}\,
9c_{\rm br}C_{\rm ball}(d,\vartheta)e^{-\eta\tau}\IK(u,T).
\]
Using the smallness assumption \eqref{eq:mono-smallness-final}, we conclude that the full
contribution at depth \(\tau\) is bounded below by
\[
\frac{q-1}{q}\,
\frac{c_{\rm str}}2 e^{-\eta\tau}\IK(u,T).
\]
Proposition~\ref{prop:local-kernel-input}(b) also gives \(\IK(u,T)>0\), so each of these lower
bounds is nonnegative. Summing over \(0\le\tau\le n_c\) and keeping only the term \(\tau=0\), we
obtain
\begin{equation}\label{eq:shallow-lower-clean}
D_t^{\le n_c}(u)
\ge
\frac{q-1}{q}\,
\frac{c_{\rm str}}2\IK(u,n_t).
\end{equation}

Pointwise positivity alone does not yet dominate the potentially negative deep part. The
quantitative estimate in Proposition~\ref{prop:local-kernel-input}(a), again with \(A=1\), gives
\[
\IK(u,n_t)
\ge
c\,n_t^{-d/2}(\log n_t)^{-B}e^{-(1+\kappa)n_t}
\]
uniformly over \(u\in\mathfrak B_t\). Since \(\eta c_{\rm cut}>1+\kappa\), the deep bound
\eqref{eq:deep-bound} is
\[
o\!\left(
n_t^{-d/2}(\log n_t)^{-B}e^{-(1+\kappa)n_t}
\right).
\]
Combining this with \eqref{eq:shallow-lower-clean} gives \(D_t(u)>0\) for all sufficiently large
\(t\). This proves the lemma.
\end{proof}

The local-improvement lemma allows us to force a near-maximizing configuration to agree with the
monochromatic configuration throughout a growing ball. The exterior estimate in
Proposition~\ref{prop:local-kernel-input}(c) then shows that the remaining disagreement has
negligible influence at the origin.

\subsection{Asymptotic one-site extremality}
We now prove Theorem~\ref{thm:worstini}.

\begin{proof}[Proof of Theorem~\ref{thm:worstini}]
Fix \(c\in[q]\), and let \(\mathfrak c\equiv c\) be the monochromatic configuration on
\(\mathbb Z^d\). Put
\[
S_t:=\sup_{\sigma_0\in[q]^{\mathbb Z^d}}\mathfrak m_t(\sigma_0,c).
\]

To avoid needing an attainment argument, fix \(\varepsilon>0\) and choose an initial condition
\(\sigma^{(\varepsilon)}\) such that
\[
\mathfrak m_t(\sigma^{(\varepsilon)},c)\ge S_t-\varepsilon.
\]
Define \(n_c,R_t,\mathfrak B_t\) as in Lemma~\ref{lem:local-improvement}. By applying
Lemma~\ref{lem:local-improvement} one site at a time inside the finite set \(\mathfrak B_t\), we may
modify \(\sigma^{(\varepsilon)}\) into a configuration \(\widetilde\sigma^{(\varepsilon)}\) satisfying
\[
\widetilde\sigma^{(\varepsilon)}(u)=c
\qquad\text{for every }u\in\mathfrak B_t,
\]
while not decreasing the one-site bias:
\[
\mathfrak m_t(\widetilde\sigma^{(\varepsilon)},c)
\ge
S_t-\varepsilon.
\]

We now compare \(\widetilde\sigma^{(\varepsilon)}\) with the monochromatic configuration
\(\mathfrak c\). Set
\[
A_0:=\frac12\sqrt{1-c_{\rm cut}}\in(0,1).
\]
For all sufficiently large \(t\), uniformly in \(0\le\tau\le n_c\), with \(T:=n_t-\tau\), one has
\begin{equation}\label{eq:Rt-vs-T-clean}
R_t\ge A_0\sqrt{T\log\log T}.
\end{equation}
Indeed,
\[
\frac{R_t^2}{T\log\log T}
=
\frac{n_t-n_c}{T}\cdot
\frac{\log\log(n_t-n_c)}{\log\log T}
\ge
1-c_{\rm cut}+o(1).
\]

Since \(\widetilde\sigma^{(\varepsilon)}\) and \(\mathfrak c\) agree on \(\mathfrak B_t\), the only
possible discrepancy in the magnetization expansion comes either from depths \(\tau>n_c\), or from
terminal boundary sites outside \(\mathfrak B_t\).

By the same argument as in \eqref{eq:deep-bound}, the deep part satisfies
\begin{equation}\label{eq:deep-bound-worstini}
|\mathrm{Deep}|
\le
C e^{-\eta n_c}.
\end{equation}

Consider the shallow part \(0\le\tau\le n_c\), and set \(T=n_t-\tau\). If the truncated terminal
sausage has no branch-out, then the support of
\[
g_{x_0,\widetilde\sigma^{(\varepsilon)},c}(\cdot,\theta_t+\tau)
-
g_{x_0,\mathfrak c,c}(\cdot,\theta_t+\tau)
\]
is contained in \(\mathbb Z^d\setminus\mathfrak B_t\), and the absolute value of the difference is
at most \(2\). This contribution is therefore bounded by
\[
C e^{-\eta\tau}
\sum_{u\notin\mathfrak B_t}|\IK(u,T)|.
\]

If the truncated terminal sausage has at least one branch-out, then a nonzero contribution at a
site \(z\) requires that the truncated terminal sausage reaches some
\(u\notin\mathfrak B_t\). If the number of branch-outs is \(b\), then
\(z\in B_\infty(u,b)\). Hence the branching contribution is bounded by
\[
C e^{-\eta\tau}
\sum_{u\notin\mathfrak B_t}
\sum_{b\ge1}\vartheta^b
\sum_{z\in B_\infty(u,b)}|\IK(z,T)|.
\]
Combining the branching and non-branching parts, and allowing the \(b=0\) term, gives
\[
\mathfrak m_t(\widetilde\sigma^{(\varepsilon)},c)
-
\mathfrak m_t(\mathfrak c,c)
\le
C\sum_{\tau=0}^{n_c}e^{-\eta\tau}
\sum_{u\notin\mathfrak B_t}
\sum_{b\ge0}\vartheta^b
\sum_{z\in B_\infty(u,b)}|\IK(z,n_t-\tau)|
+
C e^{-\eta n_c}.
\]
By \eqref{eq:Rt-vs-T-clean}, Proposition~\ref{prop:local-kernel-input}(c) applies with
\(A=A_0\) and \(\zeta=\vartheta\). Therefore, for some \(\Delta>0\),
\[
\sum_{u\notin\mathfrak B_t}
\sum_{b\ge0}\vartheta^b
\sum_{z\in B_\infty(u,b)}|\IK(z,n_t-\tau)|
\le
C e^{-(1+\kappa)(n_t-\tau)}(\log n_t)^{-\Delta},
\]
uniformly in \(0\le\tau\le n_c\), where we used
\(T=n_t-\tau\asymp n_t\), and hence \((\log T)^{-\Delta}=O((\log n_t)^{-\Delta})\).
Since \(\eta>1+\kappa\), this yields
\[
\sum_{\tau=0}^{n_c}e^{-\eta\tau}
e^{-(1+\kappa)(n_t-\tau)}
\le
C e^{-(1+\kappa)n_t}.
\]
Moreover, because \(\eta c_{\rm cut}>1+\kappa\), the deep term
\(e^{-\eta n_c}\) is exponentially smaller than \(e^{-(1+\kappa)n_t}\). Hence
\begin{equation}\label{eq:almost-mono-bound}
\mathfrak m_t(\widetilde\sigma^{(\varepsilon)},c)
-
\mathfrak m_t(\mathfrak c,c)
\le
C e^{-(1+\kappa)n_t}(\log n_t)^{-\Delta}.
\end{equation}

Since
\[
\mathfrak m_t(\widetilde\sigma^{(\varepsilon)},c)\ge S_t-\varepsilon,
\]
we get from \eqref{eq:almost-mono-bound}
\[
S_t-\varepsilon
\le
\mathfrak m_t(\mathfrak c,c)
+
C e^{-(1+\kappa)n_t}(\log n_t)^{-\Delta}.
\]
Letting \(\varepsilon\downarrow0\),
\begin{equation}\label{eq:St-upper-by-mono}
S_t
\le
\mathfrak m_t(\mathfrak c,c)
+
C e^{-(1+\kappa)n_t}(\log n_t)^{-\Delta}.
\end{equation}

By Proposition~\ref{prop:all-one-magnetization-bound},
\[
\mathfrak m_t(\mathfrak c,c)
\ge
c_{\rm all}e^{-(1+\kappa)t}
\ge
c_{\rm all}e^{-(1+\kappa)}e^{-(1+\kappa)n_t}.
\]
Thus the error term in \eqref{eq:St-upper-by-mono} is at most
\[
C(\log n_t)^{-\Delta}\mathfrak m_t(\mathfrak c,c).
\]
Consequently,
\[
S_t
\le
\left(1+C(\log n_t)^{-\Delta}\right)\mathfrak m_t(\mathfrak c,c).
\]
Equivalently, after adjusting the constant,
\[
\mathfrak m_t(\mathfrak c,c)
\ge
\left(1-C(\log n_t)^{-\Delta}\right)S_t.
\]
Since \(n_t=\lfloor t\rfloor\), this is
\[
\mathfrak m_t(\mathfrak c,c)
\ge
\left(1-C(\log t)^{-\Delta}\right)
\sup_{\sigma_0\in[q]^{\mathbb Z^d}}\mathfrak m_t(\sigma_0,c),
\]
as claimed.
\end{proof}

%%%%%%%%%%%%%%%%%%%%%%%%%%%%%%%%%%%%%%%%%%%%%%%%%%%%%%%%%%%%%%%%%%%%%%%%%%%%%%%%%%%%%%%%%%%%%%%%%%%%%%%%%%%%%%%%%%%%%%%%%%%%%%%%%%%%%%%%%%%%%%%%%%%%%%%%%%%%%%%%%%%%%%%%%%%%%%%%%%%%%%%%%%%%%%%%%%%%%%%%%%%%%%%%%%%%%%%%%%%%%%%%%%%%%%%%%%%%%%%%
\medskip

\appendix
\section{Fourier bounds and local kernel estimates}\label{sec:appendix}
This appendix establishes the Fourier bounds and local kernel estimates used in the proof of one-site extremality.
The argument has three stages. In Lemma~\ref{lem:varphi-input}, we first verify the uniform Fourier properties of the signed displacement kernels associated with a single sausage. We then use them to prove the local-limit expansion in Proposition~\ref{prop:Fourier-six-exp}. Finally, we combine that expansion with the tilted renewal representation to derive
Proposition~\ref{prop:local-kernel-input}.

Recall \(\varepsilon=1/8\), \(p_{\mathrm{stop}}=1/2\), and
\(\eta=1+\log 2\) from \eqref{eq:pstop-choice}.
We begin with the displacement moment bound needed to control the contribution of histories with
at least three branch-outs.
For every integer \(t\ge1\), define
\[
A_t
:=
\{N_1>1,\ldots,N_{t-1}>1,\ N_t\geq1\},
\]
where \((N_s)_{s\ge0}\) and \(J_s\) are the branching process and its branch-out count from
Subsection~\ref{subsec:branch-length}. The intermediate conditions are vacuous when \(t=1\).

\begin{lemma}\label{lem:zt_bound}
On \(A_t\), let
\[
V_t
:=
\max\bigl\{|v|:\text{a particle of the branching process is at }v
\text{ at time }t\bigr\},
\]
and set \(V_t=0\) on \(A_t^c\).
For every integer \(k\ge1\), there exists \(C_{k,d}<\infty\) such that for every integer \(t\ge1\),
\[
\mathbb E\bigl[(V_t)^k\mathbf 1_{A_t}\mathbf 1_{\{J_t\ge3\}}\bigr]
\le
C_{k,d}e^{-2(1-\varepsilon)t}\vartheta^3(1-\vartheta)^{-(k/2+1)}.
\]
\end{lemma}

\begin{proof}
Choose a particle attaining the maximum in the definition of \(V_t\). Along its ancestral line,
each branch-out changes the position by one nearest-neighbor step. Hence, on \(A_t\),
\(V_t\le J_t\).
Thus

\begin{align*}
\mathbb E\bigl[(V_t)^k\mathbf 1_{A_t}\mathbf 1_{\{J_t\ge3\}}\bigr]
\le
\mathbb E\bigl[J_t^k\mathbf 1_{A_t}\mathbf 1_{\{J_t\ge3\}}\bigr]
=
\sum_{j\ge3} j^k\,\mathbb P(A_t,\ J_t=j)
\le
\sum_{j\ge3} j^k\,\mathbb P(A_t,\ J_t\ge j).
\end{align*}
By Lemma~\ref{lem:branch-penalty},
\[
\mathbb P(A_t,\ J_t\ge j)\le e^{1-\varepsilon}\,\vartheta^j\,e^{-2(1-\varepsilon)t},
\qquad j\ge1.
\]
Hence
\begin{equation}\label{eq:zt-sum}
\mathbb E\bigl[(V_t)^k\mathbf 1_{A_t}\mathbf 1_{\{J_t\ge3\}}\bigr]
\le
e^{1-\varepsilon}e^{-2(1-\varepsilon)t}\sum_{j\ge3} j^k\vartheta^j.
\end{equation}

It remains to bound the series. For every \(a\ge0\) and \(x\in(0,1)\),
\begin{equation}\label{eq:poly-geometric}
\sum_{n\ge0}(n+1)^a x^n \le C_a(1-x)^{-a-1},
\end{equation}
for some constant \(C_a<\infty\) depending only on \(a\). Applying \eqref{eq:poly-geometric} with
\(a=k\) and \(x=\vartheta\), we get
\[
\sum_{j\ge3} j^k\vartheta^j
=
\vartheta^3\sum_{n\ge0}(n+3)^k\vartheta^n
\le
3^k C_k\,\vartheta^3(1-\vartheta)^{-k-1}.
\]
Since \(\vartheta\le1/2\), one has
\[
(1-\vartheta)^{-k-1}
=
(1-\vartheta)^{-k/2}(1-\vartheta)^{-(k/2+1)}
\le
2^{k/2}(1-\vartheta)^{-(k/2+1)}.
\]
Substituting this into \eqref{eq:zt-sum} proves the claim.
\end{proof}

\smallskip

\begin{lemma}\label{lem:varphi-input}
There exist constants \(\rho\in(0,\pi]\), \(\gamma\in(0,1)\), \(\sigma_\pm>0\), \(A_4>0\), \(C_6>0\), and
\[
M_k:=\sup_{\ell\ge1}\sum_{z\in\mathbb Z^d}(1+|z|)^{k}|\varphi_\ell(z)|,
\qquad
M_{d+7}<\infty,
\]
such that for every \(\ell\ge1\), the Fourier transform \(\widehat\varphi_\ell\) satisfies:

\begin{enumerate}
\item \(|\widehat\varphi_\ell(\xi)|\le1\) for all \(\xi\in\mathbb T^d\), and \(|\widehat\varphi_\ell(\xi)|=1\) if and only if \(\xi=0\);

\item on \(B_{\mathbb T}(0,\rho)\), the logarithm \(\Gamma_\ell(\xi):=\log\widehat\varphi_\ell(\xi)\) is well-defined and
\[
\Gamma_\ell(\xi)
=
-\frac{\sigma_\ell}{2}|\xi|^2
+
\mathfrak a_\ell|\xi|^4
+
\mathfrak b_\ell\sum_{j=1}^d \xi_j^4
+
R_{\ell,6}(\xi),
\qquad
|R_{\ell,6}(\xi)|\le C_6|\xi|^6,
\]
with \(\sigma_-\le\sigma_\ell\le\sigma_+\) and \(|\mathfrak a_\ell|+|\mathfrak b_\ell|\le A_4\);

\item
\[
\sup_{\xi\in\mathbb T^d\setminus B_{\mathbb T}(0,\rho)}|\widehat\varphi_\ell(\xi)|\le 1-\gamma.
\]
\end{enumerate}
\end{lemma}

\noindent
\subsubsection*{Proof}
For \(\ell\ge1\), set
\[
D_\ell:=\sum_{x:\ \ell(x)=\ell}\mu(x)p(x)=e^{-\ell}w_\ell.
\]
By Lemma~\ref{lem:Zp-bound},
\begin{equation}\label{eq:Dl-bounds-proof}
\frac{p_{\mathrm{stop}}}{2(1-p_{\mathrm{stop}})}e^{-\eta\ell}
\le
D_\ell
\le
\frac{3p_{\mathrm{stop}}}{2(1-p_{\mathrm{stop}})}e^{-\eta\ell}.
\end{equation}

\smallskip
\noindent\textbf{Step 1: moment bounds.}
Fix an integer \(k\in\{1,\dots,d+7\}\). Remark
\[
\sum_{z\in\mathbb Z^d}|z|^k|\varphi_\ell(z)|
=
D_\ell^{-1}\sum_{x:\ \ell(x)=\ell}|y(x)|^k|p(x)|\mu(x).
\]
We bound the numerator by splitting according to \(b(x)\).

If \(b(x)=0\), then \(y(x)=0\), so this contribution vanishes.

If \(b(x)=1\), then \(|y(x)|=1\) and \(|p(x)|=p_1\leq 1\). Hence, using \eqref{eq:mu-joint},
\[
\sum_{x:\ \ell(x)=\ell,\ b(x)=1}|y(x)|^k|p(x)|\mu(x)
\le
p_1\,\mu(\ell(x_1)=\ell,\ b(x_1)\ge1)
\le
C\alpha e^{-\eta\ell}.
\]

If \(b(x)=2\), then \(|y(x)|\le2\) and \(0\le p(x)\le p_1\), so
\[
\sum_{x:\ \ell(x)=\ell,\ b(x)=2}|y(x)|^k|p(x)|\mu(x)
\le
2^k p_1\,\mu(\ell(x_1)=\ell,\ b(x_1)\ge2)
\le
C\alpha^2e^{-\eta\ell}.
\]

If \(b(x)\ge3\), then \(|p(x)|\le1\). We decompose according to the last integer time
\(a\in\{0,\dots,\ell-1\}\) before the first clock ring. The event ``no ring and no stop up to \(a\)''
has probability \(e^{-\eta a}\).

Conditioned on this event, the remainder of the sausage is dominated by
a branching excursion of duration \(\ell-a\). On
\(\{\ell(x)=\ell,\ b(x)\ge3\}\), the dominating process lies in \(A_{\ell-a}\),
has \(J_{\ell-a}\ge3\), and contains a particle whose distance from the origin is at least
\(|y(x)|\). Therefore
\begin{align*}
\sum_{x:\ \ell(x)=\ell,\ b(x)\ge3}|y(x)|^k\mu(x)
&\le
\sum_{a=0}^{\ell-1}e^{-\eta a}
\mathbb E\bigl[(V_{\ell-a})^k\mathbf 1_{A_{\ell-a}}\mathbf 1_{\{J_{\ell-a}\ge3\}}\bigr]\\
&\le
C\vartheta^3\sum_{a=0}^{\ell-1}e^{-\eta a}e^{-2(1-\varepsilon)(\ell-a)}(1-\vartheta)^{-(k/2+1)}
\leq
C\alpha^3e^{-\eta\ell},
\end{align*}
where we used Lemma~\ref{lem:zt_bound}, \(\vartheta=C_{\mathrm{bp}}\alpha\), and
\(\eta<2(1-\varepsilon)\).

Combining the three cases and dividing by \eqref{eq:Dl-bounds-proof}, we obtain
\[
\sup_{\ell\ge1}\sum_{z\in\mathbb Z^d}|z|^k|\varphi_\ell(z)|<\infty.
\]

Applying the same bounds to the off-origin mass gives
\[
\sum_{z\neq0}|\varphi_\ell(z)|\le C\alpha
\qquad\text{uniformly in }\ell.
\]
Since \(\sum_z\varphi_\ell(z)=1\), this implies
\[
|\varphi_\ell(0)-1|
=
\Big|\sum_{z\neq0}\varphi_\ell(z)\Big|
\le
\sum_{z\neq0}|\varphi_\ell(z)|
\le
C\alpha.
\]
Together with the preceding moment bounds, this shows that \(M_{d+7}<\infty\), where
\[
M_{k}:=\sup_{\ell\ge1}\sum_{z\in\mathbb Z^d}(1+|z|)^{k}|\varphi_\ell(z)|,
\qquad k\geq 1.
\]
Therefore, for every \(\xi\in\mathbb T^d\), noting that $\varphi_\ell$ is an even function,
\begin{align*}
\widehat\varphi_\ell(\xi)
&=
\sum_{z\in\mathbb Z^d}e^{i\xi\cdot z}\varphi_\ell(z)
=
\varphi_\ell(0)+\sum_{z\neq0}\cos(\xi\cdot z)\varphi_\ell(z) \notag
\ge
\varphi_\ell(0)-\sum_{z\neq0}|\varphi_\ell(z)|
\ge
1-2C\alpha.
\end{align*}
After further decreasing \(\beta_0\), we may assume \(2C\alpha\le1/2\), and then
\begin{equation}\label{eq:hat-positive-half}
\widehat\varphi_\ell(\xi)\ge \frac12
\qquad\text{for all }\ell\ge1,\ \xi\in\mathbb T^d.
\end{equation}

\smallskip
\smallskip
\noindent\textbf{Step 2: symmetry and the quadratic term.}
The graphical construction, the stopping rule defining sausages, and the signed weight \(p(x)\) are
invariant under coordinate permutations and sign changes. Therefore, for every signed permutation \(R\)
of \(\mathbb Z^d\),
\[
\varphi_\ell(Rz)=\varphi_\ell(z).
\]
Hence \(\widehat\varphi_\ell\) is real-valued, even, and invariant under signed permutations. In particular,
\[
\widehat\varphi_\ell(0)=1,
\qquad
\nabla\widehat\varphi_\ell(0)=0.
\]

Since \(\widehat\varphi_\ell\) is even and signed-permutation invariant, its Hessian at \(0\) is a scalar
multiple of the identity:
\[
-\nabla^2\widehat\varphi_\ell(0)=\sigma_\ell I_d,
\qquad
\sigma_\ell:=\sum_{z\in\mathbb Z^d}z_1^2\varphi_\ell(z).
\]
The upper bound \(\sigma_\ell\le \sigma_+\) follows immediately from the \(k=2\) moment bound in Step~1.

We next prove a uniform positive lower bound. By symmetry, under
\(\{\ell(x_1)=\ell,\ b(x_1)=1\}\), the displacement \(y(x_1)\) is uniformly distributed over
\(\{\pm e_1,\dots,\pm e_d\}\). Hence
\[
\sum_{x:\ \ell(x)=\ell,\ b(x)=1}\mu(x)p(x)\,y_1(x)^2
=
\frac{p_1}{d}\,\mu(\ell(x_1)=\ell,\ b(x_1)=1).
\]
Therefore
\begin{align*}
\sum_{x:\ \ell(x)=\ell}\mu(x)p(x)\,y_1(x)^2
&\ge
\frac{p_1}{d}\,\mu(\ell(x_1)=\ell,\ b(x_1)=1)
-
\sum_{x:\ \ell(x)=\ell,\ b(x)\ge3}|y(x)|^2|p(x)|\mu(x)\\
&\ge
\left(\frac{C}{d}-C'\alpha^2\right)\alpha e^{-\eta\ell},
\end{align*}
using Lemma~\ref{lem:canonical-sausages} (2), \eqref{eq:Xi-1-lower}, and the \(k=2\) case of Step~1.
After shrinking \(\beta_0\) again if necessary, the bracket is positive, so
\[
\sum_{x:\ \ell(x)=\ell}\mu(x)p(x)\,y_1(x)^2\ge c\,\alpha e^{-\eta\ell}
\]
for some \(c>0\). Dividing by \eqref{eq:Dl-bounds-proof} gives
\[
\sigma_\ell\ge c\alpha>0
\qquad\text{uniformly in }\ell.
\]
Together with the \(k=2\) upper bound from Step~1, this gives
\begin{equation*}
0<c\alpha\le\sigma_\ell\le C\alpha<\infty.
\end{equation*}
Thus we may take \(\sigma_-:=c\alpha\) and \(\sigma_+:=C\alpha\).
%%%

We now prove a global quadratic lower bound on \(1-\widehat\varphi_\ell\). Since \(\widehat\varphi_\ell\) is real,
\[
1-\widehat\varphi_\ell(\xi)
=
D_\ell^{-1}\sum_{x:\ \ell(x)=\ell}\mu(x)p(x)\bigl(1-\cos(\xi\cdot y(x))\bigr).
\]
Split the sum according to \(b(x)=0,1,2,\ge3\). The \(b(x)=0\) term vanishes because \(y(x)=0\). The
\(b(x)=2\) term is nonnegative because \(p(x)\ge0\) on \(\Xi_2\), so it may be discarded.

For \(b(x)=1\), symmetry over the \(2d\) unit vectors gives
\[
\sum_{x:\ \ell(x)=\ell,\ b(x)=1}\mu(x)p(x)\bigl(1-\cos(\xi\cdot y(x))\bigr)
=
\frac{p_1\,\mu(\ell(x_1)=\ell,\ b(x_1)=1)}{d}\sum_{j=1}^d(1-\cos\xi_j).
\]
Since \(|\xi_j|\le\pi\) and \(1-\cos u\ge \frac{2}{\pi^2}u^2\) on \([-\pi,\pi]\),
\[
\sum_{j=1}^d(1-\cos\xi_j)\ge \frac{2}{\pi^2}|\xi|^2.
\]
Hence, using \eqref{eq:Dl-bounds-proof}, Lemma~\ref{lem:canonical-sausages} (2), and \eqref{eq:Xi-1-lower},
\begin{equation}\label{eq:xi1-positive}
D_\ell^{-1}\!\!\sum_{x:\ \ell(x)=\ell,\ b(x)=1}\mu(x)p(x)\bigl(1-\cos(\xi\cdot y(x))\bigr)
\ge
c_1\alpha|\xi|^2
\end{equation}
for some \(c_1>0\).

For \(b(x)\ge3\), use \(1-\cos u\le u^2/2\), \(|p(x)|\le1\), and the \(k=2\) bound from Step~1:
\begin{align}
\left|
D_\ell^{-1}\!\!\sum_{x:\ \ell(x)=\ell,\ b(x)\ge3}\mu(x)p(x)\bigl(1-\cos(\xi\cdot y(x))\bigr)
\right|
&\le
\frac{|\xi|^2}{2D_\ell}\sum_{x:\ \ell(x)=\ell,\ b(x)\ge3}|y(x)|^2\mu(x)\notag\\
&\le
c_2\alpha^3|\xi|^2
\label{eq:high-branch-negative}
\end{align}
for some \(c_2>0\). After shrinking \(\beta_0\) once more so that \(c_2\alpha^2\le c_1/2\),
\eqref{eq:xi1-positive} and \eqref{eq:high-branch-negative} yield
\begin{equation}\label{eq:one-minus-hat}
1-\widehat\varphi_\ell(\xi)\ge c_0\alpha|\xi|^2
\qquad\text{for all }\ell\ge1,\ \xi\in\mathbb T^d,
\end{equation}
with \(c_0:=c_1/2\).

Since \(\widehat\varphi_\ell(\xi)\ge1/2\) by \eqref{eq:hat-positive-half}, \eqref{eq:one-minus-hat} implies
\[
|\widehat\varphi_\ell(\xi)|=\widehat\varphi_\ell(\xi)\le 1
\qquad\text{for all }\xi\in\mathbb T^d,
\]
and equality holds if and only if \(\xi=0\). This proves item (1).

\smallskip
\noindent\textbf{Step 3: Taylor expansion of \(\Gamma_\ell=\log\widehat\varphi_\ell\).}
Because \(\varphi_\ell\) is symmetric,
\[
\widehat\varphi_\ell(\xi)=\sum_{z\in\mathbb Z^d}\cos(\xi\cdot z)\,\varphi_\ell(z).
\]
We first expand \(\widehat\varphi_\ell\) itself. For every \(u\in\mathbb R\),
\[
\cos u = 1-\frac{u^2}{2}+\frac{u^4}{24}+r_6(u),
\qquad
|r_6(u)|\le \frac{|u|^6}{720}.
\]
Applying this with \(u=\xi\cdot z\) and summing against \(\varphi_\ell(z)\), we get
\begin{align}
\widehat\varphi_\ell(\xi)
&=
1-\frac12\sum_{z}\varphi_\ell(z)(\xi\cdot z)^2
+\frac1{24}\sum_{z}\varphi_\ell(z)(\xi\cdot z)^4
+\widetilde R_{\ell,6}(\xi),
\label{eq:hatphi-taylor}
\end{align}
where
\begin{equation*}
|\widetilde R_{\ell,6}(\xi)|
\le
\frac{|\xi|^6}{720}\sum_z |z|^6|\varphi_\ell(z)|
\le
\frac{M_6}{720}|\xi|^6.
\end{equation*}

By the signed-permutation symmetry, the quadratic form
\[
Q_\ell(\xi):=\sum_z\varphi_\ell(z)(\xi\cdot z)^2
\]
must equal \(\sigma_\ell|\xi|^2\). Likewise, the quartic form
\[
T_{4,\ell}(\xi):=\sum_z\varphi_\ell(z)(\xi\cdot z)^4
\]
is invariant under signed permutations, so it lies in the two-dimensional space spanned by
\(|\xi|^4\) and \(\sum_{j=1}^d\xi_j^4\). Therefore there exist \(a'_\ell,b'_\ell\in\mathbb R\) such that
\[
\frac1{24}T_{4,\ell}(\xi)=a'_\ell|\xi|^4+b'_\ell\sum_{j=1}^d\xi_j^4.
\]
Since the coefficients of \(T_{4,\ell}\) are bounded by \(M_4\), we have
\[
|a'_\ell|+|b'_\ell|\le C(M_4,d)\le C.
\]
Substituting this into \eqref{eq:hatphi-taylor} yields
\begin{equation}\label{eq:hatphi-expansion}
\widehat\varphi_\ell(\xi)
=
1-\frac{\sigma_\ell}{2}|\xi|^2
+a'_\ell|\xi|^4
+b'_\ell\sum_{j=1}^d\xi_j^4
+\widetilde R_{\ell,6}(\xi),
\qquad
|\widetilde R_{\ell,6}(\xi)|\le C|\xi|^6.
\end{equation}

We now pass from \(\widehat\varphi_\ell\) to \(\Gamma_\ell=\log\widehat\varphi_\ell\). Define
\[
u_\ell(\xi):=1-\widehat\varphi_\ell(\xi).
\]
By \eqref{eq:one-minus-hat} and the \(k=2\) moment bound,
\[
0\le u_\ell(\xi)\le C|\xi|^2
\qquad\text{for all }\ell,\ \xi.
\]
Choose \(\rho\in(0,1]\) so small that
$C\rho^2\le 1/4$.
Then for \(|\xi|\le \rho\), one has \(0\le u_\ell(\xi)\le1/4\), uniformly in \(\ell\). Hence
\[
\Gamma_\ell(\xi)=\log(1-u_\ell(\xi))
=
-u_\ell(\xi)-\frac12u_\ell(\xi)^2+\psi(u_\ell(\xi)),
\]
where
\[
|\psi(u)|\le 2|u|^3
\qquad (|u|\le1/4).
\]
Using \eqref{eq:hatphi-expansion},
\[
u_\ell(\xi)
=
\frac{\sigma_\ell}{2}|\xi|^2
-a'_\ell|\xi|^4
-b'_\ell\sum_{j=1}^d\xi_j^4
+O(|\xi|^6).
\]
Therefore
\[
u_\ell(\xi)^2
=
\frac{\sigma_\ell^2}{4}|\xi|^4+O(|\xi|^6),
\qquad
\psi(u_\ell(\xi))=O(|\xi|^6),
\]
uniformly in \(\ell\). Consequently,
\[
\Gamma_\ell(\xi)
=
-\frac{\sigma_\ell}{2}|\xi|^2
+
\left(a'_\ell-\frac{\sigma_\ell^2}{8}\right)|\xi|^4
+
b'_\ell\sum_{j=1}^d\xi_j^4
+
R_{\ell,6}(\xi),
\qquad
|R_{\ell,6}(\xi)|\le C_6|\xi|^6
\]
for \(|\xi|\le\rho\), with \(C_6<\infty\) independent of \(\ell\). Set
\[
\mathfrak a_\ell:=a'_\ell-\frac{\sigma_\ell^2}{8},
\qquad
\mathfrak b_\ell:=b'_\ell.
\]
Since \(\sigma_\ell\in[\sigma_-,\sigma_+]\) and \(|a'_\ell|+|b'_\ell|\le C\), we may choose \(A_4<\infty\) so that
\[
|\mathfrak a_\ell|+|\mathfrak b_\ell|\le A_4
\qquad\text{for all }\ell.
\]
This proves item (2).

\smallskip
\noindent\textbf{Step 4: spectral gap away from \(0\).}
With the choice of \(\rho\) from Step~3, \eqref{eq:one-minus-hat} gives
\[
1-\widehat\varphi_\ell(\xi)\ge c_0\alpha|\xi|^2\ge c_0\alpha\rho^2
\qquad\text{for all }|\xi|\ge \rho.
\]
Set
\[
\gamma:=c_0\alpha\rho^2.
\]
Then
\[
\widehat\varphi_\ell(\xi)\le 1-\gamma
\qquad\text{for all }\xi\in\mathbb T^d\setminus B_{\mathbb T}(0,\rho).
\]
Since \(\widehat\varphi_\ell(\xi)\ge1/2\) by \eqref{eq:hat-positive-half}, this implies
\[
|\widehat\varphi_\ell(\xi)|\le 1-\gamma
\qquad\text{for all }\xi\in\mathbb T^d\setminus B_{\mathbb T}(0,\rho).
\]
This proves item (3).

\qed

\smallskip
\subsubsection*{Proof of Proposition \ref{prop:Fourier-six-exp}}
We compare each mixed convolution kernel with
its Gaussian expansion and control the error in both \(\ell^\infty\) and \(\ell^1\). The
low-frequency region is governed by the Taylor expansion, while the spectral gap makes the
high-frequency contribution exponentially small.
Fix an admissible profile \(\vs=(s_1,\dots,s_t)\), and write

\[
r:=\sum_{i=1}^t s_i,\qquad
\widehat\Phi_{\vs}(\xi)
=
\prod_{i=1}^t \widehat\varphi_i(\xi)^{s_i}
=
e^{r\Gamma_{\vs}(\xi)},
\qquad
\Gamma_{\vs}(\xi):=\frac1r\sum_{i=1}^t s_i\Gamma_i(\xi).
\]
On \(B(0,\rho)\), Lemma~\ref{lem:varphi-input} gives
\[
\Gamma_{\vs}(\xi)
=
-P_{\vs}(\xi)+Q_{4,\vs}(\xi)+R_{\vs,6}(\xi),
\]
where
\[
P_{\vs}(\xi):=\frac{\sigma_{\vs}}2|\xi|^2,
\qquad
Q_{4,\vs}(\xi):=\mathfrak a_{\vs}|\xi|^4+\mathfrak b_{\vs}\sum_{j=1}^d \xi_j^4,
\qquad
R_{\vs,6}(\xi):=\frac1r\sum_{i=1}^t s_iR_{i,6}(\xi).
\]
Decrease \(\rho\) if necessary so that
\[
A_4\rho^2+C_6\rho^4\le \frac{\sigma_-}{8}.
\]
The conclusions of Lemma~\ref{lem:varphi-input} remain valid after this decrease: part~(2)
restricts to the smaller ball, while \eqref{eq:hat-positive-half} and
\eqref{eq:one-minus-hat} preserve part~(3), after replacing \(\gamma\) by
\(\min\{\gamma,c_0\alpha\rho^2/4\}\) if necessary.
Choose \(\chi\in C^\infty(\mathbb T^d)\) such that

\[
0\le \chi\le1,\qquad
\chi\equiv1 \text{ on }B(0,\rho/2),\qquad
\supp\chi\subset B(0,\rho)\subset(-\pi,\pi)^d.
\]
Since \(\operatorname{supp}\chi\subset(-\pi,\pi)^d\), we also regard \(\chi\) as a compactly supported smooth function on \(\mathbb R^d\). 

Fix \(N:=d+1\). We use the following standard weighted Fourier estimate: for every
\(F\in W^{N,1}(\mathbb T^d)\), or every \(F\in W^{N,1}(\mathbb R^d)\),
\begin{equation}\label{eq:weighted-Fourier}
(1+|x|)^N\left|
\frac1{(2\pi)^d}\int F(\xi)e^{-ix\cdot\xi}\,d\xi
\right|
\le
C_N\sum_{|\alpha|\le N}\|\partial^\alpha F\|_{L^1},
\end{equation}
and, since \(N>d\),
\begin{equation}\label{eq:weighted-Fourier-l1}
\left\|
\frac1{(2\pi)^d}\int F(\xi)e^{-ix\cdot\xi}\,d\xi
\right\|_{\ell^1(\mathbb Z^d)}
\le
C_N'\sum_{|\alpha|\le N}\|\partial^\alpha F\|_{L^1}.
\end{equation}

\smallskip
\noindent\textbf{Uniform derivative bounds.}
For every multi-index \(\alpha\) with \(|\alpha|\le d+7\), the Fourier series of
\(\widehat\varphi_\ell\) may be differentiated termwise, since
\[
\sum_{z\in\mathbb Z^d}|z^\alpha|\,|\varphi_\ell(z)|
\le
\sum_{z\in\mathbb Z^d}(1+|z|)^{d+7}|\varphi_\ell(z)|
\le
M_{d+7}.
\]
Hence
\[
\partial^\alpha \widehat\varphi_\ell(\xi)
=
\sum_{z\in\mathbb Z^d}(iz)^\alpha e^{i\xi\cdot z}\varphi_\ell(z),
\]
so
\begin{equation}\label{eq:hatvarphi-derivative-bound}
\sup_{\ell\ge1}\sup_{\xi\in\mathbb T^d}
|\partial^\alpha \widehat\varphi_\ell(\xi)|
\le
M_{d+7}
\qquad (|\alpha|\le d+7).
\end{equation}
On \(B(0,\rho)\), Lemma~\ref{lem:varphi-input}(2) and the bounds
\(\sigma_\ell\in[\sigma_-,\sigma_+]\), \(|\mathfrak a_\ell|+|\mathfrak b_\ell|\le A_4\),
give
\[
|\Gamma_\ell(\xi)|\le C_0|\xi|^2
\qquad (\ell\ge1,\ |\xi|\le\rho),
\]
hence
\[
|\widehat\varphi_\ell(\xi)|=e^{\Re\Gamma_\ell(\xi)}\ge e^{-C_0\rho^2}>0
\qquad (\ell\ge1,\ |\xi|\le\rho).
\]
Combining this with \eqref{eq:hatvarphi-derivative-bound} and the Fa\`a di Bruno formula for
\(\Gamma_\ell=\log \widehat\varphi_\ell\), we get
\begin{equation*}
\sup_{\ell\ge1}\sup_{|\xi|\le\rho}
|\partial^\alpha \Gamma_\ell(\xi)|
\le
C_\alpha
\qquad (|\alpha|\le d+7).
\end{equation*}
Since \(R_{\ell,6}\) is the remainder after subtracting the quadratic and quartic Taylor
polynomials of \(\Gamma_\ell\), Taylor's theorem gives
\begin{equation}\label{eq:R6-derivative-bound}
|\partial^\alpha R_{\ell,6}(\xi)|
\le
C_\alpha |\xi|^{(6-|\alpha|)_+}
\qquad (\ell\ge1,\ |\xi|\le\rho,\ |\alpha|\le N).
\end{equation}
Averaging over the profile, the same bound holds for \(R_{\vs,6}\).

\smallskip

By Fourier inversion,
\[
\Phi_{\vs}(x)
=
\frac1{(2\pi)^d}\int_{\mathbb T^d}\widehat\Phi_{\vs}(\xi)e^{-ix\cdot\xi}\,d\xi
=:I_{\mathrm{in}}(x)+I_{\mathrm{out}}(x),
\]
where
\[
I_{\mathrm{in}}(x):=
\frac1{(2\pi)^d}\int_{\mathbb T^d}\chi(\xi)e^{r\Gamma_{\vs}(\xi)}e^{-ix\cdot\xi}\,d\xi,
\qquad
I_{\mathrm{out}}(x):=
\frac1{(2\pi)^d}\int_{\mathbb T^d}(1-\chi(\xi))\widehat\Phi_{\vs}(\xi)e^{-ix\cdot\xi}\,d\xi.
\]
Also
\[
H_{\vs}(x)
=
\frac1{(2\pi)^d}\int_{\mathbb R^d}e^{-rP_{\vs}(\xi)}e^{-ix\cdot\xi}\,d\xi,
\]
and, since $i\partial_{x_j}e^{-ix\cdot\xi}=\xi_j e^{-ix\cdot\xi}$,
\[
rQ_{4,\vs}(i\nabla_x)H_{\vs}(x)
=
\frac1{(2\pi)^d}\int_{\mathbb R^d}rQ_{4,\vs}(\xi)e^{-rP_{\vs}(\xi)}e^{-ix\cdot\xi}\,d\xi.
\]
Thus
\[
H_{\vs}(x)+rQ_{4,\vs}(i\nabla_x)H_{\vs}(x)
=:J_{\mathrm{in}}(x)+J_{\mathrm{tail}}(x),
\]
where
\[
J_{\mathrm{in}}(x):=
\frac1{(2\pi)^d}\int_{\mathbb R^d}
\chi(\xi)e^{-rP_{\vs}(\xi)}\bigl(1+rQ_{4,\vs}(\xi)\bigr)e^{-ix\cdot\xi}\,d\xi,
\]
\[
J_{\mathrm{tail}}(x):=
\frac1{(2\pi)^d}\int_{\mathbb R^d}
(1-\chi(\xi))e^{-rP_{\vs}(\xi)}\bigl(1+rQ_{4,\vs}(\xi)\bigr)e^{-ix\cdot\xi}\,d\xi.
\]
We therefore write
\[
\mathrm{Err}_{\vs}
=
\bigl(I_{\mathrm{in}}-J_{\mathrm{in}}\bigr)
+I_{\mathrm{out}}
-J_{\mathrm{tail}}
=:E_{\mathrm{in}}+E_{\mathrm{out}}+E_{\mathrm{tail}}.
\]

We estimate the terms in the order \(E_{\mathrm{out}},E_{\mathrm{tail}},E_{\mathrm{in}}\): the
first two are exponentially small, while the inner term determines the polynomial remainder.

\smallskip
\noindent\textbf{High-frequency term.}
Since \(\supp(1-\chi)\subset\mathbb T^d\setminus B(0,\rho/2)\),
\eqref{eq:hat-positive-half} and \eqref{eq:one-minus-hat} give
\[
|\widehat\varphi_i(\xi)|\le1-\gamma.
\]
Write \(\widehat\Phi_{\mathbf s}\) as a product of \(r\) factors, $\widehat\Phi_{\mathbf s}=u_1\cdots u_r$, 
where each \(u_\nu\) is one of the functions \(\widehat\varphi_i\). On \(\operatorname{supp}(1-\chi)\), every undifferentiated factor is bounded by \(1-\gamma\), while every derivative of order at most \(N\) is bounded by \(M_{d+7}\). By Leibniz' rule, for each multi-index \(\alpha\) with \(|\alpha|\le N\),
\[
\sup_{\supp(1-\chi)}|\partial^\alpha \widehat\Phi_{\vs}|
\le
C_\alpha r^{|\alpha|}(1-\gamma)^{r-|\alpha|}
\le
C_\alpha' e^{-cr}
\qquad (|\alpha|\le N).
\]
for some $c=c(\gamma)>0$. Hence
\[
\sum_{|\alpha|\le N}\left\|\partial^\alpha\bigl((1-\chi)\widehat\Phi_{\mathbf s}\bigr)\right\|_{L^1(\mathbb T^d)}
\le
C e^{-cr},
\]
and \eqref{eq:weighted-Fourier}--\eqref{eq:weighted-Fourier-l1} imply
\begin{equation}\label{eq:Eout-bounds-final}
|E_{\mathrm{out}}(x)|\le Ce^{-cr}(1+|x|)^{-N},\qquad
\|E_{\mathrm{out}}\|_{\ell^\infty(\mathbb Z^d)}\le C e^{-c r},
\qquad
\|E_{\mathrm{out}}\|_{\ell^1(\mathbb Z^d)}\le C e^{-c r}.
\end{equation}

\smallskip
\noindent\textbf{Gaussian tail term.}
Set
\[
F_r(\xi):=(1-\chi(\xi))e^{-rP_{\vs}(\xi)}\bigl(1+rQ_{4,\vs}(\xi)\bigr).
\]
On \(\supp(1-\chi)\), one has \(|\xi|\ge\rho/2\), so
\[
e^{-rP_{\vs}(\xi)}
\le
e^{-\sigma_-r\rho^2/16}\,e^{-\sigma_-r|\xi|^2/4}.
\]
Since \(Q_{4,\vs}\) is a quartic polynomial with uniformly bounded coefficients, every
derivative of \(F_r\) of order at most \(N\) is bounded by a finite linear combination of
terms of the form
\[
Cr^C(1+|\xi|)^C e^{-\sigma_-r\rho^2/16}e^{-\sigma_-r|\xi|^2/4}.
\]
Hence
\[
\sum_{|\alpha|\le N}\|\partial^\alpha F_r\|_{L^1(\mathbb R^d)}\le Ce^{-cr}.
\]
for some \(c>0\).
Applying \eqref{eq:weighted-Fourier}--\eqref{eq:weighted-Fourier-l1},
\begin{equation}\label{eq:Etail-bounds-final}
|E_{\mathrm{tail}}(x)|=|J_{\mathrm{tail}}(x)|\le Ce^{-cr}(1+|x|)^{-N},
\quad
\|E_{\mathrm{tail}}\|_{\ell^1(\mathbb Z^d)}\le Ce^{-cr},
\quad
\|E_{\mathrm{tail}}\|_{\ell^\infty(\mathbb Z^d)}\le Ce^{-cr}.
\end{equation}

\smallskip
\noindent\textbf{Inner error term.}
Set
\[
q_r(\eta):=rQ_{4,\vs}\!\left(\frac{\eta}{\sqrt r}\right)=\frac1r\,Q_{4,\vs}(\eta),
\qquad
u_r(\eta):=rR_{\vs,6}\!\left(\frac{\eta}{\sqrt r}\right).
\]
Then, for \(|\alpha|\le N\),
\begin{equation}\label{eq:q-u-derivative-final}
|\partial^\alpha q_r(\eta)|\le C_\alpha r^{-1}(1+|\eta|)^4,
\qquad
|\partial^\alpha u_r(\eta)|\le C_\alpha r^{-2}(1+|\eta|)^6,
\end{equation}
the second bound following from \eqref{eq:R6-derivative-bound}.
Changing variables \(\eta=\sqrt r\,\xi\), we obtain
\[
E_{\mathrm{in}}(x)
=
r^{-d/2}\frac1{(2\pi)^d}
\int_{\mathbb R^d}
B_r(\eta)e^{-i(x/\sqrt r)\cdot\eta}\,d\eta,
\]
where
\[
B_r(\eta)
:=
\chi\!\left(\frac{\eta}{\sqrt r}\right)e^{-P_{\vs}(\eta)}
\Bigl[e^{q_r(\eta)+u_r(\eta)}-1-q_r(\eta)\Bigr].
\]
On \(\supp \chi(\eta/\sqrt r)\), one has \(|\eta|\le \rho\sqrt r\), hence
\[
|q_r(\eta)|\le A_4\rho^2|\eta|^2,
\qquad
|u_r(\eta)|\le C_6\rho^4|\eta|^2,
\qquad
|q_r(\eta)+u_r(\eta)|\le \frac{\sigma_-}{8}|\eta|^2.
\]
Writing \(z_r:=q_r+u_r\) and using
\[
e^{z_r}-1-q_r
=
u_r+\int_0^1(1-\theta)z_r^2e^{\theta z_r}\,d\theta,
\]
together with \eqref{eq:q-u-derivative-final}, Leibniz' rule, and Fa\`a di Bruno,
we obtain for \(|\alpha|\le N\),
\[
|\partial^\alpha(e^{z_r}-1-q_r)|
\le
C_\alpha r^{-2}(1+|\eta|)^M e^{\sigma_-|\eta|^2/8}
\]
for some \(M=M(d)\). Multiplying by \(\chi(\eta/\sqrt r)e^{-P_{\vs}(\eta)}\), and using
\(e^{-P_{\vs}(\eta)}\le e^{-\sigma_-|\eta|^2/2}\), gives
\begin{equation*}
|\partial^\alpha B_r(\eta)|
\le
C_\alpha r^{-2}(1+|\eta|)^M e^{-c|\eta|^2}
\qquad (|\alpha|\le N).
\end{equation*}
Therefore
\[
\sum_{|\alpha|\le N}\|\partial^\alpha B_r\|_{L^1(\mathbb R^d)}\le Cr^{-2}.
\]
Applying \eqref{eq:weighted-Fourier} to the scaled integral,
\[
\left(1+\frac{|x|}{\sqrt r}\right)^N |E_{\mathrm{in}}(x)|
\le
Cr^{-d/2}\sum_{|\alpha|\le N}\|\partial^\alpha B_r\|_{L^1}
\le
Cr^{-d/2-2}.
\]
Hence
\[
|E_{\mathrm{in}}(x)|
\le
Cr^{-d/2-2}\left(1+\frac{|x|}{\sqrt r}\right)^{-N}.
\]
Since \(N>d\),
\begin{equation}\label{eq:Ein-bounds-final}
\|E_{\mathrm{in}}\|_{\ell^\infty(\mathbb Z^d)}\le Cr^{-d/2-2},
\qquad
\|E_{\mathrm{in}}\|_{\ell^1(\mathbb Z^d)}\le Cr^{-2}.
\end{equation}

\smallskip
Combining \eqref{eq:Eout-bounds-final}, \eqref{eq:Etail-bounds-final},
and \eqref{eq:Ein-bounds-final}, and using \(e^{-cr}\le Cr^{-2}\) and
\(e^{-cr}\le Cr^{-2-d/2}\), we obtain
\[
\|\mathrm{Err}_{\vs}\|_{\ell^1(\mathbb Z^d)}\le C_{\mathrm F}r^{-2},
\qquad
\|\mathrm{Err}_{\vs}\|_{\ell^\infty(\mathbb Z^d)}\le C_{\mathrm F}r^{-2-d/2}.
\]
This proves \eqref{eq:ErrBounds}, and hence \eqref{eq:main-expansion}.
\qed

\smallskip
\subsubsection*{Local kernel estimates}

We now convert the profilewise local-limit expansion into estimates for the renewal-averaged
kernel \(\IK\). The pointwise remainder estimate controls the error in the moderate-deviation
positivity argument, while the \(\ell^1\)-estimate controls the remainder in the weighted
exterior-mass argument.

\begin{lemma}[Moderate-deviation positivity]\label{lem:Phi-positive-moddev}
Fix \(A>0\). There exists \(r_0(A)<\infty\) such that, for every admissible profile
\(\mathbf s=(s_1,\ldots,s_t)\) satisfying
\[
\sum_{i=1}^t s_i=r\ge r_0(A)
\]
and every \(x\in\mathbb Z^d\) satisfying
\[
|x|\le A\sqrt{r\log\log r},
\]
one has
\[
\Phi_{\mathbf s}(x)\ge \frac12 H_{\mathbf s}(x)>0.
\]
\end{lemma}

\begin{proof}
The expansion \eqref{eq:main-expansion}, Lemma~\ref{lem:derivative-bound}, and
\eqref{eq:ErrBounds} give
\[
\Phi_{\mathbf s}(x)
\ge
\left[
1-\frac{C_{\mathrm{der}}}{r}
\left(1+\frac{|x|^4}{r^2}\right)
\right]H_{\mathbf s}(x)
-C_{\mathrm F}r^{-2-d/2}.
\]
On the stated scale,
\[
\frac{C_{\mathrm{der}}}{r}
\left(1+\frac{|x|^4}{r^2}\right)
\le
\frac{C}{r}\bigl(1+(\log\log r)^2\bigr)=o(1).
\]
Moreover, uniformly in the admissible profile,
\[
H_{\mathbf s}(x)
\ge
(2\pi\sigma_+r)^{-d/2}
\exp\!\left(-\frac{A^2}{2\sigma_-}\log\log r\right)
\ge
c_A r^{-d/2}(\log r)^{-B_A}
\]
for suitable \(c_A>0\) and \(B_A<\infty\). Consequently,
$r^{-2-d/2}=o\bigl(H_{\mathbf s}(x)\bigr)$
uniformly on this scale. Increasing \(r_0(A)\) proves the claim.
\end{proof}

We combine this with 
Lemma~\ref{lem:LDP}. Recall that
\(\bar\ell_\kappa=\tilde{\mathbb E}_\kappa[L_1]\) is the mean sausage length under the tilted
law.

\begin{proof}[Proof of Proposition~\ref{prop:local-kernel-input}]
Fix
\[
\delta\in(0,\bar\ell_\kappa/2),\qquad
a_-:=\frac1{\bar\ell_\kappa+\delta},\qquad
a_+:=\frac1{\bar\ell_\kappa-\delta},
\]
and, for each \(s\), set
\[
\mathrm{SP}_s:=\{r\in\mathbb N:a_-s\le r\le a_+s\},\quad
p_r(s):=\tilde{\mathbb P}_\kappa(S_r=s),\quad
K_{r,s}(x):=
\begin{cases}
K(x;r,s),&p_r(s)>0,\\
0,&p_r(s)=0.
\end{cases}
\]
Thus
\begin{equation}\label{eq:local-input-tilt}
\IK(x,s)
=
e^{-(1+\kappa)s}\sum_{r\ge1}p_r(s)K_{r,s}(x).
\end{equation}
Lemma~\ref{lem:LDP} supplies constants \(c_{\mathrm{ld}}>0\) and \(s_0<\infty\) such that,
for \(s\ge s_0\),
\begin{equation}\label{eq:local-input-renewal}
E_s:=\sum_{r\notin\mathrm{SP}_s}p_r(s)\le e^{-c_{\mathrm{ld}}s},
\qquad
\sum_{r\in\mathrm{SP}_s}p_r(s)\ge\frac1{3\bar\ell_\kappa},
\end{equation}
and also
\begin{equation}\label{eq:local-input-renewal-upper}
\sup_{s\ge1}\sum_{r\ge1}p_r(s)<\infty.
\end{equation}
We will also use
\begin{equation}\label{eq:K-pointwise-one}
|K_{r,s}(x)|\le1.
\end{equation}
Indeed, Lemma~\ref{lem:varphi-input} gives
\(\lvert\widehat\varphi_\ell\rvert\le1\), so Fourier inversion gives
\(\lvert\Phi_{\mathbf s}(x)\rvert\le1\) for every admissible profile; conditional averaging
then proves \eqref{eq:K-pointwise-one}.

\smallskip
\noindent\textit{Proof of {\rm(a)}.}
Fix \(A>0\), and suppose that
\[
|u|\le A\sqrt{s\log\log s}.
\]
Because \(r\asymp s\) uniformly over \(r\in\mathrm{SP}_s\), there exists
\(A_1=A_1(A,\delta)\) such that, for all sufficiently large \(s\),
\[
|u|\le A_1\sqrt{r\log\log r}
\qquad (r\in\mathrm{SP}_s).
\]
Lemma~\ref{lem:Phi-positive-moddev} therefore applies to every admissible profile satisfying
\[
\sum_i s_i=r,\qquad \sum_i i\,s_i=s.
\]
In addition, using \(r\in[a_-s,a_+s]\) and
\(\sigma_{\mathbf s}\in[\sigma_-,\sigma_+]\), we obtain constants \(c_*>0\) and
\(B_{\mathrm{md}}(A)<\infty\) such that
\[
H_{\mathbf s}(u)
\ge
(2\pi a_+\sigma_+s)^{-d/2}
\exp\!\left(-\frac{A^2}{2a_-\sigma_-}\log\log s\right)
\ge
c_*s^{-d/2}(\log s)^{-B_{\mathrm{md}}(A)}.
\]
It follows that
\begin{equation}\label{eq:typical-K-lower}
K_{r,s}(u)
\ge
\frac{c_*}{2}s^{-d/2}(\log s)^{-B_{\mathrm{md}}(A)}
\qquad (r\in\mathrm{SP}_s,\ p_r(s)>0).
\end{equation}

Write
\[
T_s(u):=\sum_{r\in\mathrm{SP}_s}p_r(s)K_{r,s}(u),
\qquad
R_s(u):=\sum_{r\notin\mathrm{SP}_s}p_r(s)K_{r,s}(u).
\]
Equations \eqref{eq:local-input-renewal} and \eqref{eq:typical-K-lower} imply
\begin{equation}\label{eq:typical-mass-lower}
T_s(u)
\ge
\frac{c_*}{6\bar\ell_\kappa}
s^{-d/2}(\log s)^{-B_{\mathrm{md}}(A)}.
\end{equation}
On the other hand, \eqref{eq:K-pointwise-one} and \eqref{eq:local-input-renewal} give
\[
|R_s(u)|\le E_s\le e^{-c_{\mathrm{ld}}s}.
\]
After increasing the threshold in \(s\), the last quantity is at most
\(\frac12T_s(u)\). Substitution in \eqref{eq:local-input-tilt} yields
\[
\IK(u,s)
\ge
\frac12e^{-(1+\kappa)s}T_s(u)
\ge
\frac{c_*}{12\bar\ell_\kappa}
s^{-d/2}(\log s)^{-B_{\mathrm{md}}(A)}e^{-(1+\kappa)s},
\]
which proves (a).

\smallskip
\noindent\textit{Proof of {\rm(b)}.}
Fix \(A>0\) and \(\zeta\in(0,1)\), and continue to assume
\[
|u|\le A\sqrt{s\log\log s}.
\]
Let \(r\in\mathrm{SP}_s\) satisfy \(p_r(s)>0\), and fix an admissible profile
\(\mathbf s\) satisfying
\[
\sum_i s_i=r,\qquad \sum_i i\,s_i=s.
\]
For \(z\in B_\infty(u,k)\), one has \(|z-u|\le\sqrt d\,k\) and
\[
|z|^2-|u|^2\ge-2|u|\,|z-u|.
\]
Since \(\sigma_{\mathbf s}\ge\sigma_-\),
\[
H_{\mathbf s}(z)
\le
\exp\!\left(\frac{\sqrt d\,|u|}{r\sigma_-}k\right)H_{\mathbf s}(u).
\]
Uniformly over \(r\in\mathrm{SP}_s\), one has \(|u|/r\to0\). Thus, for all large \(s\), $e^{\frac{\sqrt d\,|u|}{r\sigma_-}}\le\zeta^{-1/2}$,
and hence
\begin{equation}\label{eq:weighted-H-ball}
\sum_{k\ge1}\zeta^k
\sum_{z\in B_\infty(u,k)}H_{\mathbf s}(z)
\le
C_{\mathrm{ball}}(d,\zeta)H_{\mathbf s}(u).
\end{equation}

If \(z\in B_\infty(u,k)\), then
\[
1+\frac{|z|^4}{r^2}
\le
C\left(1+\frac{|u|^4}{r^2}+k^4\right),
\qquad
\frac{|u|^4}{r^2}\le C(\log\log r)^2.
\]
Lemma~\ref{lem:derivative-bound} and the argument leading to
\eqref{eq:weighted-H-ball} therefore give
\begin{align*}
&\sum_{k\ge1}\zeta^k
\sum_{z\in B_\infty(u,k)}
\left|rQ_{4,\mathbf s}(i\nabla)H_{\mathbf s}(z)\right|\\
&\qquad\le
\frac{C(1+(\log\log r)^2)}{r}
\left[
\sum_{k\ge1}\zeta^{k/2}(2k+1)^d(1+k^4)
\right]H_{\mathbf s}(u)
=o\bigl(H_{\mathbf s}(u)\bigr).
\end{align*}
Similarly, by the pointwise error bound in \eqref{eq:ErrBounds},
\[
\sum_{k\ge1}\zeta^k
\sum_{z\in B_\infty(u,k)}|\mathrm{Err}_{\mathbf s}(z)|
\le
C_{\mathrm F}r^{-2-d/2}
\sum_{k\ge1}\zeta^k(2k+1)^d
=o\bigl(H_{\mathbf s}(u)\bigr),
\]
where the last comparison follows from the moderate-deviation lower bound on
\(H_{\mathbf s}(u)\) used above. Increasing the threshold once more, both of the last two
displays are at most
\(\frac12C_{\mathrm{ball}}(d,\zeta)H_{\mathbf s}(u)\).
Together with \eqref{eq:main-expansion}, \eqref{eq:weighted-H-ball}, and
Lemma~\ref{lem:Phi-positive-moddev}, this gives
\begin{equation}\label{eq:profile-ball-domination}
\sum_{k\ge1}\zeta^k
\sum_{z\in B_\infty(u,k)}|\Phi_{\mathbf s}(z)|
\le
4C_{\mathrm{ball}}(d,\zeta)\Phi_{\mathbf s}(u).
\end{equation}
Conditional averaging in \eqref{eq:profile-ball-domination} and the triangle inequality yield
\begin{equation}\label{eq:K-ball-domination}
\sum_{k\ge1}\zeta^k
\sum_{z\in B_\infty(u,k)}|K_{r,s}(z)|
\le
4C_{\mathrm{ball}}(d,\zeta)K_{r,s}(u)
\qquad (r\in\mathrm{SP}_s,\ p_r(s)>0).
\end{equation}

For the typical and atypical contributions defined above,
\eqref{eq:typical-mass-lower} and \eqref{eq:local-input-renewal} again imply
\(E_s\le\frac12T_s(u)\) for all large \(s\). Consequently,
\[
\IK(u,s)\ge\frac12e^{-(1+\kappa)s}T_s(u).
\]
On the other hand, \eqref{eq:K-ball-domination} and \eqref{eq:K-pointwise-one} give
\begin{align*}
\sum_{k\ge1}\zeta^k
\sum_{z\in B_\infty(u,k)}|\IK(z,s)|
&\le
e^{-(1+\kappa)s}
\left[
4C_{\mathrm{ball}}(d,\zeta)T_s(u)
+E_s\sum_{k\ge1}\zeta^k(2k+1)^d
\right]\\
&\le
\frac92C_{\mathrm{ball}}(d,\zeta)e^{-(1+\kappa)s}T_s(u),
\end{align*}
where we used
\(\sum_{k\ge1}\zeta^k(2k+1)^d\le C_{\mathrm{ball}}(d,\zeta)\).
Combining the last two displays proves (b).

\smallskip
\noindent\textit{Proof of {\rm(c)}.}
Fix \(A>0\) and \(\zeta\in(0,1)\), and put
\[
\mathcal R_s:=A\sqrt{s\log\log s}.
\]
For \(r\ge1\), define
\[
U_{r,s}
:=
\sum_{|u|>\mathcal R_s}\sum_{b\ge0}\zeta^b
\sum_{z\in B_\infty(u,b)}|K_{r,s}(z)|.
\]
By \eqref{eq:local-input-tilt} and the triangle inequality, the expression on the
left-hand side of (c) is at most
\begin{equation}\label{eq:outside-reduction}
e^{-(1+\kappa)s}\sum_{r\ge1}p_r(s)U_{r,s}.
\end{equation}

We first treat \(r\in\mathrm{SP}_s\), where the expansion gives a Gaussian-tail bound with an
\(r^{-2}\) remainder; the atypical renewal counts will require only the uniform
\(\ell^1\)-estimate. Standard lattice Gaussian-tail estimates, uniform in
\(\sigma_{\mathbf s}\in[\sigma_-,\sigma_+]\), give, for \(L\ge2\sqrt r\),

\[
\sum_{|z|>L}H_{\mathbf s}(z)
\le
C\exp\!\left(-\frac{L^2}{4\sigma_+r}\right)
\]
and
\[
\sum_{|z|>L}
\left(1+\frac{|z|^4}{r^2}\right)H_{\mathbf s}(z)
\le
C\exp\!\left(-\frac{L^2}{8\sigma_+r}\right).
\]
These bounds follow, for example, by comparing each lattice sum with the corresponding
Gaussian integral and absorbing the resulting polynomial factor into the weaker exponential.
Using \eqref{eq:main-expansion}, Lemma~\ref{lem:derivative-bound}, and the
\(\ell^1\)-error estimate in \eqref{eq:ErrBounds}, we obtain
\[
\sum_{|z|>L}|\Phi_{\mathbf s}(z)|
\le
C_{\mathrm G}\exp\!\left(-\frac{L^2}{8\sigma_+r}\right)
+C_{\mathrm G}r^{-2}.
\]
Conditional averaging therefore yields
\begin{equation}\label{eq:K-tail-local-input}
\sum_{|z|>L}|K_{r,s}(z)|
\le
C_{\mathrm G}\exp\!\left(-\frac{L^2}{8\sigma_+r}\right)
+C_{\mathrm G}r^{-2}
\qquad (L\ge2\sqrt r).
\end{equation}

For fixed \(b\), each \(z\) belongs to \(B_\infty(u,b)\) for at most
\((2b+1)^d\) choices of \(u\). Furthermore, if \(|u|>\mathcal R_s\) and
\(z\in B_\infty(u,b)\), then
\[
|z|>\mathcal R_s-\sqrt d\,b.
\]
Consequently,
\begin{equation}\label{eq:outside-multiplicity}
U_{r,s}
\le
\sum_{b\ge0}\zeta^b(2b+1)^d
\sum_{|z|>\mathcal R_s-\sqrt d\,b}|K_{r,s}(z)|.
\end{equation}
Split the last sum at \(b=\mathcal R_s/(2\sqrt d)\). If
\(b\le\mathcal R_s/(2\sqrt d)\), then
\(\mathcal R_s-\sqrt d\,b\ge\mathcal R_s/2\). For all sufficiently large \(s\),
this lower bound is at least \(2\sqrt r\). Since \(r\le s\) whenever
\(p_r(s)>0\), while \(r\ge a_-s\) on \(\mathrm{SP}_s\),
\eqref{eq:K-tail-local-input} gives
\[
\sum_{|z|>\mathcal R_s-\sqrt d\,b}|K_{r,s}(z)|
\le
C\exp\!\left(-\frac{\mathcal R_s^2}{32\sigma_+s}\right)+Cs^{-2}
\le
C(\log s)^{-\Delta(A)},
\]
where
\[
\Delta(A):=\frac{A^2}{32\sigma_+}.
\]
The contribution of these \(b\)'s to \eqref{eq:outside-multiplicity} is therefore at most
\(C(\log s)^{-\Delta(A)}\).

For \(b>\mathcal R_s/(2\sqrt d)\), Lemma~\ref{lem:Phi-l1} gives
\[
\sum_{|z|>\mathcal R_s-\sqrt d\,b}|K_{r,s}(z)|
\le
\|K_{r,s}\|_1\le C_{\ell^1}.
\]
Since a polynomially weighted geometric tail decays exponentially,
\[
\sum_{b>\mathcal R_s/(2\sqrt d)}\zeta^b(2b+1)^d
\le
Ce^{-c_\zeta\mathcal R_s}
\le
C(\log s)^{-\Delta(A)}
\]
for all large \(s\). We conclude that
\begin{equation}\label{eq:outside-typical-bound}
U_{r,s}\le C(\log s)^{-\Delta(A)}
\qquad (r\in\mathrm{SP}_s).
\end{equation}

For arbitrary \(r\), another application of Lemma~\ref{lem:Phi-l1} and the same
multiplicity count give
\[
U_{r,s}
\le
\|K_{r,s}\|_1\sum_{b\ge0}\zeta^b(2b+1)^d
\le C.
\]
Therefore, by \eqref{eq:local-input-renewal},
\eqref{eq:local-input-renewal-upper}, and \eqref{eq:outside-typical-bound},
\[
\sum_{r\ge1}p_r(s)U_{r,s}
\le
C(\log s)^{-\Delta(A)}
\sum_{r\in\mathrm{SP}_s}p_r(s)
+C\sum_{r\notin\mathrm{SP}_s}p_r(s)
\le
C_{\mathrm{out}}(A,\zeta)(\log s)^{-\Delta(A)}
\]
for all sufficiently large \(s\). Substitution in \eqref{eq:outside-reduction} proves (c).
\end{proof}

%%%%%%%%%%%%%%%%%%%%%%%%%%%%%%%%%%%%%%%%%%%%%%%%%%%%%%%%%%%%%%%%%%%%%%%%%%%%%%%%%%%%%%%%%%%%%%%%%%%%%%%%%%%%%%%%%%%%%%%%%%%%%%%%%%%%%%%%%%%%%%%%%%%%%%%%%%%%%%%%%%%%%%%%%%%%%%%%%%%%%%%%%%%%%%%%%%%%%%%%%%%%%%%%%%%%%%%%%%%%%%%%%
%%%%%%%%%%%%%%%%%%%%%%%%%%%%%%%%%%%%%%%%%%%%%%%%%%%%%%%%%%%%%%%%%%%%%%%%%%%%%%%%%%%%%%%%%%%%%%%%%%%%%%%%%%%%%%%%%%%%%%%%%%%%%%%%%%%%%%%%%%%%%%%%%%%%%%%%%%%%%%%%%%%%%%%%%%%%%%%%%%%%%%%%%%%%%%%%%%%%%%%%%%%%%%%%%%%%%%%%%%%%%%%%%%%%%%%%%%%%%%%%
\bibliographystyle{amsplain}
\bibliography{Potts}

\end{document}